\documentclass[12pt,reqno,hidelinks]{amsart}
\usepackage{amssymb}
\usepackage{amsmath, amssymb, verbatim, url, mathrsfs}
\usepackage{mathrsfs}
\usepackage{dutchcal}
\usepackage{dsfont}
\usepackage{mathtools}
\usepackage{verbatim}
\usepackage{amsthm}
\usepackage{framed}
\usepackage{cite}
\usepackage{wasysym}
\usepackage{upgreek}
\usepackage{color}
\usepackage[dvipsnames]{xcolor}
\usepackage{tensor}
\usepackage{accents}
\usepackage{dsfont}
\usepackage[colorlinks,linkcolor=blue,citecolor=blue]{hyperref}
\usepackage{hyperref}
\usepackage{enumerate} 
\usepackage[normalem]{ulem}
\usepackage{longtable}
\usepackage{mathtools}

\numberwithin{equation}{section}

\newtheorem{proposition}{Proposition}[section]
\newtheorem{lemma}[proposition]{Lemma}
\newtheorem{corollary}[proposition]{Corollary}
\newtheorem{theorem}[proposition]{Theorem}

\newtheorem{definition}[proposition]{Definition}
\newtheorem{problem}[proposition]{Problem}

\theoremstyle{definition}
\newtheorem{remark}[proposition]{Remark}

\newcommand{\vertiiii}[1]{{\left\vert\kern-0.25ex\left\vert\kern-0.25ex\left\vert\kern-0.25ex\left\vert #1 \right\vert\kern-0.25ex\right\vert\kern-0.25ex\right\vert\kern-0.25ex\right\vert}}
\newcommand{\vertiii}[1]{{\left\vert\kern-0.25ex\left\vert\kern-0.25ex\left\vert #1 \right\vert\kern-0.25ex\right\vert\kern-0.25ex\right\vert}}

\newcommand{\rrho}{\rho}

\newcommand{\rw}{\mathring{v}}

\newcommand{\rp}{p}
\newcommand{\rPhi}{\Phi}

\newcommand{\bx}{\mathbf{x}}
\newcommand{\by}{\mathbf{y}}

\newcommand{\Rbb}{\mathbb{R}}
\newcommand{\Zbb}{\mathbb{Z}}

\newcommand{\del}[1]{{\partial_{#1}}}

\newcommand{\AND}{{\quad\text{and}\quad}}

\newcommand{\Li}{L^\infty}
\newcommand{\la}{\langle}
\newcommand{\ra}{\rangle_t}

\newcommand{\starcup}{$\sqcup$\kern-0.58em{$\star$}}

\newcommand{\p}[1]{
\begin{pmatrix}
  #1
\end{pmatrix}
}

\let\origmaketitle\maketitle
\def\maketitle{
	\begingroup
	\def\uppercasenonmath##1{} 
	\let\MakeUppercase\relax 
	\origmaketitle
	\endgroup
}

\allowdisplaybreaks

\DeclareMathOperator{\dist}{dist}

\DeclareMathOperator{\supp}{supp}

\begin{document}

\title[Improved local existence and improved continuation criterion]{Localized continuation criterion, improved local existence and uniqueness for the Euler--Poisson system in a bounded domain}

\address{Center for Mathematical Sciences, Huazhong University of Science and Technology,
1037 Luoyu Road, Wuhan, Hubei Province, China; Beijing International Center for Mathematical Research (BICMR), Peking University, No.5 Yiheyuan Road Haidian District, Beijing, China. }
\email{chao.liu.math@foxmail.com}
\author{Chao Liu}

\begin{abstract}
To understand the formations of singularities of the  Euler--Poisson system with vacuum, we revisit   Makino's star model in this article.
We first remedy,  to some extent, the inconveniences of Makino's star model and remove its imposed nonphysically exterior free-falling velocity field by only specifying the velocity field on a compact support of the density.
Makino has coined the presence of such
an exterior velocity field as a ``Cheshire
cat phenomenon'' and he  \cite{Makino1986} and Rendall \cite{Rendall1992} have both emphasized the difficulty of removing this phenomenon.
Moreover, we obtain an improved local existence and uniqueness theorem for
initial data for the density and velocity that both have compact support.
Finally, we are able to prove a localized strong continuation criterion in which the
breakdown of solutions is only controlled by quantities defined on
the compact support of the solution. In addition, the localized strong continuation criterion also generalizes the continuation criterion from the incompressible Euler equation in a bounded domain \cite{Ferrari1993,Shirota1993} to the compressible Euler--Poisson system.



\vspace{2mm}

\end{abstract}

\maketitle


\section{Introduction} \label{S:INTRO}
The existence and uniqueness of solutions for the Euler system and Euler--Poisson system with vacuum are notoriously difficult problems that have been widely studied in the recent three decades.  Studies on systems with physical vacuum boundaries have become increasingly active in recent years. However, motivated by the formation of physical vacuum boundaries,  the initial value problem of the Euler--Poisson system with  vacuum requires further attention (in some sense, it models a giant interstellar gas cloud).
Let us first introduce the initial value problem we are investigating and then further explain its significance and
why we are interested in solving it.
\begin{problem}[The initial value problem considered]\label{prob}\ 
\begin{enumerate}[(i)] 
\item \label{q:1} \underline{Variables:}  We use functions $\rrho:[0,T)\times \Rbb^3\rightarrow \Rbb_{\geq 0}$, $ p:[0,T)\times \Rbb^3\rightarrow \Rbb_{\geq 0}$, and $\rPhi:[0,T)\times\Rbb^3 \rightarrow \Rbb$, for some constant $T>0$, to describe the distribution of the mass density, the pressure of the fluids and the Newtonian potential of the fluid, respectively, and denote $\Omega(t):=\supp\rrho(t,\cdot) =\{\bx\in \Rbb^3\;|\;\rrho(t,\bx)>0\}\subset \Rbb^3$ is the changing volume occupied by the fluid at time $t$. Then the vacuum is identified by $\rrho(t,\bx)=0$ in $\Omega^\mathsf{c}(t)$ and the fluids by $\rrho(t,\bx)>0$ in $\Omega(t)$. We focus on the \textit{isentropic ideal} fluid throughout this article, that is, the \textit{equation of state}  is given by
\begin{align}\label{e:eos1}
 p= K  \rrho^{\gamma}  \quad \text{for} \quad \bx\in \Rbb^3
\end{align}
where $\gamma>1$ and $K\in \Rbb_{>0}$ are both given constants. The velocity of fluids is only defined on $\Omega(t)$ and denoted by  $\mathring{\mathbf{v}}:=(\mathring{v}^i):[0,T)\times \Omega(t)\rightarrow \Rbb^3$. According to these, we emphasize that: $(a)$ the velocity $\mathring{\mathbf{v}}$ is \textbf{not} defined in $\Omega^\mathsf{c}(t)$; $(b)$ the density $\rho =0$ on $\partial\Omega(t)$ where  $\partial \Omega(t)$ is the boundary of the volume $\Omega(t)$ (i.e.,  the moving interface between the fluid and the exterior vacuum).

\item \label{q:2} \underline{Euler--Poisson equations} determine the developments of the fluid in $\Omega(t)$ and the Newtonian potential on $\Rbb^3$, that is,
\begin{align}
\del{0} \rrho + \mathring{v}^i \del{i} \rrho+ \rrho \del{i} \mathring{v}^i = & 0  \quad  &&\text{in}\quad\Omega(t),  \label{e:NEul1} \\
\rrho \del{0} \mathring{v}^k + \rrho \mathring{v}^i \del{i} \mathring{v}^k + \delta^{ik} \del{i}  p = & -   \rrho \del{}^k  \Phi &&\text{in}\quad \Omega(t),  \label{e:NEul2} \\
\Delta  \Phi =  &  \rrho,
 &&\text{in}\quad \Rbb^3,  \label{e:NEul3}
\end{align}
for $t\in [0,T)$. The Newtonian potential $ \Phi$ is given by
\begin{equation}\label{e:Newpot}
 \Phi(t,\bx)=-\frac{1}{4\pi}\int_{\Omega(t)}				 \frac{\rrho(t,\by)}{|\bx-\by|}d^3 \by \quad \text{for} \quad  (t,\bx)\in [0,T) \times \Rbb^3,
\end{equation}

\item \label{q:3} \underline{Initial data} are prescribed by
\begin{align}\label{e:inidata}
\rrho(0,\bx)=\rrho_0(\bx) \quad \text{for} \quad \bx\in \Rbb^3 \AND \rw^k(0,\bx)= \rw^k_0(\bx) \quad \text{for} \quad \bx\in \Omega(0).
\end{align}
We assume $\Omega(0)\subset \Rbb^3$ is a precompact set throughout this article.
\end{enumerate}
\end{problem}

In this article, we consider \textit{regular solutions} to the initial value problem \ref{prob}:
\begin{definition}\label{t:clsl}
	A set of functions $(\rrho,\rw^i,\rPhi,\Omega(t))$ is called \textbf{a regular solution} to the above Problem \ref{prob} on $t\in[0,T)$ for $T>0$ if it solves the system \eqref{e:NEul1}--\eqref{e:NEul3} and satisfies the following conditions:
	\begin{enumerate}
		\item \label{D:1} there is a \textit{$C^1$-extension} $v^i$ of $\rw^i$ to the boundary, i.e., there exists\footnote{$v^i\in C^1([0,T)\times \overline{\Omega(t)},\Rbb^3)$ means $\rw^i$ and $\del{\mu}\rw^i$ have continuous extensions to $[0,T)\times \overline{\Omega(t)}$. }
		$v^i\in C^1([0,T)\times \overline{\Omega(t)},\Rbb^3)$, such that $v^i(t,\bx)=\rw^i(t,\bx)$ for $(t,\bx)\in[0,T)\times \Omega(t)$;
		\item \label{D:2}  $\rrho\in C^1([0,T)\times \Rbb^3,\Rbb)$ and $\rPhi$ is given by \eqref{e:Newpot};
		\item \label{D:3} $c_s^2=K\gamma \rho^{\gamma-1} \in C^1([0,T)\times \Rbb^3,\Rbb)$ where $c_s^2:=d\rp/d\rrho =K\gamma \rrho^{\gamma-1}$ is the sound speed.
	\end{enumerate}
\end{definition}
\begin{remark}
	According to   Problem \ref{prob}, although it is an initial value problem, there is an interface $\partial\Omega(t)$ between the fluid and the vacuum. However, on the other hand, we should emphasize that Problem \ref{prob} is not a free initial boundary problem  despite the existence of the boundary $\partial \Omega(t)$ since there is no prescribed conditions imposed on the boundary $\partial\Omega(t)$. 
\end{remark}
\begin{remark}\label{r:ffb}
	Although Problem \ref{prob} is  a pure initial value  problem, due to the existence of this boundary, it is beneficial to see the property of the boundary $\partial\Omega(t)$.  This property is a consequence of   the regular solutions (see Definition \ref{t:clsl}). 
Note \eqref{D:2} has already implied $c_s^2\in C^0$ and away from the boundary $\partial\Omega(t)$, we have $\nabla c_s^2 \in C^0$, thus the key point of Definition \ref{t:clsl}.\eqref{D:3}  is that $\nabla c_s^2$ is continuous across the boundary. Further, since $\nabla c_s^2 =0$ in $\Rbb^3\setminus \overline{\Omega(t)}$,  the condition  \eqref{D:3} implies the boundary $\partial\Omega(t)$ satisfies
\begin{equation}\label{e:ffbdry0}
 \del{i}(c_s^2(\rho))=0 \quad \text{on} \quad \del{}\Omega(t).
\end{equation}
This condition \eqref{e:ffbdry0} has been explicitly pointed out in \cite{Liu2000} for one-dimensional case (see also \cite[p.554]{Zeng2020} for a three-dimensional generalization of it which is slightly weaker than \eqref{e:ffbdry0}). Now let us briefly explain the condition \eqref{e:ffbdry0} is equivalent to, for the equation of state \eqref{e:eos1},  the \textit{free-fall boundary} given by \cite{Brauer1998} if $\Omega(t)\in C^1$.  Note \eqref{e:NEul2}  can be rewritten as
\begin{equation}\label{e:ffba}
 \del{0} \mathring{v}^k +  \mathring{v}^i \del{i} \mathring{v}^k + \frac{1}{\gamma-1} \delta^{ik} \del{i}  c_s^2 =  -    \del{}^k  \Phi  \quad \text{on} \quad \Omega(t).
\end{equation}
If we have a regular solutions to Problem \ref{prob}, let us extend the $C^1$ vector field $v^k$ (we remind $v^k$ is a $C^1$-extension of $\ring{v}^k$ according to Definition \ref{t:clsl}.$(1)$) to a slightly larger neighborhood of $\overline{\Omega(t)}$, for example, we can use the extension as that in \cite[Eq. $(2.9)$]{Luo2014}. 
Using the extended velocity field and taking the limit on the both sides of the equation \eqref{e:ffba} by approaching to  the boundary $\partial \Omega(t)$,  with the help of $\del{i}c_s^2 = 0$ at the vacuum boundary,  one can verify that
\begin{equation}\label{e:ffeq0}
	\del{0} v^k +  v^i \del{i} v^k   =  -    \del{}^k  \Phi  \quad \text{on} \quad \partial \Omega(t) ,
\end{equation}
which is the free-fall equation characterizing the free-falling of  test particles acted only by the Newtonian gravity. This means that the evolution of the vacuum boundary is independent of the gas
inside the region, and thus it is free-falling due to gravity.  Please note  for any $C^1$-extension, we have the free-fall equation \eqref{e:ffeq0} at the boundary.
\end{remark}

\begin{remark}
Comparing with the \textit{physical vacuum boundary condition} (a soft boundary, see, for example, \cite{Coutand2010,Coutand2012,Jang2009,Liu2000,Luo2014}), $
-\infty<\nabla_{\mathbf{n}}(c^2_s)<0$ on $\partial\Omega(t)$ where  $\mathbf{n}:=(n^i)$ is the outward unit
normal to $\partial \Omega(t)$ (i.e., $\nabla c_s^2$ is not continuous across the boundary), the free-fall boundaries do \textit{not} have \textit{physical vacuum singularities} at the boundary.
\end{remark}

\begin{remark}\label{t:inirmk}
	One of the final motivations, at least for us, of studying the initial value problem of the Euler--Poisson system with compact support  is to find conditions that lead to the physical vacuum singularities, that is, to study formations of  the physical vacuum (modeling the formations of stars from molecular clouds). The condition \eqref{D:3} in Definition \ref{t:clsl} helps detect such singularities. That is why we require the regularity  \eqref{D:3}. However, we have to admit and point out, in this article, in order to simplify the proof, except the uniqueness theorem proven for the regular solutions, we only focus on a type of stronger solutions, Makino type solutions with $C^1$ domain (see Definition \ref{t:clsl2}), for the existence theorem and continuation criteria. The general existence theorem and continuation criteria for regular solutions may leave for future investigations.
\end{remark}

Although the initial value problem \ref{prob} is much simpler than the initial boundary problem for the Euler--Poisson system with the physical vacuum boundaries, there are still some \textit{motivations and purposes} to revisit it:

$(1)$ For some situations of the Euler equation, it is well known that the $\nabla c_s^2$ can not be continuous across the boundary after a finite time (for example, see \cite{Liu1997,Makino1986a,Pan2005}). However, for the Euler--Poisson system,  this result has only been proven for spherically symmetric cases (see \cite{Makino1990} for Makino solution and a remark of generalizations on page $497$ in \cite{Liu2000}). We would like to see if this phenomenon is true for non-spherical symmetric solutions of  Problem \ref{prob}, which is intimately related to Makino's conjecture in \cite[Page $617$-$618$]{Makino1992}. 
We do not intend to achieve this goal in this article, but we have made some preparations for it.
We derive an improved continuation criterion that only relies on quantities on $\Omega(t)$ instead of $\Rbb^3$. It is an improvement of the continuation criterion in \cite{Brauer1998}, and to some extent, it is the desired estimate to localize the singularities and answer Brauer's conjecture in \cite[\S VII.$(4)$]{Brauer1998}. If we know the solution has to blow up at a finite time, this improved continuation criterion will help us rule out impossible types of singularities.

$(2)$ To the best of our knowledge, almost all works on the initial value problem of the Euler--Poisson system with vacuum have been based on Makino's ideas of symmetrizing the hyperbolic system. A notable exception is \cite{Liu1997}, which considered a system
with damping and proved the local existence in one dimension. 
The symmetrized form must be solved in the entire space $\Rbb^3$, as it is an initial value problem without boundary constraints, and the Poisson equation must be solved on $\Rbb^3$. Energy estimates, obtained through integration by parts, play a key role in solving this system, and boundary terms will be eliminated if certain regularity conditions are met.  However, Makino's symmetrized form is not equivalent to the Euler--Poisson system when the density vanishes. There are additional velocity field equations in $\Omega^\mathsf{c}(t)$ (or it can be viewed as additional regularities of the velocity in $\Omega^\mathsf{c}(t)$, see \cite[Page $617$]{Makino1992} and \cite[Definition $(ii)$ in Page $166$]{Makino1990}, or see later \eqref{e:otvel}) have been added to the original Euler--Poisson system, i.e., as Makino wrote: ``in spite of the absence of the
matter the velocity is forced to vary in obedience to the gravitational field''.  It is also known as the Cheshire cat phenomenon and needs   improvements (see \cite[\S$6$]{Makino1986}). In this article, we try to remove the  regularities and existence of the velocity in $\Rbb^3 \setminus \overline{\Omega(t)}$ but only keep the regularities at the boundary $\partial\Omega(t)$ (see Definition and Remark \ref{r:ffb}). In this setting, namely only giving the velocity on   $\Omega(t)$ and requiring the regularities at the boundary $\partial \Omega(t)$, and we try to prove a local existence and uniqueness theorem for the above Problem \ref{prob}. However, in this article, we can not provide a complete proof for the local existence of \textit{regular solutions} to this initial value problem (recall Remark \ref{t:inirmk}), but only for a subclass of the regular solution, i.e., Makino type solution with a $C^1$ domain (see Definition \ref{t:clsl2}).  Nevertheless, the uniqueness theorem in this article works for  regular solutions. Due to the defects of this article, it leaves the general local existence theorem for the regular solution for future explorations.

In summary, based on the above motivations, this article  focuses on the following \textit{aims}: $(1)$ Improve the \textit{local existence} and gain the \textit{uniqueness} theorem of Problem \ref{prob} by only imposing the initial velocity on the compact support $\Omega(0)$ of the density, which improves Makino's existence theorem in \cite{Makino1986} and Brauer's uniqueness theorem in \cite{Brauer1998}; $(2)$ Obtain better \textit{continuation criteria} by controlling  quantities  that again are only defined on $\Omega(t)$. We emphasize although,  in the strong continuation criterion (Theorem \ref{t:conpri2}), we can reduce the possible types of singularities to the blowups of \textit{some components} of the full gradient of the velocity in the interior of $\Omega(t)$, but  as a \textit{caveat}, in a small near-boundary region $\Omega_\epsilon(t)$ we can not do so, and the possible blowups are from the full gradient of the velocity.
These continuation criteria improve Brauer's in \cite{Brauer1998} and achieve, to some extent, Brauer's expectations on localizing the possible singularities.

To our knowledge, the study of the initial value problem for the Euler--Poisson  with vacuum  was started by Makino \cite{Makino1986,Makino1987}, who proved a local existence and uniqueness theorem using the theory of
symmetric hyperbolic systems and a new regularisation procedure for the density.  Later, a series of works \cite{Makino1990,Makino1986a,Perthame1990,Makino1992,Liu1997} continued to study the evolution of Makino's solutions to the Euler equation with or without gravity. Note that the blowup results of the Euler--Poisson system with vacuum \cite{Makino1990,Perthame1990,Makino1992} of Makino solutions are under spherical symmetry, but there are no symmetric assumptions for Euler equations in \cite{Makino1986a}.  The advantage of this method is that it is easy to prove blowups, but it usually does not give more information about the detailed behavior of the finite-time blowup solutions.  We point out that \cite{Makino1992} specified \textit{Makino's conjecture} that any tame
solution, including nonsymmetric tame solution, will become not tame after a finite time.  Brauer and Karp \cite{Brauer2018,Brauer2020} first proved the local wellposedness (in the Hadamard sense) for the
Euler–Poisson(–Makino) system including densities that fall off at infinity. By generalizing the method in \cite{Beale1984,Chemin1990}, Brauer \cite{Brauer1998}  gives a strong continuation criterion for Makino solutions, especially a more detailed classification of the type of blowups.

\subsection{Remarks on the basic ideas}
On the one hand, the original Makino's problem  in \cite{Makino1986} includes the Euler--Poisson equation and data on $\Rbb^3$ and   Makino solutions which regularize the density and require an extra equation of the free falling (see Definition \ref{e:tmsl}).  On the other hand, compared with the original Makino's problem, Problem \ref{prob} only involves the equations and data on $\Omega(t)$ and regular solutions on $\Omega(t)$. The regular solutions in Definition \ref{t:clsl}, in fact, also put constraints on the near-boundary behaviors of the solutions. In other words, they lead to the regularities at the boundary $\partial\Omega(t)$ (see Remark \ref{r:ffb}).  If considering the uniqueness theorem (see Theorem \ref{t:unithm1}), then based on an observation that the regular solution is unique once the data in $\Omega(t)$ is chosen, we notice that the exterior component of  Makino solutions in $\Rbb^3\setminus \overline{\Omega}$ can not affect the interior part of the solutions in $\Omega(t)$. This fact inspires us that the Makino solutions can be used to generate a localized solution on $\Omega(t)$. Since once the data on $\Omega(t)$ is chosen, all the Makino solutions lead to one determined interior solution, which leads us to the Makino type solution with a $C^1$ domain (see Definition \ref{t:clsl2}). It is clear whatever the exterior parts of Makino solutions on $\Rbb^3\setminus \overline{\Omega}$ are, once the data on $\Omega(t)$ is given, the interior solution on $\Omega(t)$  (i.e., the Makino type solution with a $C^1$ domain) is independent of the exterior Makino solutions, but it uniquely solves Problem \ref{prob} without any other extra requirements. Thus, the Makino type solution with a $C^1$ domain yields a regular solution to Problem \ref{prob} without any further exterior regularities on $\Rbb^3\setminus \overline{\Omega}$ but only the regularities at the boundary $\partial \Omega(t)$.

In summary, Problem \ref{prob} is given independent of the exterior information of the velocity $v^i$ on $\Rbb^3\setminus\overline{\Omega(t)}$. The \textit{aim} is to find a \textit{regular solution} defined by Definition \ref{t:clsl} which is independent of the exterior information of the velocity $v^i$ on $\Rbb^3\setminus\overline{\Omega(t)}$ as well. The idea for the local existence theorem for Problem \ref{prob} is based on the uniqueness theorem which implies the exterior information of the velocity $v^i$ on $\Rbb^3\setminus\overline{\Omega(t)}$ can not affect the regular solution on $\Omega(t)$. Thus, if possible,  we can   extend the data \textit{arbitrarily} to Makino data, then use Makino's method to obtain a local Makino solution. Since the interior component of the Makino solution is independent of the exterior part, and the interior part of Makino solution does solve Problem \ref{prob}, we arrive at a Makino type solution with a $C^1$ domain. Just because we attempt to use Makino's method, it reduces the regular solutions to only a subclass, i.e.,  the Makino type solutions with a $C^1$ domain. Therefore, the Makino’s method only provides a way to construct a proper subclass of the regular solution to Problem \ref{prob} The exterior component of the Makino solution does not affect the regular solution.  However, we still need boundary regularity to limit the class of the solutions. Due to this, Problem \ref{prob} and its regular solution have nothing to do with the exterior regularities and the existence of the velocity fields on $\Rbb^3\setminus\overline{\Omega(t)}$. In this point of view, we overcome the Cheshire's cat to some extent.

Please note there is a difference on the regularity between Makino and regular solutions.  The regular solution requires $\rho^{ \gamma-1 }\in C^1$ (recall Definition \ref{t:clsl}), while the Makino requires a stronger regularity $\rho^{\frac{\gamma-1}{2}}\in C^1$ than $\rho^{ \gamma-1 }\in C^1$ (see Definition \ref{e:tmsl} in \S\ref{s:lextm}). 
Although our existence theorem can only include a \textit{subclass} of the \textit{regular solutions} constructed by Makino  solutions,  our uniqueness theorem can imply that all possible \textit{regular solution} is unique up to the same data. The uniqueness is more general than that in \cite{Brauer1998} since the proof is independent of the symmetric hyperbolic system. A relative entropy-entropy flux method given in \cite{Luo2014} is applied. 
A more general local existence for the complete set of regular solutions (with regularity $\rho^{ \gamma-1 }\in C^1$) is desired for the formation problem of the physical vacuum singularities. Methods beyond Makino's symmetrization must be invented, and this problem is open for future studies.



\subsection{Notation and convention}

\subsubsection{Vectors and components} \label{s:vecnota}
We will use unless otherwise stated, boldface, e.g., $\bx$, $\by$, $\mathbf{v}$ for Latins and the Greeks, e.g., $\xi$, to denote vectors and normal font with indices, e.g., $x^i$, $y^i$, $\xi^i$ and $v^i$, to denote the components of vectors (we also use these components to express the vector if it clear from the context).

\subsubsection{Indices and summation convention}\label{iandc}
Throughout this article, unless stated otherwise, we adopt the notation system in general relativity (i.e. \textit{Einstein notation}, see \cite{Wald2010} for more). In specific, we use upper indices to represent components of vectors,
lower indices to represent components of covectors, and use lower case Latin letters, e.g. $i, j,k$, for spatial indices that run from $1$ to $3$, and lower case Greek letters, e.g. $\alpha, \beta, \gamma$, for spacetime indices
that run from $0$ to $3$ (we also denote time coordinate $t$ by $x^0:=t$). We use the \textit{Einstein summation convention} as well. That is, when an index variable (i.e. dummy index) appears twice, once in an upper superscript and once in a lower subscript, in a single term and is not otherwise defined, it implies the summation of that term over all the values of the index. For example, for a summation $z=\sum^3_{i=1} x^i y_i$, we simply use Einstein notation to denote $z= x^i y_i$. In addition, we raise and lower indices by the Euclidean metric $\delta^{ij}$ and $\delta_{ij}$, that is, for example,
\begin{equation*}
	V^{ij}:=V_{lk}\delta^{li}\delta^{kj}, \quad V_{ij}:=V^{lk}\delta_{li}\delta_{kj} \AND V^{ij}:=V^i_k\delta^{kj}.
\end{equation*}

\subsubsection{Lagrangian descriptions}\label{s:lag}
Suppose a field $f:[0,T)\times \Omega(t) \rightarrow V$ (or a property $f$ in Eulerian description) and the flow $\chi:[0,T)\times \Omega(0)\rightarrow \Omega(t)\subset \Rbb^3$ generated by a vector field  $\mathbf{v}:=(v^i)$, such that  $\chi(t,\xi)=\bx \in \Omega(t)$ for every  $(t,\xi)\in[0,T)\times \Omega(0)$ where $T>0$ is a constant, $\Omega(t) \subset \Rbb^3$ is a domain depending on $t$ and $V \subset \Rbb^n$ for some $n \in \Zbb_{\geq 1}$, we denote
\begin{equation}\label{e:undf}
	\underline{f}(t,\xi):=f(t,\chi(t,\xi))
\end{equation}
describing the property $\underline{f}$ of the parcel labeled by the initial position $\xi$ at time $t$ (the property $\underline{f}$ along the flow, i.e., in Lagrangian description). Using this notation, we have
\begin{equation*}
	\del{t}\underline{f}(t,\xi)=\underline{D_t f  }(t,\xi)
\end{equation*}
where $D_t$ is the \textit{material derivative} (i.e. $D_t:=\del{t}+v^i\del{i}$, see, for instance, \cite{Chorin1993}).
According to the definition of $L^\infty$, if $\chi(t,\Omega(0))=\Omega(t)$, we conclude that
\begin{equation*}
	\|f(t)\|_{\Li(\Omega(t))}=\|\underline{f}(t)\|_{\Li(\Omega(0))}.
\end{equation*}

\subsubsection{Sets}\label{s:sets}
We denote $U$ is \textit{strictly contained} in $V$ by $U\subset\subset V$, which represents the closure of $U$ is a compact subset of $V$. A set $U$ is called \textit{precompact} if its closure $\overline{U}$ is compact.
We denote the following subset $\Omega_\epsilon $ of $\Omega $ for some constant  $\epsilon>0$ by
\begin{align}\label{e:omgep}
\Omega_\epsilon:= \{\bx \in \Omega\;|\; \dist(\bx,\del{}\Omega)<\epsilon\},
\end{align}
where  $\dist(\bx,\del{}\Omega):=\inf_{\by\in\del{}\Omega}|\bx-\by|$ and denote
\begin{equation*}
	\mathring{\Omega}_\epsilon :=\Omega \setminus \Omega_\epsilon .
\end{equation*}
We also use $B(\bx_0, r)$ to denote a ball centered at $\bx_0$ with the radius $r$. Therefore, by \eqref{e:omgep}, we denote
\begin{equation*}
	B_\epsilon(\bx_0,r):=\{\bx \in B(\bx_0, r)\;|\; \dist(\bx,\del{}B(\bx_0, r))<\epsilon\},
\end{equation*}
The complement set of a set $\Omega$ is denoted by $\Omega^\mathsf{c}:= \Rbb^3 \setminus \Omega $.

\subsection{Preliminary  concepts}\label{s:defs}
To state the main theorem concisely, we first introduce a few definitions.
\subsubsection{Decomposition of velocities}
For simplicity of notations, we denote
\begin{equation*}\label{e:defW}
	V^i_j(t,\bx):=v^i_{,j}(t,\bx):=\begin{cases}
		\del{j}\rw^i(t,\bx) \quad & \text{if}\quad (t,\bx)\in[0,T)\times \Omega(t)\\
		\lim_{(t^\prime,\bx^\prime)\rightarrow(t,\bx)}\del{j}\rw^i(t^\prime,\bx^\prime)\quad & \text{if}\quad (t,\bx)\in[0,T)\times \del{}\Omega(t)
	\end{cases}.
\end{equation*}
Let us lower the index of $V^i_j$ by $\delta_{ki}$, i.e., $V_{jk}(t,\bx):=  \delta_{ki}V^i_j(t,\bx)$ and decompose
\begin{equation*}
	V_{jk}(t,\bx)
	=   \frac{1}{2}\bigl(V_{jk}(t,\bx)+V_{kj}(t,\bx)\bigr)+\frac{1}{2}\bigl(V_{jk}(t,\bx)-V_{kj}(t,\bx)\bigr)
	=   \Theta_{jk}(t,\bx)-\Omega_{jk}(t,\bx)
\end{equation*}
where $\Theta_{jk}$ is the symmetric deformation component ($\Theta_{jk}=\Theta_{kj}$) of $V_{jk}$ and $\Omega_{jk}$ the antisymmetric rotation component ($\Omega_{jk}=-\Omega_{kj}$) defined by
\begin{align}
	\Theta_{jk}(t,\bx):=& \frac{1}{2}\bigl(V_{jk}(t,\bx)+V_{kj}(t,\bx)\bigr) \label{e:WThOm0a}\\
	\Omega_{jk}(t,\bx):=& \frac{1}{2}\bigl(V_{kj}(t,\bx)-V_{jk}(t,\bx)\bigr). \label{e:WThOm0b}
\end{align}
Then we have the identities,
\begin{equation*}
	V_{kj}=\Theta_{jk}+\Omega_{jk} \AND V_{jk}=\Theta_{jk}-\Omega_{jk}.
\end{equation*}
We also denote the divergence of the velocity
\begin{equation}\label{e:theta}
	\Theta(t,\bx):=\delta^{jk}\Theta_{jk}(t,\bx)=\delta^{jk}V_{jk}(t,\bx)= v^j_{,j}(t,\bx).
\end{equation}

To state the continuation criteria for localized singularities, we introduce a solution with a stronger regularity than the regular solution.
\begin{definition}\label{t:clsl2}
	A set of functions $(\rrho,\rw^i,\rPhi,\Omega(t))$ is called \textbf{a Makino type solution with a $C^k$ domain} ($k\in\Zbb_{\geq 1}$) to the above Problem \ref{prob} on $t\in[0,T)$ for $T>0$ if it is a regular solution (see Definition \ref{t:clsl}) to  Problem \ref{prob} and satisfies the following additional conditions:
	\begin{equation*}
		\rho^{\frac{\gamma-1}{2}} \in C^1([0,T)\times \Rbb^3,\Rbb) \AND \partial\Omega(t)\in C^k \;(k\in\Zbb_{\geq 1}).
	\end{equation*}
\end{definition}

\subsection{Main Theorem}
After the above concepts, we are in a position to present the main theorem. This theorem states the improved local existence, uniqueness, and continuation criterion.
\begin{theorem}[Main Theorem]
	\label{t:mainthm}
	Suppose the initial data $(\rrho_0, \rw^i_0)$ of  Problem \ref{prob} is given by  \eqref{e:inidata}, $\Omega(0)$ is a bounded, precompact and simply connected $C^1$-domain, $\rrho_0\in C^1( \Rbb^3,\Rbb_{\geq 0}) $, $\rw^i_0 \in C^1(\Omega(0),\Rbb^3) $. If
	\begin{equation*}
		(a) \quad 1<\gamma\leq \frac{5}{3}, \quad (\rrho_0)^{\frac{\gamma-1}{2 }} \in H^3(\Rbb^3) \AND \rw^i_0\in H^3(\Omega(0),\Rbb^3),
	\end{equation*}
	or if
	\begin{align*}
		(b)\quad 1<\gamma<2, \quad  \rrho_0 \in H^3(\Rbb^3), \quad (\rrho_0)^{\frac{\gamma-1}{2 }}
		\in H^4(\Rbb^3) \AND
		\rw^i_0\in H^4(\Omega(0),\Rbb^3),
	\end{align*}
	then:
	
$(1)$ \emph{(The local existence and uniqueness)}
there is a constant $T>0$, such that  Problem \ref{prob}  has a unique regular solution $(\rrho,\rw^i,\rPhi, \Omega(t)) $ defined by Definition \ref{t:clsl} and this unique regular solution is a Makino type solution with a $C^1$ domain. Moreover,  $\rrho^{\frac{\gamma-1}{2 }}
	\in C^{s-2}([0,T)\times\Rbb^3)$ and $ \Phi \in  C^2([0,T)\times\mathbb{R}^3)$, and the solution $\rw^i$ satisfies
	\begin{equation*}
		\rw^i\in C^0([0,T),H^{s}(\Omega(t)))\cap C^1([0,T),H^{s-1}(\Omega(t)))
	\end{equation*}
	where $s=3,4$ according $(a)$ and $(b)$ respectively;
	
$(2)$ \emph{(The strong continuation criterion)}
	Suppose $1<\gamma \leq 5/3$ and let  $(\rrho,\rw^i,\rPhi,\Omega(t))$ be a Makino type solution with a $C^1$ domain to  Problem \ref{prob} with the property $(w,\mathring{v}^i)\in C^0([0,T^\star),H^{s}(\Omega(t)))\cap C^1([0,T^\star),H^{s-1}(\Omega(t)))$, and $[0,T^\star)$ is the maximal  interval of the existence of  this solution. Then either
\begin{enumerate}
	\item $T^\star=\infty$;
	\item or $T^\star<\infty$ and  there is a small constant $\epsilon>0$ such that
	\begin{align*}
		&\int^{T^\star}_0 \Bigl(\|V_{jk}(s)\|_{L^\infty( \Omega_\epsilon(s))}+\|\Theta(s)\|_{L^\infty( \mathring{\Omega}_\epsilon(s) )}\notag  \\
		&\hspace{1.5cm} +\|\Omega_{jk}(s)\|_{ L^\infty( \mathring{\Omega}_\epsilon(s) )}+\|\nabla  w(s)\|_{ L^\infty(\Omega(s))} \Bigr) ds=\infty.
	\end{align*}
	where $\mathring{\Omega}_\epsilon(s):=\Omega(s)\setminus \Omega_\epsilon(s)$.
\end{enumerate}
\end{theorem}

\subsection{Overview and outline}
\S\ref{s:U&LEclsol} aims to prove the main Theorem \ref{t:mainthm}.$(1)$, that is, the local existence and  uniqueness theorem of the regular solution to  Problem \ref{prob}. The idea of the proof is to connect the regular solution of  Problem \ref{prob} with the Makino solution. We introduce the local existence of the Makino solution in  \S\ref{s:lextm}. Then prove the local existence of the solution to  Problem \ref{prob} by Calder\'{o}n's extensions of some data of regular solutions and by transforming the Makino solution to the solution with the compact support. The uniqueness of the regular solution is given in \S \ref{s:unqthm} based on Lemma \ref{t:unqr3} (see Appendix \ref{a:Lpf}, a lemma proven by using the relative entropy-entropy flux method given in \cite[\S $2$, Step $2$]{Luo2014}) and the Stein extension theorem. One ingredient of proving the uniqueness theorem is the preservation of the regularity of the initial $C^1$ boundary, which will be presented in \S\ref{s:C1bdy}.

In \S\ref{s:contpri}, we introduce the improved continuation criteria of  Problem \ref{prob}. These continuation criteria allow us to continue the local solutions to a larger time interval provided the suitably bounded spatial derivatives. The continuation criteria  mainly help us determine the information of singularities if the solution breaks down at a finite time, and the singularities are all localized in the compact support of the density. First, we give a weak version (see \S\ref{s:ecp}) and it will help us extend the regular solutions to several types of singularities in the compact support of the density. The idea of the proof is to use Makino's formulation of the Euler--Poisson system and energy estimates. However, the standard Gagliardo–Nirenberg–Moser estimates have been replaced by the revised Gagliardo–Nirenberg–Moser estimates on the Lipschitz bounded
domain. Then the strong continuation criterion (see \S\ref{s:scp}) helps us reduce the types of interior singularities.  The key tool of the proof for the strong continuation criterion is the Ferrari--Shirota--Yanagisawa inequality (a bounded domain  version of the Beale--Kato--Majda estimate). This inequality allows us to control the velocity $\|v^i\|_{W^{1,\infty}}$ ``almost only'' by expansion and rotation. In order to use Ferrari--Shirota--Yanagisawa (FSY) inequality, there are two difficulties: the first one comes from the boundary regularity and the second one is since the velocity at the boundary may not be orthogonal to the outer normal of the boundary, that is, the \textit{second difficulty} is the orthogonal condition of the boundary limiting velocity $v^i\delta_{ij} n^j=0$ is not always true in this theorem.
Our idea to overcome these two difficulties and use the FSY inequality is: we first select a smooth surface (the $C^\infty$-approximation lemma of boundary \ref{t:Cinfapx} ensures this) near the boundary $\del{}\Omega(t)$, then decompose the velocity field into two parts around this $C^\infty$ surface, one of these parts is ``tangent'' to the $C^\infty$ surface on it and the other orthogonal to it, then analyze these two parts respectively.

In the end, let us emphasize some key aspects of this article in case causing some confusions. 
\begin{enumerate}
	\item We define three types of solutions throughout this paper: regular solutions (Definition \ref{t:clsl}), the Makino type solutions with a $C^k$ domain (Definition \ref{t:clsl2}) and Makino solution (Definition \ref{e:tmsl}). We emphasize the Makino type solutions with a $C^k$ domain can be viewed as a restriction of the Makino solution to the $C^k$ domain $\Omega(t)$. The regular solution has weaker regularity $\rho^{\gamma-1}\in C^1$ than the Makino type solutions with a $C^k$ domain. We introduce these three types of solutions because, in this article, we prove the existence theorem and continuation criteria for the Makino type solutions with   a $C^k$ domain but the uniqueness theorem for a larger class, the regular solutions. In order to easily prove the local existence theorem, we have to use Makino solutions to extend the Makino type solutions with a $C^k$ domain to the $\Rbb^3$. That is why we introduce Makino solutions as well.  
\item We emphasize although we extend Makino type solutions with a $C^k$ domain to Makino solutions with freely indicated Makino's data. However, the Makino type solutions with a $C^k$ domain do not rely on the Makino solutions on $\Rbb^3\setminus \overline{\Omega(t)}$ if they have the same components of data in $\Omega(t)$. In other words, the Makino type solutions with a $C^k$ domain can be extended to all kinds of Makino solutions if they have the same components of data on $\Omega(t)$, but the \textit{key point} is all the Makino extensions on $\Rbb^3\setminus \overline{\Omega(t)}$ do not affect the interior component of solutions on $\overline{\Omega(t)}$ which is ensured by the uniqueness theorem \ref{t:unithm1}.  Thus, although a Makino type solution with a $C^1$ domain can be extended to many different Makino solutions, the interior component of Makino solution (i.e., the Makino type solution with a $C^1$ domain) is uniquely determined independent of the outside component in $\Rbb^3 \setminus \overline{\Omega(t)}$ of the Makino solutions. This procedure can erase the regularities of the Makino solutions on $\Rbb^3 \setminus \overline{\Omega(t)}$, but of course, can not erase the regularities of Makino solutions on the boundary $  \partial\Omega(t)$. 
\end{enumerate}

\section{Local existence and uniqueness of solutions to  Problem \ref{prob}}\label{s:U&LEclsol}
This section contributes to the local existence and uniqueness theorems of the regular solution to  Problem \ref{prob}. However, we will not prove the general existence theorem of regular solutions.  We will construct the regular solution of  Problem \ref{prob} by the famous Makino solution (with a slightly stronger regularity $\rho^{\frac{\gamma-1}{2}}\in C^1$). We first recall the local existence of the Makino solution (see \cite{Makino1986}) in  \S\ref{s:lextm}, then, by using Makino solutions, construct the local solutions to  Problem \ref{prob} by Calder\'{o}n's extensions of initial data  and transforming the solution of  Problem \ref{prob} to the Makino solution. The uniqueness of the regular solution is given in \S \ref{s:unqthm} based on Lemma \ref{t:unqr3} (see Appendix \ref{a:Lpf}, a Lemma proven by using the relative entropy-entropy flux method given in \cite[\S $2$, Step $2$]{Luo2014}) and Stein extension theorem. One ingredient of proving the uniqueness theorem is the preservation of the regularity of the initial $C^1$ boundary, which will be presented in \S\ref{s:C1bdy}.

\subsection{Local existence of Makino solutions}  \label{s:lextm}
Before discussing the regular solutions of  Problem \ref{prob},  we first recall the relevant Makino solutions and their local existences given by Makino  \cite{Makino1992,Makino1990,Makino1987}.
The Cauchy problem of Euler--Poisson equations with variables defined in $[0,T)\times\Rbb^3$ for some constant $T>0$ is given by
\begin{align}
\del{0} \rrho + v^i \del{i} \rrho+ \rrho \del{i} v^i = & 0  \quad  &&\text{in}\quad  [0,T)\times\Rbb^3,  \label{e:MNEul1} \\
\rrho \del{0} v^k + \rrho v^i \del{i} v^k + \delta^{ik} \del{i} p = & -  \rrho \del{}^k  \Phi &&\text{in}\quad  [0,T)\times\Rbb^3,  \label{e:MNEul2} \\
\Delta  \Phi =   &  \rrho
 &&\text{in}\quad  [0,T)\times\Rbb^3.  \label{e:MNEul3}
\end{align}
where $\rrho$ and $p$ satisfy \eqref{e:eos1} and the velocity $\mathbf{v}:=(v^k):[0,T)\times \Rbb^3\rightarrow \Rbb^3$.
The \textit{initial data} are prescribed by
\begin{align}\label{e:inidata2}
(\rrho,v^k)=(\rrho_0, v^k_0), \quad \text{on} \quad  \{0\}\times \Rbb^3,
\end{align}
and $\Omega(0)$ defined in \S\ref{S:INTRO}.$(i)$ is precompact. Let us define the Makino solution (see \cite{Makino1990,Makino1992}).
\begin{definition}\label{e:tmsl}
	A set of functions $(\rrho,v^i,\rPhi)$ is called \textbf{a Makino solution} of the Cauchy problem \eqref{e:MNEul1}--\eqref{e:inidata2} on $t\in[0,T)$ for $T>0$,
    if it solves the system \eqref{e:MNEul1}--\eqref{e:inidata2} and satisfies the following conditions:
	\begin{enumerate}
		\item $(\rrho,v^i)\in C^1([0,T)\times \Rbb^3, \Rbb^4)$ and $\rPhi$ is given by \eqref{e:Newpot};
		\item $\rrho^{\frac{\gamma-1}{2}}\in C^1([0,T)\times \Rbb^3)$ and $v^i$ satisfies an extra equation of the free falling,
		\begin{equation}\label{e:otvel}
		D_t v^i(t,\bx) =\del{t} v^i(t,\bx) +v^k(t,\bx) \del{k}v^i(t,\bx) =- \del{}^i \Phi(t,\bx) ,
		\end{equation}
		for $(t,\bx) \in [0,T)\times \Omega^\mathsf{c}(t) $.
	\end{enumerate}
\end{definition}
We refer to, throughout this article, the problem of finding the Makino solutions of Euler--Poisson equation \eqref{e:MNEul1}--\eqref{e:inidata2} as solving the \textit{Makino problem}.

\begin{remark}\label{r:tstptl}
The equation of the free-falling \eqref{e:otvel} can be interpreted as an equation of motion of a test particle. A test particle  is a useful idealized model of an object whose mass is assumed to be negligible, that is, the mass is considered  insufficient to alter the behavior of the rest of the system. \eqref{e:otvel} expounds that a test particle moves along the integral curves of its velocity field $v^i(t, \bx(t))$ and the acceleration of this test particle is caused only by Newtonian gravity.
\end{remark}

A standard method to obtain the local existence of the Euler--Poisson system is to transform this system into a symmetric hyperbolic system, then the theory of symmetric hyperbolic systems can be applied to derive the local existence. However, there is a severe difficulty when the vacuum appears. The vanishing of the density leads to the degenerated or unbounded coefficients of the symmetric system (detailed explanations of this difficulty and ideas on how to overcome this difficulty can be found, for example, in \cite{Makino1986,Brauer2018,Liu2018a}).
Let us briefly recall Makino's ideas \cite{Makino1986} in this section. In order to avoid the above difficulty, we introduce Makino's density
\begin{align}\label{e:rhal}
\rrho= \Bigl(4K \frac{\gamma}{(\gamma-1)^2}\Bigr)^{-\frac{1}{\gamma-1}  }w^{\frac{2}{\gamma-1}},
\end{align}
that is
\begin{equation}\label{e:rhal2}
w=\frac{2\sqrt{K\gamma}}{\gamma-1}\rrho^{\frac{\gamma-1}{2}}.
\end{equation}
Then, we rewrite  \eqref{e:MNEul1}--\eqref{e:MNEul3}, with the help of the Makino density $w$, as a symmetric hyperbolic system with a nonlocal term $\del{}^k \Phi$,
\begin{align}
\del{0}w+v^i \del{i} w+\frac{\gamma-1}{2 }w\del{i}v^i=&0, \label{e:eual1} \\
\del{0}v^k+v^i \del{i} v^k +\frac{\gamma-1}{ 2}w\delta^{ik}\del{i}w=& - \del{}^k \Phi, \label{e:eual2}  \\
\Delta  \Phi=&  \rrho, \label{e:eual3}
\end{align}
with initial data
\begin{equation*}
(w, v^i)=(w_0, v^i_0), \quad \text{on} \quad  \{0\}\times \Rbb^3,
\end{equation*}

Then using the above formulation, by constructing contraction mappings and with the help of the fixed point theorem, Makino \cite{Makino1986} arrived at the following local existence theorem. We omit the detailed proofs which appeared in  \cite{Makino1986}.

\begin{theorem}\emph{(The local existence of Makino solution, see \cite{Makino1986})} \label{t:Newext}
	Assume the initial data $(\rrho_0, v^i_0)\in C^1(\Rbb^3,\Rbb^4) $, $\rrho_0\geq 0$ and has compact support. If
	\begin{equation*}
	(a) \quad 1<\gamma\leq \frac{5}{3}, \quad (\rrho_0)^{\frac{\gamma-1}{2 }}\in H^3(\Rbb^3) \AND  v^i_0\in H^3(\Rbb^3, \Rbb^3),
	\end{equation*}
	or if
	\begin{equation*}
	(b)\quad 1<\gamma<3, \quad \rrho_0 \in H^3(\Rbb^3), \quad (\rrho_0)^{\frac{\gamma-1}{2 }}\in H^4(\Rbb^3) \AND  v^i_0\in H^4(\Rbb^3, \Rbb^3),
	\end{equation*}
	then there is a constant $T>0$, such that the Cauchy problem \eqref{e:MNEul1}--\eqref{e:inidata2} has a Makino solution $(\rrho, v^i)\in C^1([0,T)\times \Rbb^3, \Rbb^4) $ and $ \Phi \in  C^2([0,T)\times\mathbb{R}^3)$, and the solution $(\rrho^{\frac{\gamma-1}{2 }},v^i)$ satisfies
	\begin{equation*}
		(\rrho^{\frac{\gamma-1}{2 }},v^i)\in C^0([0,T),H^{s}(\Rbb^3))\cap C^1([0,T),H^{s-1}(\Rbb^3))
	\end{equation*}
	where $s=3,4$ according $(a)$ and $(b)$ respectively.
\end{theorem}

\subsection{The local existence of  regular solutions to  Problem \ref{prob}}\label{s:ulecls}
This section gives the local existence theorem of the regular solution to  Problem \ref{prob} for certain  initial data.
We first prove the following lemma which states the restrictions onto $(t,\bx)\in [0,T) \times\overline{\Omega(t)}$ of the Makino solution $(\rrho,v^i)$ of the Makino problem  \eqref{e:MNEul1}--\eqref{e:inidata2} is a regular solution to  Problem \ref{prob}.
\begin{lemma}\label{t:relaclta}
	Suppose there is a constant $T>0$ and $(\rrho, v^i,\rPhi)$ is a Makino solution to the system \eqref{e:MNEul1}--\eqref{e:inidata2}.  Define $\Omega(t):= \supp\rrho (t,\cdot) $ and a function $\rw^i:[0,T) \times \Omega(t) \rightarrow \Rbb^3$ satisfying that
	\begin{equation*}
	\rw^i(t,\bx):=v^i(t,\bx)
	\end{equation*}
	for $(t,\bx)\in [0,T) \times \Omega(t)$.
	Then  $(\rrho,\rw^i,\rPhi, \Omega(t))$ is a regular solution of  Problem \ref{prob}.
\end{lemma}
\begin{proof}
	To conclude this lemma, we only need to compare Definition \ref{t:clsl} of regular solutions and Definition \ref{e:tmsl} of Makino solutions. Noting $\del{\mu}\rho^{\gamma-1}=2\rrho^{\frac{\gamma-1}{2}}\del{\mu}\rrho^{\frac{\gamma-1}{2}} $, we conclude  $\rrho^{\frac{\gamma-1}{2}}\in C^1([0,T)\times \Rbb^3,\Rbb)$ implies  $\rrho^{\gamma-1}\in C^1([0,T)\times \Rbb^3,\Rbb)$. Since $v^i\in C^1([0,T)\times \Rbb^3,\Rbb^3)$ and $\rw^i:=v^i$ in $[0,T)\times\Omega(t)$, then there is an extension $v^i|_{[0,T)\times\overline{\Omega(t)}}$ of $\rw^i$ to the boundary $\del{}\Omega(t)$ and $v^i|_{[0,T)\times\overline{\Omega(t)}}\in C^1([0,T)\times \overline{\Omega(t)},\Rbb^3)$. 	
  	Then, we complete the proof.
\end{proof}

Now let us give the local existence theorem of regular solutions to  Problem \ref{prob}.
\begin{theorem}[The local existence theorem]  \label{t:Newext2}
	Suppose the initial data $(\rrho_0, \rw^i_0)$ of  Problem \ref{prob}  is given by  \eqref{e:inidata}, $\Omega(0)$ is precompact and satisfies   Calder\'{o}n’s uniform cone condition (see Appendix \ref{s:surf}, Definition \ref{t:conecd}), $\rrho_0\in C^1( \Rbb^3,\Rbb_{\geq 0}) $, $\rw^i_0 \in C^1(\Omega(0),\Rbb^3) $. If
	\begin{equation*}
    (a) \quad 1<\gamma\leq \frac{5}{3}, \quad (\rrho_0)^{\frac{\gamma-1}{2 }} \in H^3(\Rbb^3) \AND \rw^i_0\in H^3(\Omega(0),\Rbb^3),
	\end{equation*}
	or if
	\begin{align*}
	(b)\quad 1<\gamma<2, \quad  \rrho_0 \in H^3(\Rbb^3), \quad (\rrho_0)^{\frac{\gamma-1}{2 }}
	\in H^4(\Rbb^3) \AND
	 \rw^i_0\in H^4(\Omega(0),\Rbb^3),
	\end{align*}
	then there is a constant $T>0$, such that  Problem \ref{prob} has a regular solution $(\rrho,\rw^i,\rPhi, \Omega(t)) $ defined by Definition \ref{t:clsl} and this  regular solution is a Makino type solution with a $C^1$ domain.  Moreover,  $\rrho^{\frac{\gamma-1}{2 }}
	\in C^{s-2}([0,T)\times\Rbb^3)$ and $ \Phi \in  C^2([0,T)\times\mathbb{R}^3)$, and the solution $v^i$ satisfies
	\begin{equation*}
     \rw^i\in C^0([0,T),H^{s}(\Omega(t)))\cap C^1([0,T),H^{s-1}(\Omega(t)))
	\end{equation*}
	where $s=3,4$ according to $(a)$ and $(b)$ respectively.
\end{theorem}
\begin{proof}
	We take three steps to prove this theorem:
	
	\underline{Step $1$: Extend $\rw_0^i$ from $\Omega(0)$ to $\Rbb^3$.} Since $\Omega(0)$ satisfies Calder\'{o}n’s uniform cone condition, by applying   Calder\'{o}n extension theorem (see Appendix \ref{s:ana}, Theorem \ref{t:Cdnext}), there exists a simple $(s,2)$-extension  operator $E:H^s(\Omega(0)) \rightarrow H^s(\Rbb^3)$ (for $s=3,4$), such that $v_0^i:=E\rw^i_0=\rw^i_0$ in $\Omega(0)$ for every $\rw^i_0\in H^s(\Omega(0))$, and $\|v^i_0\|_{H^s(\Rbb^3)}\leq K\|\rw^i_0\|_{H^s(\Omega(0))}<\infty$, i.e., the extended data $v^i_0\in H^s(\Rbb^3)\subset C^1(\Rbb^3)$.
	
	\underline{Step $2$: Makino's local existence implies the existence of an extended regular solution.} Using the extended data $(\rrho_0, v_0^i)\in C^1(\Rbb^3,\Rbb^4)$ and noting that $(\rrho_0, v_0^i)$ satisfies the requirements $(a)$ or $(b)$ in Theorem \ref{t:Newext}, Theorem \ref{t:Newext} implies there is a constant $T>0$, such that the Cauchy problem \eqref{e:MNEul1}--\eqref{e:inidata2} has a Makino solution $(\rrho, v^i)\in C^1([0,T)\times \Rbb^3, \Rbb^4) $ and $ \Phi \in  C^2([0,T)\times\mathbb{R}^3)$, which also implies $\rrho^{\frac{\gamma-1}{2}}\in C^{s-2}([0,T)\times \Rbb^3)$ by the Definition \ref{e:tmsl} of Makino solutions.
	
	\underline{Step $3$: Restrain Makino solutions to regular solutions.} Now we have Makino solutions $(\rrho,v^i,\rPhi)$ to Makino's problem \eqref{e:MNEul1}--\eqref{e:inidata2}. by defining $\Omega(t):= \supp\rrho (t,\cdot) $ and a function $\rw^i:[0,T) \times \Omega(t) \rightarrow \Rbb^3$ such that $
	\rw^i(t,\bx):=v^i(t,\bx)$,
	for $(t,\bx)\in [0,T) \times \Omega(t)$.  We are able to use Lemma \ref{t:relaclta} to conclude $(\rrho,\rw^i,\rPhi, \Omega(t))$ is a regular solution to the initial value Problem \ref{prob}.
	In the end, noting that for any $H^s(\Rbb^3)$-function $f$, $\|f\|_{H^s(\Omega(t))}\leq \|f\|_{H^s(\Rbb^3)}$ due to the definition of Calder\'{o}n extension, we complete the proof.
\end{proof}
\begin{remark}
	Since we do not attempt to prove the general existence theorem of the regular solution,
	there may be regular solutions of Problem \ref{prob} which are not from Makino solutions. However, if there is a Makino solution, by  the uniqueness theorem \ref{t:unithm1} later, it is the only solution and there is no other non-Makino type of regular solutions.
\end{remark}

\subsection{Preservation of $C^1$-boundary}\label{s:C1bdy}
We assume the initial boundary $\del{}\Omega(0)$ of the support of the fluid  is a $C^1$ boundary, and, in this section, prove $C^1$ regularity of this boundary preserves during the evolution, i.e., $\del{}\Omega(t)\in C^1$ for $t$ in the existence interval of the regular solution.

\begin{lemma}\label{t:surfc2}
	Suppose $k\in\Zbb_{\geq1}$, $\Omega_0\subset \Rbb^3$ is a precompact domain satisfying  $\del{}\Omega_0\in C^k$, $C^k \ni \mathbf{v}:=( v^i): (T_0,T_1) \times \Rbb^3 \rightarrow \Rbb^3$ $(T_0<0<T_1)$ is a vector field.
	Then
	\begin{enumerate}
		\item \label{t:surfc2.a} \emph{(existence of flow)} there exists a bounded subset $ [0,T)  \times [0,T)  \times \mathcal{D} \subset [0,T)  \times [0,T)  \times \Rbb^3$ where $T\leq T_1$ and $\mathcal{D}\supset \Omega_0$ and a unique time-dependent flow
		$C^k \ni \varphi: [0,T)  \times [0,T)  \times \mathcal{D}
		\ni (t,t_0,\xi) \mapsto \bx\in \Rbb^3$ generalized by the vector field $\mathbf{v}$;
				
		\item \label{t:surfc2.a2}\emph{(integral curves)} for every $\xi \in \mathcal{D}
		$, the curve $\varphi^{(0,\xi)}:=\varphi(\cdot,0,\xi):[0,T)  \rightarrow \Rbb^3$ is the unique maximal integral curve of the vector field $\mathbf{v}$ starting at $(0,\xi)$ (i.e., $\varphi^{(0,\xi)}(0)=\xi$);

		\item \label{t:surfc2.a3}\emph{(inverse)} for every $(t,t_0)\in [0,T) \times [0,T) $,  $\varphi^{(t,t_0)}:=\varphi(t,t_0,\cdot):\Rbb^3\rightarrow \Rbb^3$ is a $C^k$  diffeomorphism with inverse $\varphi^{(t_0,t)}$;
		
		\item \label{t:surfc2.b} \emph{(diffeomorphism invariance of the boundary)} denote $\Omega_t:=\varphi^{(t, 0)}(\Omega_0)$, then $\del{}\Omega_t =\varphi^{(t, 0)}\bigl(\del{}\Omega_0 \bigr)$ and  $\Rbb^3\setminus \overline{\Omega_t} =\varphi^{(t, 0)}\bigl(\Rbb^3\setminus \overline{\Omega_0} \bigr)$;

		\item \label{t:surfc2.c} \emph{(regularity of evolutional boundary)}  $\del{}\Omega_t\in C^k$;

		\item \label{t:surfc2.d} \emph{(foliations)}
		the lateral boundary $[0,T)  \times \del{}\Omega_t$ is foliated by  integral curves starting from $\del{}\Omega_0$, that is, $
		[0,T) \times \del{}\Omega_t=  \bigl\{\bigl(t,\varphi^{(0,\xi)}(t)\bigr)\;|\;t\in [0,T) , \xi\in \del{}\Omega_0\bigr\}$.
	\end{enumerate}
\end{lemma}

\begin{proof}
	\eqref{t:surfc2.a}, \eqref{t:surfc2.a2} and \eqref{t:surfc2.a3} are the direct consequence of the fundamental theorem on flows (see Theorem \ref{t:flth2} and \ref{t:flowft}).
	
    \eqref{t:surfc2.b} is due to the diffeomorphism invariance of the boundary, see Theorem \ref{t:bdrythm} (see Appendix \ref{s:surf}).
	
	\eqref{t:surfc2.c} Since $\del{}\Omega_0\in C^k$, by Definition \ref{t:bdryreg} and Theorem \ref{t:bdsbmf} (see Appendix \ref{s:surf}), implies that for every point $\xi_0\in \del{}\Omega_0$, there is a ball $B_r(\xi_0)$ and a bijective mapping $\psi:B_r(\xi_0)\rightarrow \mathcal{B}\subset \Rbb^3$ such that $(1)$ $\psi(B_r(\xi_0)\cap \Omega_0)\subset \Rbb^3_+$; $(2)$ $\psi(B_r(\xi_0)\cap \del{} \Omega_0)\subset \del{}\Rbb^3_+$; $(3)$ $\psi\in C^k$ and $\psi^{-1}\in C^k$.
	It is direct that
	$\varphi \in C^k$ yields 	
	$\varphi^{(t,0)} \in C^k(\Rbb^3,\Rbb^3)$. Then for any point $\bx_0\in \del{}\Omega_t$, there is $\xi_0\in \del{}\Omega_0$, some neighborhood $\mathcal{U}(\bx_0)$ of $\bx_0$ and $r,R>0$, such that $\bx_0:=\varphi^{(t,0)}(\xi_0)$ and $B_{R}(\bx_0)\subset \mathcal{U}(\bx_0)=\varphi^{(t,0)}(B_r(\xi_0))$ (since $ \varphi^{(t,0)}$ is a $C^k$ diffeomorphism with inverse $ (\varphi^{(t,0)})^{-1}=\varphi^{(0,t)}$ by \eqref{t:surfc2.a}). Then there is a bijective mapping $\psi\circ(\varphi^{(t,0)})^{-1} :B_R(\bx_0)\rightarrow \mathcal{B} \subset \Rbb^3$ such that, we \textit{claim}, $(1)$ $\psi\circ(\varphi^{(t,0)})^{-1}(B_R(\bx_0)\cap \Omega_t)\subset \Rbb^3_+$; $(2)$ $\psi\circ(\varphi^{(t,0)})^{-1}(B_R(\bx_0)\cap \del{} \Omega_t)\subset \del{}\Rbb^3_+$; $(3)$ $\psi\circ(\varphi^{(t,0)})^{-1}\in C^k$ and $\bigl((\varphi^{(t,0)})^{-1}\circ \psi\bigr)^{-1}\in C^k$. We expound $(1)$ explicitly and then $(2)$ can be similarly proved. The inclusion relations of mappings, with the help of that $ (\varphi^{(t,0)})^{-1}=\varphi^{(0,t)}$, $\Omega_t =\varphi^{(t,0)}( \Omega_0)$ and $\del{}\Omega_t =\varphi^{(t,0)}(\del{}\Omega_0)$ by \eqref{t:surfc2.b}, indicate that
	\begin{align*}
	(\varphi^{(t,0)})^{-1}(B_R(\bx_0)\cap \Omega_t)\subset & (\varphi^{(t,0)})^{-1}\bigl(B_R(\bx_0)\bigr)\cap (\varphi^{(t,0)})^{-1}\bigl(\Omega_t\bigr)  \notag  \\
	\subset & (\varphi^{(t,0)})^{-1}\bigl(\mathcal{U}(\bx_0)\bigr)\cap \Omega_0
	= B_r(\xi_0)\cap \Omega_0.
	\end{align*}
	Then using $\psi(B_r(\xi_0)\cap \Omega_0)\subset \Rbb^3_+$, we conclude that $\psi\circ(\varphi^{(t,0)})^{-1}(B_R(\bx_0)\cap \Omega_t)\subset \Rbb^3_+$. This, by Definition \ref{t:bdryreg} and Theorem \ref{t:bdsbmf} again, concludes that  $\del{}\Omega_t\in C^k$.

	\eqref{t:surfc2.d} Using the previous \eqref{t:surfc2.b},
	\begin{align*}
	[0,T)  \times \del{}\Omega_t
	=[0,T)  \times  \varphi^{(t,0)}(\del{}\Omega_0)
	= & \bigl\{\bigl(t,\varphi (t,0, \xi)\bigr)\;|\;t\in [0,T) , \xi\in \del{}\Omega_0 \bigr\}\notag  \\
	= &  \bigl\{\bigl(t,\varphi^{(0,\xi)}(t)\bigr)\;|\;t\in [0,T) , \xi\in \del{}\Omega_0\bigr\},
	\end{align*}	
	which means all the integral curves $\varphi^{(0,\xi)}(t)$ (for all $t\in [0,T)  $ and $\xi\in \del{}\Omega_0$) form the lateral boundary $[0,T)  \times \del{}\Omega_t$.
\end{proof}

This lemma gives a geometric preparation for the boundary evolution of  Problem \ref{prob}. Under the assumptions of the local existence Theorem \ref{t:Newext2}, in order to prove $\del{}\Omega(t)\in C^k$, we have the velocity field to play the role of the above $C^k$ vector field, then we have $\del{}\Omega_t\in C^k$ by using Lemma \ref{t:surfc2} and letting $\Omega_0=\Omega(0)$. However, this has not implied $\del{}\Omega(t)\in C^k$ yet, since, by recalling the definition, $\Omega(t):=\supp\rrho(t,\cdot) =\{\bx\in \Rbb^3\;|\;\rrho(t,\bx)>0\}$. Therefore, to conclude $\del{}\Omega(t)\in C^k$, we have to verify $\varphi^{(t,0)}\bigl(\Omega(0)\bigr)=\Omega_t=\Omega(t)=\supp\rrho(t,\cdot)$.

\begin{theorem}\label{t:C1bdry}
	Suppose $(\rrho,\rw^i,\rPhi,\Omega(t))$ is a regular solution to  Problem \ref{prob} on $t\in[0,T)$ for $T>0$, if $\del{}\Omega(0)$ is a $C^1$ boundary, then $\del{}\Omega(t)$ is of $C^1$ for  $t\in [0,T)$ as well.
\end{theorem}
\begin{proof}
\underline{Firstly, let us derive a flow and $\Omega_t$.} Steps $1$ and $2$ from the proof of Theorem \ref{t:Newext2} imply there is a vector field $ v^i\in C^1((T_0,T_1)\times \Rbb^3, \Rbb^3)$ where $T_0<0<T_1$ (using these steps in future and past respectively). Then by letting $\Omega_0:=\Omega(0)$ and Lemma \ref{t:surfc2}, we obtain the flow $\varphi$ described by the preceding lemma.
Denote the flow $\varphi$ restricted on the submanifold $[0,T)\times \{0\}\times \mathcal{D}	$ by $\chi:[0,T)\times \mathcal{D} \ni (t,\xi)\mapsto \bx\in \Rbb^3$,
\begin{equation*}
	\chi(t,\xi):=\varphi(t,0,\xi), \quad \chi^t(\xi):=\chi(t,\xi)=\varphi^{(t,0)}(\xi) \AND \chi^\xi(t):=\chi(t,\xi)=\varphi^{(0,\xi)}(t) .
\end{equation*}
Then $\Omega_t=\chi^t\bigl(\Omega(0)\bigr)$ and $\del{}\Omega_t\in C^1$ by Lemma \ref{t:surfc2}.\eqref{t:surfc2.c}.

\underline{Secondly, we prove $\Omega(t)=\Omega_t=\chi^t\bigl(\Omega(0)\bigr)$.} In other words, we have to verify that the image of the initial support of the density is still the support of the density, i.e., $\chi^t\bigl(\supp\rrho(0,\cdot)\bigr) =\supp\rrho(t,\cdot)$. For every $\bx\in \chi^t\bigl(\Omega(0)\bigr)$, there exists a point $\xi\in \Omega(0)$ which means $\rrho(0,\xi) > 0$, such that $\bx= \chi^t(\xi) $. Because of the fact $\chi^{t}(\xi)=\chi^\xi(t)$ for every $\xi\in \Omega(0)$, by denoting the integral curve (characteristic) starting from $\chi^\xi(0)=\xi$ by
\begin{equation*}
    \chi^\xi(t)=\chi(t,\xi)=\bx,
\end{equation*}
we have (by the definition of integral curves)
\begin{align}\label{e:dxw1}
\frac{d}{dt}  \chi^\xi(t) = \mathbf{v}(t, \chi^\xi(t)).  
\end{align}
Then \eqref{e:NEul1} for any point on the integral curve $\chi^\xi(t)$, with the help of \eqref{e:dxw1}, yields
\begin{align*}
\frac{d}{dt} \rrho(t,\chi^\xi(t))=&D_t\rrho(t,\bx)=\del{0}\rrho(t,\bx)+v^i(t,\bx) \del{i} \rrho(t,\bx) \notag \\
=&- \rrho(t,\chi^\xi(t)) \del{i} v^i(t,\chi^\xi(t)),
\end{align*}
which implies that
\begin{align}\label{e:aa0}
\rrho(t, \bx)=\rrho( 0 , \xi) \exp{\Bigl(-  \int^t_0 \del{i} v^i(\tau,\chi^\xi(\tau)) d\tau\Bigr)}.
\end{align}
This expression \eqref{e:aa0} directly leads to $\rrho(t, \bx)>0$ which means $\bx\in \Omega(t)$, and further, we derive $\chi^t\bigl(\Omega(0)\bigr) \subset \Omega(t)$.

On the other hand, for every $\bx\in \Omega(t)$, that is, $\rrho(t,\bx)>0$. Still  \eqref{e:aa0} implies $\rrho( 0 , \xi)>0$ where, as above, $\bx= \chi^t(\xi) $. This yields $\xi\in\Omega(0)$, $\bx\in \chi^t(\Omega(0))$ and, in turn, concludes  $\Omega(t) \subset \chi^t\bigl(\Omega(0)\bigr)$.

Therefore, $\Omega(t)=\Omega_t=\chi^t\bigl(\Omega(0)\bigr)$, which, in turn, implies $\del{}\Omega(t)\in C^1$ due to $\del{}\Omega_t\in C^1$, and it is direct that $T$ can be extended to the existence interval of the solution by using the above arguments continuously. Then we complete this proof. 	
\end{proof}

The following Proposition \ref{t:bdgl} gives the complete boundary $\del{}\bigl([0,T]\times  \Omega(t)\bigr)$ (lateral, top and bottom boundaries) is of Lipschitz. This property enables us to use Stein Extension Theorem \ref{t:extentionS} (see Appendix \ref{s:extentionS}) to extend the solution in $[0,T]\times \Omega(t)$ to the whole $\Rbb^4$, which is a crucial step for the uniqueness theorem.

\begin{proposition}[Boundary Gluing]\label{t:bdgl}
	Suppose $\mathcal{D}_l:=[0,T]\times \del{}\Omega(t)$, $\mathcal{D}_t:=\{T\}\times\Omega(T)$ and $\mathcal{D}_b:=\{0\}\times\Omega(0)$ are  lateral, top and bottom boundaries given by the solution of  Problem \ref{prob}, and $\del{}\Omega(0)$ is a $C^1$ boundary of $\Omega(0)$. Then the lateral boundary $\mathcal{D}_l=[0,T]\times \del{}\Omega(t)\in C^1$ and the complete boundary
	\begin{align*}
		\del{}\bigl([0,T]\times  \Omega(t)\bigr)=\mathcal{D}_l \cup \mathcal{D}_b \cup \mathcal{D}_t
	\end{align*}
is a Lipschitz boundary.
\end{proposition}
We put the proof of this proposition in Appendix \ref{s:bdgl} to avoid deviating from the main objective too far.

\subsection{The uniqueness of the solution to  Problem \ref{prob}}\label{s:unqthm}
In this section, we prove the uniqueness theorem of the regular solution of  Problem \ref{prob} (the proof of this theorem is inspired by the relative entropy-entropy flux method given in \cite[\S $2$, Step $2$]{Luo2014}). This theorem helps us confirm that there is only one regular solution and this one has to be the restriction of the Makino one, and rule out all other possibilities if the initial data are in the form of Theorem \ref{t:Newext2}.$(1)$ and $(2)$. Although the exterior data of the velocity $v^i$ of the corresponding Makino solutions can be selected with certain freedoms, the regular solution is unique and independent of selections of the exterior data of the velocity of the corresponding Makino solutions once the interior data of the velocity are fixed. This phenomenon is essentially due to the degenerated sound cones near the boundary (see \cite{Brauer1998} for details). Before stating the uniqueness theorem, we first present the following lemma which is a direct result of Lemma \ref{t:unqr3} (see Appendix \ref{a:Lpf}).  
\begin{lemma}\label{t:samedt}
	Under the assumptions of Lemma \ref{t:unqr3}, if initial data satisfy
	\begin{equation*}
	\Omega_1(0)=\Omega_2(0), \quad \rrho_1(0,\bx)=\rrho_2(0,\bx) \AND v^i_1(0,\bx)=v^i_2(0,\bx)
	\end{equation*}
	for $\bx\in \Omega_1(0)=\Omega_2(0)$, then the solutions satisfy
	\begin{equation*}
	\Omega_1(t)=\Omega_2(t), \quad \rrho_1(t,\bx)=\rrho_2(t,\bx) \AND v^i_1(t,\bx)=v^i_2(t,\bx),
	\end{equation*}
	for $(t,\bx)\in[0,T)\times\Omega_1(t)\bigl(=[0,T)\times\Omega_2(t)\bigr)$.
\end{lemma}
\begin{proof}
	Lemma \ref{t:unqr3} (see Appendix \ref{a:Lpf}) implies
	\begin{align}\label{e:inteta1}
		\int_{\Rbb^3} \eta^\star(t,\bx)d\bx  \leq  \int_{\Rbb^3} \eta^\star(0,\bx)d\bx+ C\sup_{0\leq \tau<T}\bigl(\|\nabla_\bx \mathbf{v}_1(\tau,\cdot)\|_{L^\infty}+Z(\tau)\bigr)\int^t_0\int_{\Rbb^3} \eta^\star(\tau,\bx) d\bx d\tau
	\end{align}
	where $\eta^\star\geq 0$, $Z(\tau)$ are defined in Lemma \ref{t:unqr3}.
	Note $\eta^\star (t,\bx) \equiv 0$ for any $(t,\mathbf{x})\in [0,T)\times (\Omega_1(t)\cup \Omega_2(t))^\mathsf{c}$ and use the initial data
	\begin{equation*}
		\Omega_1(0)=\Omega_2(0), \quad \rrho_1(0,\bx)=\rrho_2(0,\bx) \AND \rw^i_1(0,\bx)=\rw^i_2(0,\bx)
	\end{equation*}
	for $\bx\in \Omega_1(0)=\Omega_2(0)$, then the above domain of integration can be shrunk to $\Omega_1(t)\cup\Omega_2(t)$ and above inequality \eqref{e:inteta1}, with the help of $v^i_\ell \in
	W^{1,\infty}([0,T)\times \Rbb^3, \Rbb^3)$, becomes
	\begin{align*}\label{e:inteta2}
		\int_{\Omega_1(t)\cup\Omega_2(t)} \eta^\star(t,\bx)d\bx \leq & \int_{\Omega_1(0)} \eta^\star(0,\bx)d\bx \notag \\
		&+C\sup_{0\leq \tau<T}(\|\nabla_\bx \mathbf{v}_1(\tau,\cdot)\|_{L^\infty}+Z(\tau))\int^t_0\int_{\Omega_1(\tau))\cup\Omega_2(\tau))} \eta^\star(\tau,\bx) d\bx d\tau \notag  \\
		\leq & C^\prime \int^t_0\int_{\Omega_1(\tau)\cup\Omega_2(\tau))} \eta^\star(\tau,\bx) d\bx d\tau.
	\end{align*}
	Then by the Gronwall's inequality, we arrive at
	\begin{align*}
		\int_{\Omega_1(t)\cup\Omega_2(t)} \eta^\star(t,\bx)d\bx \equiv 0
	\end{align*}
	for $t\in[0,T)$. Since $\eta^\star\geq 0$ (see \eqref{e:etaest} in Lemma \ref{t:unqr3}), then $\eta^\star \equiv 0$ for $(t,\bx)\in[0,T)\times (\Omega_1(t)\cup \Omega_2(t))$, by \eqref{e:etaest} in Lemma \ref{t:unqr3}, we conclude 	
	\begin{equation*} \rrho_1(t,\bx)=\rrho_2(t,\bx) \AND \rw^i_1(t,\bx)=\rw^i_2(t,\bx),
	\end{equation*}
	for $(t,\bx)\in[0,T)\times (\Omega_1(t)\cup \Omega_2(t))$ and further $
	\Omega_1(t)=\Omega_2(t)$ since they are supports of densities $\rrho_1$ and $\rrho_2$, respectively. Then we complete this proof.
\end{proof}

\begin{definition}\label{t:exsl}
	For a solution $(\rrho,\rw^i,  \Phi, \Omega(t))$ to Euler--Poisson system \eqref{e:NEul1}--\eqref{e:NEul3}  on $[0,T)\times \Omega(t)$, suppose $\rrho\in C^1([0,T)\times\Rbb^3)$ and $\rw^i\in C^1([0,T)\times \Omega(t), \Rbb^3)$, and if there exists an extension function $v^i  \in
	W^{1,\infty}([0,T)\times \Rbb^3, \Rbb^3)$, such that $v^i(t,\bx)=\rw^i(t,\bx)$ for $(t,\bx)\in[0,T)\times \Omega(t)$, then we call this solution $(\rrho,\rw^i,  \Phi, \Omega(t))$ \textbf{extendable}.
\end{definition}

Note that the local existence theorem \ref{t:Newext2} gives the local  solutions to Problem \ref{prob} by Calder\'on extension of the data and the symmetric hyperbolic system theory. Thus, it is a proper subclass of regular solutions with $\rrho^{\frac{\gamma-1}{2}}\in C^1([0,T)\times \Rbb^3,\Rbb)$. 
The following uniqueness Theorem \ref{t:unithm} indicates any  solutions (not necessary to be the regular solutions of  Problem \ref{prob}) with the same initial data in the hydrodynamic region $\Omega(t)$ are the same if these solutions are \textit{extendable}.
After giving the uniqueness theorem \ref{t:unithm}, we point out, in the strong uniqueness theorem \ref{t:unithm1}, every regular solution defined by Definition \ref{t:clsl} is extendable, and with the help of the uniqueness theorem \ref{t:unithm}, it follows the regular solution exists and is unique under the data given in the local existence theorem \ref{t:Newext2}.

\begin{theorem}\emph{(The uniqueness theorem)} \label{t:unithm}
	Suppose $1<\gamma\leq 2$, $T>0$, $(\rrho_\ell,\rw^i_\ell,  \Phi_\ell, \Omega_\ell(t))$,
	for $\ell=1,2$, are two \emph{extendable} solutions to the Euler--Poisson system \eqref{e:NEul1}--\eqref{e:NEul3}  on $[0,T)\times \Omega(t)$, that is they are, respectively, two extendable solutions of the system,
	\begin{align}
		\del{0} \rrho_\ell + \mathring{v}^i_\ell \del{i} \rrho_\ell+ \rrho_\ell \del{i} \mathring{v}^i_\ell = & 0  \quad &&\text{in}\quad [0,T)\times \Omega_\ell(t) , \label{e:treul1}\\		
		\rrho_\ell \del{0} \mathring{v}^k_\ell + \rrho_\ell \mathring{v}^i_\ell \del{i} \mathring{v}^k_\ell + \delta^{ik} \del{i}  p_\ell = & -   \rrho_\ell \del{}^k  \Phi_\ell, && \text{in}\quad [0,T)\times \Omega_\ell(t),  \label{e:treul2}
	\end{align}
	where
	\begin{equation*} 
		\Phi_\ell(t,\bx)=   -\frac{1}{4\pi}\int_{\Rbb^3} \frac{\rho_\ell(t,\by)}{|\bx - \by |}d^3 \by, \quad \text{for} \quad (t, \bx) \in  [0,T)\times\Rbb^3 , 	
	\end{equation*}
	and the equation of state \eqref{e:eos1} holds.
	If initial data are the same,
	\begin{equation*}
		\Omega_1(0)=\Omega_2(0), \quad \rrho_1(0,\bx)=\rrho_2(0,\bx) \AND \rw^i_1(0,\bx)=\rw^i_2(0,\bx)
	\end{equation*}
	for $\bx\in \Omega_1(0)=\Omega_2(0)$, then the solutions are the same as well,
	\begin{equation*}
		\Omega_1(t)=\Omega_2(t), \quad \rrho_1(t,\bx)=\rrho_2(t,\bx) \AND \rw^i_1(t,\bx)=\rw^i_2(t,\bx),
	\end{equation*}
	for $(t,\bx)\in[0,T)\times\Omega_1(t)\bigl(=[0,T)\times\Omega_2(t)\bigr)$.
\end{theorem}

\begin{proof}
	Firstly, direct calculations (note \eqref{e:treul1}$\times\rw^k_\ell+$\eqref{e:treul2} yields the following second equation) lead to that for any $(t,\bx)\in[0,T)\times \Omega(t)$ the system \eqref{e:treul1}--\eqref{e:treul2} is equivalent to
	\begin{align*}
		\del{t}\rrho_\ell+ \del{i} (\rrho_\ell \rw^i_\ell) = & 0 \quad &&\text{in}\quad [0,T)\times \Omega_\ell(t) , 
		\\
		\del{t}(\rrho_\ell \rw^k_\ell)+\del{i} (\rrho \rw^i_\ell \rw^k_\ell) +\delta^{ik} \del{i} p_\ell = & - \rrho_\ell\del{ }^k  \Phi_\ell, && \text{in}\quad [0,T)\times \Omega_\ell(t).
	\end{align*}
	Since $(\rrho_\ell,\rw^i_\ell,  \Phi_\ell, \Omega_\ell(t))$,
	for $\ell=1,2$, are extendable (using Definition \ref{t:exsl}), there are extensions $v^i_\ell \in C^1([0,T)\times\overline{\Omega(t)}, \Rbb^3) \cap
	W^{1,\infty}([0,T)\times \Rbb^3, \Rbb^3)$ of $\rw^i$, respectively. Then by direct calculations, we can verify that $\rrho_\ell v^i_\ell \in C^1([0,T)\times\Rbb^3)$, $\rrho v^i_\ell v^k_\ell \in C^1([0,T)\times\Rbb^3)$ (note $\del{\mu}(\rrho_\ell v^i_\ell)|_{\del{}\Omega_\ell(t)}=(\del{\mu}\rrho_\ell v^i_\ell)|_{\del{}\Omega_\ell(t)}+(\rrho_\ell \del{\mu} v^i_\ell)|_{\del{}\Omega_\ell(t)}=0$) and they both solve the Euler--Poisson system given in Lemma \ref{t:unqr3}.  After verifying them, we conclude this theorem with the help of Lemma \ref{t:samedt}.
\end{proof}

Above Theorem \ref{t:unithm} claims the \textit{extendable} solution is unique, the following Theorem \ref{t:unithm1} states with the stronger conditions of the local existence Theorem \ref{t:Newext2}, the regular solutions are all extendable, furthermore, under the conditions of the local existence Theorem \ref{t:Newext2}, the regular solution is unique.
\begin{theorem}[The strong uniqueness theorem]\label{t:unithm1}
	Suppose $1<\gamma\leq 2$, $T>0$, $(\rrho_\ell,\rw^i_\ell,  \Phi_\ell, \Omega_\ell(t))$,
	for $\ell=1,2$, are two regular solutions to Problem \ref{prob}, defined by Definition \ref{t:clsl} on $[0,T)$.
	If the initial domain $\Omega(0)$ is bounded and  satisfies $\del{}\Omega(0)$ is a $C^1$ boundary, and initial data are the same,
	\begin{equation*}
		\Omega_1(0)=\Omega_2(0), \quad \rrho_1(0,\bx)=\rrho_2(0,\bx) \AND \rw^i_1(0,\bx)=\rw^i_2(0,\bx)
	\end{equation*}
	for $\bx\in \Omega_1(0)=\Omega_2(0)$, then the solutions are the same as well,
	\begin{equation*}
		\Omega_1(t)=\Omega_2(t), \quad \rrho_1(t,\bx)=\rrho_2(t,\bx) \AND \rw^i_1(t,\bx)=\rw^i_2(t,\bx),
	\end{equation*}
	for any $T_\star\in(0,T)$ and  $(t,\bx)\in[0,T_\star]\times\Omega_1(t)\bigl(=[0,T_\star]\times\Omega_2(t)\bigr)$.

	Furthermore, under the conditions of the local existence Theorem \ref{t:Newext2}, and  $\del{}\Omega(0)$ is a $C^1$ boundary of $\Omega(0)$, then there is a unique regular solution $(\rrho,\rw^i,\rPhi, \Omega(t))$ in $t\in [0,T_\star]$ given by Theorem \ref{t:Newext2} for every $0<T_\star<T$.
\end{theorem}
\begin{remark}
	This uniqueness theorem implies that if we have two extensions for the initial data on $\Omega(0)$. Then for these two different extensions, there would exist two different Makino solutions, that then needed to be restricted to $\Omega(t)$. This restriction then should result in a unique solution. 
\end{remark}
\begin{proof}
	Let us first claim every regular solution  to  Problem \ref{prob}   is extendable. In fact, since $(\rrho,\rw^i,  \Phi, \Omega(t))$ is a regular solution to  Problem \ref{prob}, by Definition \ref{t:clsl}, $\rrho \in C^1([0,T)\times \Rbb^3,\Rbb)$ and $v^i \in C^1([0,T)\times \overline{\Omega(t)},\Rbb^3)$. Due to the compactness of $\overline{\Omega(t)}$, we obtain  $v^i \in W^{1,\infty}([0,T_\star]\times \overline{\Omega(t)},\Rbb^3)$ for any $T_\star\in(0,T)$. Since $\del{}\Omega(0)$ is a $C^1$ boundary, we apply Proposition \ref{t:bdgl} to conclude $\del{}\bigl([0,T_\star]\times  \Omega(t)\bigr)=\mathcal{D}_l \cup \mathcal{D}_b \cup \mathcal{D}_t \subset \Rbb^4$ is of Lipschitz.
	With the help of \textit{Stein extension theorem \ref{t:extentionS}} (see Appendix \S\ref{s:extentionS}) for $(m,p)=(1,\infty)$, there is a Stein extension operator $E:W^{1,\infty}([0,T_\star]\times \overline{\Omega(t)})\rightarrow W^{1,\infty}(\Rbb^4)$ satisfying
$Ev^i(x^\mu)=v^i(x^\mu)$ in $[0,T_\star]\times \overline{\Omega(t)}$ and there is a constant $K>0$ such that $\|Ev^i\|_{W^{1,\infty}(\Rbb^4)}\leq K\|v^i\|_{W^{1,\infty}([0,T_\star]\times \overline{\Omega(t)})}$. Therefore, by Definition \ref{t:exsl}, the regular solution is extendable.
Further,  using the uniqueness theorem \ref{t:unithm} completes the proof of the first part.

	 If assuming conditions of the local existence theorem \ref{t:Newext2} with $C^1$ boundary $\del{}\Omega(0)$ of $\Omega(0)$ for $t\in[0,T)$, then there is at least one Makino solution given by Theorem \ref{t:Newext2} and $\rw^i\in C^0([0,T),H^{s}(\Omega(t)))\cap C^1([0,T),H^{s-1}(\Omega(t)))$. Using the above result, we conclude any two regular solutions are the same. That is, there is a unique regular solution of the system and the unique solution has to be the Makino type solution with the compact data.
	  We complete this proof.
\end{proof}

\begin{remark}	
	We point out \cite[Proposition $2$]{Brauer1998} proved the uniqueness theorem of the \textit{Makino solutions}  instead of all \textit{regular solutions}, since \cite{Brauer1998} presumes two solutions solve Makino's formulations of the Euler--Poisson system and energy estimates lead to they are equal in the lenslike shape region if data are the same.  This can be done only if $\rrho^{\frac{\gamma-1}{2}}(t)\in H^s$ rather than $\rrho^{\gamma-1}(t)\in H^s$ for all regular solutions. Therefore, the uniqueness of \cite{Brauer1998} implies the solution satisfying  $\rrho^{\frac{\gamma-1}{2}}(t)\in H^s$ is unique. In other words, it can not rule out the possibility that there are more than one solutions to solve  Problem \ref{prob} with the same data, one of them is the Makino solution satisfying $\rrho^{\frac{\gamma-1}{2}}(t)\in H^s$ and the other satisfying $\rrho^{\gamma-1}(t)\in H^s$ but $\rrho^{\frac{\gamma-1}{2}}(t)\notin H^s$. However, the above strong uniqueness theorem \ref{t:unithm1}  overcomes this difficulty.
\end{remark}

\section{continuation criteria of  Problem \ref{prob}}\label{s:contpri}
In this section, we introduce the continuation criteria of  Problem \ref{prob}. These continuation criteria theoretically allow us to continue the local solutions to a larger time interval provided the suitably bounded spatial derivatives, but they are not usually used to continue the solutions. In fact, the most crucial usage is if we know the solution blows up at a finite time, then the continuation criteria provide information on possible singularities. These continuation criteria can be viewed as improvements of those in \cite{Brauer1998}. In contrast to those in \cite{Brauer1998}, these \textit{localize the singularities} and answer the conjectures in \cite{Brauer1998} to some extent. However, we have to point out that the continuation criteria in this section only work for several conditions: $(1)$ the Makino type solutions with the $C^1$ domains (i.e., they are regular solutions but with the stronger regularities $\rho^{\frac{\gamma-1}{2}} \in C^1$); $(2)$ $1<\gamma\leq 5/3$.

First, we give a weak version (see \S\ref{s:ecp}) and it will help us extend the regular solutions to several types of  singularities. The basic idea of the proof is to use Makino's formulation of the Euler--Poisson system and energy estimates. However, the standard Gagliardo--Nirenberg--Moser estimates have been replaced by the revised Gagliardo--Nirenberg--Moser estimates on the Lipschitz bounded
domain.

The strong continuation criterion (see \S\ref{s:scp}) helps us reduce the types of interior singularities.  The key tool of the proof for the strong continuation criterion is Ferrari--Shirota--Yanagisawa inequality (a bounded domain  version of Beale--Kato--Majda estimate). This inequality allows us to control the velocity $\|v^i\|_{W^{1,\infty}}$ ``almost only'' by expansion and rotation. In order to use the  Ferrari--Shirota--Yanagisawa (FSY) inequality, there are two difficulties: the first one comes from the boundary regularity, and the second one is since the velocity at the boundary may not be orthogonal to the outer normal of  the boundary, that is, the \textit{second difficulty} is the orthogonal condition of the boundary limiting velocity $v^i\delta_{ij} n^j=0$ is not always true in this theorem.
Our idea to overcome these two difficulties and use the FSY inequality is: we first select a smooth surface (the $C^\infty$-approximation lemma of boundary \ref{t:Cinfapx} ensures this) near the boundary $\del{}\Omega(t)$, then decompose the velocity field into two parts around this $C^\infty$ surface, one of these parts is ``tangent'' to the $C^\infty$ surface on it and the other orthogonal to it, then analyze these two parts respectively.

\subsection{Weak continuation criterion of  Problem \ref{prob}}
\label{s:ecp}

We obtain the following weak continuation criterion using energy estimates on bounded $\Omega(t)$. 
Note since we will use Makino's formulations, the derivations of this section work only for Makino type  solutions  instead of the general regular solutions.
Before proceeding, let us prove the following helpful lemma, a bounded domain version of Lemma $2$ in \cite{Brauer1998}.
\begin{lemma}\label{e:estwv}
Suppose
\begin{equation}\label{e:cod}
	\int^t_0 \|\Theta (s)\|_{\Li(\Omega(s))}ds+\int^t_0  \|\nabla w(s)\|_{\Li(\Omega(s))} ds<\infty
\end{equation}	
for $t\in[0,T^\star)$. Then, if $T^\star<\infty$, for $t\in [0,T^\star)$, we have bounds
	\begin{align*}
			\| w(t)\|_{ \Li(\Omega(t))}  <\infty \AND
	 \|v^i(t)\|_{ L^\infty( \Omega(t))}   <\infty.
	\end{align*}
\end{lemma}
\begin{proof}
	Let us first use the continuation equation to prove $\| w(t)\|_{ \Li(\Omega(t))}  \leq C$.
		Reexpress \eqref{e:NEul1} and \eqref{e:eual1} in terms of Lagrangian description (recall \S \ref{s:lag}), we directly derive
	\begin{align*}
		\del{t}\ln \underline{\rrho}(t,\xi) = -\underline{\Theta}(t,\xi) \quad \text{and} \quad  \del{t}\ln \underline{ w}(t,\xi) =-\frac{\gamma-1}{2} \underline{\Theta}(t,\xi).
	\end{align*}
	Then integrating them along the integral curves of $\mathbf{v}$ yield
	\begin{align}
		\underline{\rrho}(t,\xi)=&\underline{\rrho}(0,\xi) \exp\Bigl(-\int^t_0\underline{\Theta}(s,\xi)ds\Bigr)\label{e:rrhotht1}
		\intertext{and} \underline{ w}(t,\xi)=&\underline{ w}(0,\xi) \exp\Bigl(-\frac{\gamma-1}{2}\int^t_0\underline{\Theta}(s,\xi)ds\Bigr). \label{e:rrhotht2}
	\end{align}
	Then we can estimate its $\Li$ norm,
	\begin{align}\label{e:aest}
		&\|\ln w(t)\|_{\Li(\Omega(t))}=\|\ln\underline{ w}(t)\|_{\Li(\Omega(0))}\leq 	 \|\ln w(0)\|_{\Li(\Omega(0))} + \frac{\gamma-1}{2}\int^t_0\|\underline{\Theta}(s)\|_{\Li(\Omega(0))}ds  \notag  \\
		& \hspace{1cm} \leq 	 \|\ln w(0)\|_{\Li(\Omega(0))} + \frac{\gamma-1}{2}\int^t_0\| \Theta(s)\|_{\Li(\Omega(s))}ds <C
<\infty.
	\end{align}
	for $t\in[0,T^\star)$.
	This, by \eqref{e:aest}, \eqref{e:cod} and the initial data, in turn, leads to
	\begin{equation*}\label{e:alest1}
		\| w(t)\|_{ \Li(\Omega(t))}\leq e^{\|\ln w(t)\|_{\Li(\Omega(t))}}   \leq C <\infty.
	\end{equation*}

Before proceeding, let us first prove the following useful inequalities which will be used frequently, for $s \geq 3$,
\begin{equation}\label{e:nphir1}
	\|\del{j} \Phi\|_{L^\infty(\Omega(t))} \leq C\|\rho\|_{L^\infty(\Omega(t))} \AND \|\del{j} \Phi\|_{H^s(\Omega(t))} \leq C\|\rho\|_{H^{s-1}(\Omega(t))}.
\end{equation}
To prove this, let us first note estimates from \cite[Lemma $1$]{Brauer1998} and  \cite[Lemma $2$]{Makino1986},
\begin{equation}\label{e:nphir0}
	\|\del{j} \Phi\|_{L^\infty(\Rbb^3)} \leq C\|\rho\|_{L^\infty(\Rbb^3)} \AND \|\del{j} \Phi\|_{H^s(\Rbb^3)} \leq C\|\rho\|_{H^{s-1}(\Rbb^3)}.
\end{equation}
Noting $\|\del{j} \Phi\|_{L^\infty(\Omega(t))}\leq \|\del{j} \Phi\|_{L^\infty(\Rbb^3)}$, $\|\del{j} \Phi\|_{H^s(\Omega(t))}\leq \|\del{j} \Phi\|_{H^s(\Rbb^3)}$ and $\|\rho\|_{L^\infty(\Rbb^3)}=\|\rho\|_{L^\infty(\Omega(t))}$, $\|\rho\|_{H^{s-1}(\Rbb^3)}=\|\rho\|_{H^{s-1}(\Omega(t))}$ (since $\rho|_{\Omega^\mathsf{c}(t)}=0$), \eqref{e:nphir0} yields \eqref{e:nphir}. 

Next, let us use the balance of momentum to estimate $\|v^i(t)\|_{ L^\infty( \Omega(t))}$. In order to simplify the expressions, let us define a function
\begin{equation*}
	g(s):=\int^s_0\|\Theta(\tau)\|_{\Li(\Omega(\tau))}d\tau
\end{equation*}
which, for later use, satisfies, for $s\in [0,T^\star)$,
\begin{equation}\label{e:gest}
	g(s) \leq C\int^{T^\star}_0\|V_{jk}(t)\|_{L^\infty( \Omega_\epsilon(t))}dt+
	\int^{T^\star}_0 \|\Theta(t)\|_{  L^\infty(\mathring{\Omega}_\epsilon(t))}dt.
\end{equation}
Reexpress the Euler equation  \eqref{e:eual2} (the balance of momentum) in terms of Lagrangian descriptions (recall \S \ref{s:lag}),
\begin{equation*}
	\del{t} \underline{v}^k(t,\xi)  +\frac{\gamma-1}{ 2}\underline{ w}(t,\xi) \delta^{ik}\underline{\del{i} w}(t,\xi)=  - \underline{\del{}^k \Phi}(t,\xi).
\end{equation*}
Integrating it along the integral curves of $\mathbf{v}$ (as the proof of \cite[Lemma $2$]{Brauer1998}), with the help of \eqref{e:nphir}, \eqref{e:rrhotht1}, \eqref{e:rrhotht2}, \eqref{e:aest} and notation \eqref{e:undf} in \S\ref{s:lag}, yields
\begin{align}\label{e:west1}
	&\|v^k(t)\|_{\Li(\Omega(t))} \notag \\
	\leq & \|v^k(0)\|_{\Li(\Omega(0))}+\frac{\gamma-1}{2}\int^t_0 \| w(s)\|_{\Li(\Omega(s))}\|\nabla w(s)\|_{\Li(\Omega(s))} ds+\int^t_0\|\nabla\rPhi(s)\|_{\Li(\Omega(s))}ds \notag  \\
	\leq & \|v^k(0)\|_{\Li(\Omega(0))}+\frac{\gamma-1}{2}\int^t_0 \| w(s)\|_{\Li(\Omega(s))}\|\nabla w(s)\|_{\Li(\Omega(s))} ds+C\int^t_0\|\rrho(s)\|_{\Li(\Omega(s))}ds \notag  \\
	\leq & \|v^k(0)\|_{\Li(\Omega(0))}+\frac{\gamma-1}{2} \| w(0)\|_{\Li(\Omega(0))} \int^t_0  \exp\Bigl(\frac{\gamma-1}{2}g(s)\Bigr)\|\nabla w(s)\|_{\Li(\Omega(s))} ds \notag  \\
	& +C\|\rrho(0)\|_{\Li(\Omega(0))} \int^t_0\exp\bigl( g(s)\bigr)ds\notag  \\
	\leq & \|v^k(0)\|_{\Li(\Omega(0))}+\frac{\gamma-1}{2} \| w(0)\|_{\Li(\Omega(0))} \exp\Bigl(\frac{\gamma-1}{2}g(T^\star)\Bigr)\int^t_0  \|\nabla w(s)\|_{\Li(\Omega(s))} ds \notag  \\
	& +C\|\rrho(0)\|_{\Li(\Omega(0))} T^\star \exp\bigl( g(T^\star)\bigr)  .
\end{align}
for $t\in [0,T^\star)$.
Then, by \eqref{e:cod}, we complete the proof.
\end{proof}

\begin{theorem}[The weak continuation criterion for Makino type solutions]\label{t:conpri1}
Suppose $1<\gamma \leq 5/3$ and let  $(\rrho,\rw^i,\rPhi,\Omega(t))$ be a \textbf{Makino type solution with a $C^1$ domain} (recalling Definition \ref{t:clsl2}) to  Problem \ref{prob} with the property $(w,\mathring{v}^i)\in C^0([0,T^\star),H^{s}(\Omega(t)))\cap C^1([0,T^\star),H^{s-1}(\Omega(t)))$, and $[0,T^\star)$ is the maximal  interval of the existence of  this solution. Then either
	\begin{enumerate}
		\item $T^\star=\infty$;
		\item or $T^\star<\infty$ and
			\begin{equation}\label{e:contpr1}
			\int^{T^\star}_0 \Bigl(\|\nabla \rw^i(\tau)\|_{L^\infty( \Omega(\tau))}  +\|\nabla w\|_{ L^\infty(\Omega(\tau))} \Bigr) d\tau=\infty.
		\end{equation}
	\end{enumerate} 
\end{theorem}
\begin{remark}
	We emphasize that this weak continuation criterion is for the \textit{Makino type solution with a $C^1$ domain} rather than Makino solutions or regular solutions. Therefore, for example,  at some time, if the $C^1$ domain becomes $C^0$ domain or if the velocity does not have a $C^1$ extension to the boundary, the Makino type solution with a $C^1$ domain breaks down.  
\end{remark}
\begin{proof}
	The proof of this theorem relies on an energy estimate by using Makino's formulations \eqref{e:eual1}--\eqref{e:eual2}.  The key to achieving this estimate is to use the  \textit{integration by parts} and the \textit{Reynold’s transport theorem \ref{t:RTT}} which helps us commute the time derivative and the integration with the changing domain.
	
	Firstly, according to Makino's formulations \eqref{e:eual1}--\eqref{e:eual2}, we reexpress the Euler--Poisson system \eqref{e:NEul1}--\eqref{e:NEul3} of  Problem \ref{prob} in $\Omega(t)$ and rewriting them into the matrix form yields, for $(t,\bx)\in[0,T)\times\Omega(t)$,
	\begin{align}\label{e:hypeq}
 \del{0}\mathcal{U}+A^i\del{i}\mathcal{U}=\mathcal{F}
	\end{align}
	where
	\begin{gather}      \mathcal{U}:=(w,\rw^k)^T,\quad \mathcal{F}:=(0,-\del{j}\rPhi)^T \label{e:F}
	\intertext{and}
	A^i=A^i(\mathcal{U}):=\p{\rw^i & \frac{\gamma-1}{2}w\delta^i_k \\
		\frac{\gamma-1}{2}w\delta^i_j & \rw^i\delta_{kj}}.\label{e:A}
	\end{gather}

	Secondly, let us derive a local energy estimate on $\Omega(t)$ by using the local calculus inequalities in Appendix \ref{s:calculus} (as the standard procedure, we, omit the details,  first derive the estimate by assuming all the variables are smooth and then using the approximation theorems of Sobolev spaces in \cite[\S $5.3$]{Evans2010} to conclude the derived inequality holds for Sobolev spaces as well). First acting on both sides of \eqref{e:hypeq} by $ \del{}^\beta $ (for $|\beta|\leq s$ and $s=3$ or $4$) yields
	\begin{align}\label{e:hodhyp}
 \del{0}\mathcal{U}^\beta+A^i\del{i}\mathcal{U}^\beta=\mathcal{G}
	\end{align}
	where
	\begin{align*}
	\mathcal{U}^\beta:=\del{}^\beta \mathcal{U}  \AND  \mathcal{G}:=- [\del{}^\beta, A^i]\del{i}\mathcal{U}+  \del{}^\beta  \mathcal{F} .
	\end{align*}
	Let us define the local energy,
	\begin{equation}\label{e:engy}
	\|\del{}^\beta\mathcal{U}\|_{L^2(\Omega(t))}^2=\la\mathcal{U}^\beta,  \mathcal{U}^\beta\ra:=\int_{\Omega(t)}(\mathcal{U}^\beta)^T   \mathcal{U}^\beta d^3\bx,
	\end{equation}
and calculate an important boundary term for later use. Note
\begin{align*}
	\rw^i \mathrm{Id}-A^i=\p{0 & -\frac{\gamma-1}{2}w \delta^i_k \\ -\frac{\gamma-1}{2}w \delta^i_j & 0 }.
\end{align*}
Hence, $(\mathcal{U}^\beta)^T(\rw^i  \mathrm{Id} -A^i )\mathcal{U}^\beta=-(\gamma-1)w\del{}^\beta \rw^i \del{}^\beta w$.
Then since $v^i$ is the extension of $\rw^i$ to the boundary $\del{}\Omega(t)$ (note the extension exists due to Definition \ref{t:clsl}.$(1)$ and Definition \ref{t:clsl2}, otherwise, if such extension does not exist, then the solution has already ceased to be a Makino solution with a $C^1$ domain), let us calculate, by the divergence theorem and $ w|_{\del{}\Omega(t)}\equiv 0$, the boundary term vanishes,
\begin{align}\label{e:a0-ai}
	&\int_{\Omega(t)}\del{i}\Bigl[(\mathcal{U}^\beta)^T(\rw^i  \mathrm{Id} -A^i )\mathcal{U}^\beta\Bigr] d^3\bx=-(\gamma-1)\int_{\Omega(t)}\del{i}\bigl( w\del{}^\beta \rw^i \del{}^\beta w\bigr) d^3\bx \notag  \\
	&=-(\gamma-1)\int_{\Omega(t)}\del{i}\bigl( w\del{}^\beta v^i \del{}^\beta w\bigr) d^3\bx=-(\gamma-1)\int_{\del{}\Omega(t)}   w\nu_i\del{}^\beta v^i \del{}^\beta w  d^3\bx=0
\end{align}
where $\nu_i:=\delta_{ij}\nu^j$ and $\nu^j$ is the unit outward-pointing normal to $\del{}\Omega(t)$.

	Then, noting the identity \eqref{e:a0-ai} at the boundary $\del{}\Omega(t)$, differentiating $\|\del{}^\beta\mathcal{U}\|_{L^2(\Omega(t))}^2$ with respect to $t$, inserting \eqref{e:hodhyp} in the followings and using integration by parts and Reynold’s transport theorem \ref{t:RTT} leads to, for $t\in[0,T^\star)$,
	\begin{align}\label{e:Eg1}
	\del{t}\|\del{}^\beta\mathcal{U}\|_{L^2(\Omega(t))}^2\overset{\eqref{e:engy}}{=} & \del{t}\int_{\Omega(t)}(\mathcal{U}^\beta)^T   \mathcal{U}^\beta d^3\bx \notag  \\ \overset{\text{Reynold's transport theorem}}{=} & \int_{\Omega(t)}\del{i}\bigl[(\mathcal{U}^\beta)^T   \mathcal{U}^\beta \rw^i\bigr] d^3 \bx+2\la\mathcal{U}^\beta,   \del{t}\mathcal{U}^\beta\ra  \notag  \\
	\overset{\eqref{e:hodhyp}}{=} & \int_{\Omega(t)}\del{i}\bigl[(\mathcal{U}^\beta)^T   \mathcal{U}^\beta \rw^i\bigr] d^3 \bx  -2\la\mathcal{U}^\beta, A^i\del{i}\mathcal{U}^\beta \ra +2\la\mathcal{U}^\beta, \mathcal{G}\ra   \notag  \\
\overset{\text{the integration by parts}}{=} & \int_{\Omega(t)}\del{i}\Bigl[(\mathcal{U}^\beta)^T(\rw^i  \mathrm{Id} -A^i )\mathcal{U}^\beta\Bigr] d^3\bx+ \la\mathcal{U}^\beta, (\del{i}A^i) \mathcal{U}^\beta\ra   +2\la\mathcal{U}^\beta, \mathcal{G}\ra   \notag  \\
\overset{\eqref{e:a0-ai}}{=} &  \la\mathcal{U}^\beta, (\del{i}A^i) \mathcal{U}^\beta\ra   +2\la\mathcal{U}^\beta, \mathcal{G}\ra .
	\end{align}

Before proceeding, by noting $1<\gamma \leq 5/3$ (i.e., $\rho$ is $C^3$ in $w$), with the help of \eqref{e:nphir1} and Proposition \ref{t:compoest}, we obtain
\begin{equation}\label{e:nphir}
	\|\del{j} \Phi\|_{H^s(\Omega(t))} \leq C\|\rho\|_{H^{s-1}(\Omega(t))}\leq C(\|D_w\rho\|_{C^{s-2}})(1+\|w\|^{s-2}_{L^\infty(\Omega(t))})\|w \|_{H^{s-1}(\Omega(t))}.
\end{equation}

	Standard procedures of the hyperbolic system (see, for example, \cite{Majda2012,Taylor2010}, etc.), with the help of the revised Gagliardo–Nirenberg–Moser estimates on the Lipschitz bounded domain, i.e., Proposition \ref{t:prodn}, \ref{t:commuest} and \ref{t:compoest} and Corollary \ref{t:prodn2} (see Appendix \ref{s:calculus}), and taking \eqref{e:F} and \eqref{e:A} into account, leads to
	\begin{align}
	&(1)~ \la\mathcal{U}^\beta, (\del{i}A^i) \mathcal{U}^\beta\ra  \leq     C\|\del{i}A^i\|_{\Li(\Omega(t))}\|\mathcal{U}\|^2_{H^s(\Omega(t))} \leq C\|\nabla\mathcal{U}\|_{\Li(\Omega(t))}\|\mathcal{U}\|^2_{H^s(\Omega(t))}, \label{e:Eg2} \\
	&(2)~ \la\mathcal{U}^\beta, - [\del{}^\beta, A^i]\del{i}\mathcal{U} \ra \notag  \\
	&\hspace{1cm} \leq  C\|\mathcal{U}\|_{H^s(\Omega(t))}\bigl(\|A^i\|_{H^s}\|\del{i}\mathcal{U}\|_{\Li(\Omega(t))}+\|\del{j}A^i\|_{\Li(\Omega(t))}\|\mathcal{U}\|_{H^{s}(\Omega(t))}\bigr) \notag  \\
	&\hspace{1cm} \leq  C(\|\mathcal{U}\|_{\Li(\Omega(t))})\|\nabla\mathcal{U}\|_{\Li(\Omega(t))} \|\mathcal{U}\|^2_{H^s(\Omega(t))}   \label{e:Eg3}   \\
	\intertext{\textit{and noting $1<\gamma \leq 5/3$} (i.e., $\rho$ is $C^3$ in $w$) with the help of \eqref{e:nphir}, }
	&(3)~ \la\mathcal{U}^\beta,  \del{}^\beta  \mathcal{F} \ra
	\leq    \|\mathcal{U}^\beta\|_{L^2(\Omega(t))} \| \del{j} \Phi\|_{H^s(\Omega(t))}\notag  \\
	&\hspace{1cm} \leq  C(\|D_w\rho\|_{C^{s-2}})(1+\|w\|^{s-2}_{L^\infty(\Omega(t))})  \|\mathcal{U}\|_{H^s(\Omega(t))} \|w \|_{H^{s-1}(\Omega(t))}\leq C(\| w\|_{\Li(\Omega(t))}) \|\mathcal{U}\|_{H^s(\Omega(t))}^2. \label{e:Eg4}
	\end{align}
	where we simply denote $C(  \|\mathcal{U}\|_{\Li(\Omega(t))}):= C(\|D_\mathcal{U} A^i\|_{C^{s-1}})(1+\| \mathcal{U}\|^{s-1}_{\Li(\Omega(t))})$ and $ C(\| w\|_{\Li(\Omega(t))}):=C(\|D_w\rho\|_{C^{s-2}})(1+\|w\|^{s-2}_{L^\infty(\Omega(t))}) $ to emphasize the constant depending on $  \|\mathcal{U}\|_{\Li(\Omega(t))}$ and $  \|\rho\|_{\Li(\Omega(t))}$, respectively.
	Adding up \eqref{e:Eg1} for all $| w|\leq s$, with the help of \eqref{e:Eg2}--\eqref{e:Eg4}, leads to, for $t\in[0,T^\star)$,
	\begin{equation*}
	\del{t}\|\mathcal{U}\|^2_{H^s(\Omega(t))}\leq C \bigl[C(\|\mathcal{U}\|_{\Li(\Omega(t))})\|\nabla\mathcal{U}\|_{\Li(\Omega(t))}+C(\| w\|_{\Li(\Omega(t))}) \bigr] \|\mathcal{U}\|^2_{H^s(\Omega(t))}  .
	\end{equation*}
	Then Gronwall’s inequality implies for all $t\in[0,T^\star)$,
	\begin{align}\label{e:aprest0}
	\|\mathcal{U}(t)\|_{H^s(\Omega(t))}\leq
	\|\mathcal{U}(0)\|_{H^s(\Omega(0))} \exp \int^t_0\bigl[C(\|\mathcal{U}\|_{\Li(\Omega(t))})\|\nabla\mathcal{U}(s)\|_{\Li(\Omega(s))}+C(\| w(s)\|_{\Li(\Omega(s))}) \bigr] ds   .
	\end{align}

	Then either $T^\star=\infty$ or if $T^\star<\infty$ let us prove \eqref{e:contpr1} by contradictions, i.e., otherwise, we assume
		\begin{equation*}
		\int^{T^\star}_0 \Bigl(\|\nabla \rw^i(\tau)\|_{L^\infty( \Omega(\tau))}  +\|\nabla w\|_{ L^\infty(\Omega(\tau))} \Bigr) d\tau<\infty.
	\end{equation*}
	Using \eqref{e:aprest0} and Lemma \ref{e:estwv}, with the help of \eqref{e:WThOm0a} and \eqref{e:theta},
	 we conclude that for all $t\in[0,T^\star)$, there is an estimate $
	\|\mathcal{U}(t)\|_{H^s(\Omega(t))}<\infty$. 
	It means we can continue the solution to a larger interval $[0,T^\prime)$ for $T^\prime>T^\star$ by the standard continuation criterion (for example, using the similar standard arguments of the proof of \cite[Theorem $2.2$]{Majda2012}). In specific, since $
	\|\mathcal{U}(t)\|_{H^s(\Omega(t))}<\infty$ for all $t\in[0,T^\star)$, we can use it to obtain $
	\|\mathcal{U}|_{t=T^\star-\epsilon}\|_{H^s(\Omega(T^\star-\epsilon))}<\infty$ as the initial data at time $T^\star-\epsilon$ for any $\epsilon>0$ to continue this solution beyond $T^\star$ to $T^\prime$, and this continuation relies on the local existence theorem \ref{t:Newext2} (or one can use the similar argument by Makino's extensions and restrictions of the proof of Theorem \ref{t:Newext2} with the help of Theorem \ref{t:C1bdry} to continue the solution).  
	It contradicts the fact that $T^\star$ is the maximal time of existence. We then complete the proof.
\end{proof}

\subsection{The strong continuation criterion of  Problem \ref{prob}}\label{s:scp}

Inspired by the ideas of \cite{Brauer1998}, we improve the above continuation criterion to decrease the types of singularities.
Firstly, Lemma \ref{t:FSY2} (resemble \cite[Corollary $1$]{Ferrari1993}) is a useful tool for proving the strong version of the continuation criterion and can be derived directly from the \textit{Ferrari--Shirota--Yanagisawa inequality} of Theorem \ref{t:FSY} (see Appendix \ref{s:FEY}).

\begin{lemma}\label{t:FSY2}
	Suppose for every $t\in[0,T)$, $\Omega(t)\subset \Rbb^3$ is a bounded, simply connected domain of class $C^{2,\beta}$, $\beta\in(0,1)$. If $\mathbf{v}(t,\cdot):=(v^i(t,\cdot))\in H^3(\Omega(t),\Rbb^3)$ is a vector field and $v^i\delta_{ij} n^j=0$ on $\del{}\Omega(t)$ and $n^j$ is the unit outward normal at $\bx\in\del{}\Omega(t)$, the rotation $\Omega_{jk} $ and expansion $\Theta $ are defined by \eqref{e:WThOm0b} and \eqref{e:theta}, then for $t\in[0,T)$,
	\begin{equation*}
	\|v^i(t)\|_{W^{1,\infty}(\Omega(t))} \leq C\bigl[\bigl(1+\ln^{+}\|v^i(t)\|_{H^3(\Omega(t))} \bigr)(\|\Theta(t)\|_{\Li(\Omega(t))}+\|\Omega_{jk}(t)\|_{\Li(\Omega(t))})+1\bigr]
	\end{equation*}
	where
	\begin{align*}
	\ln^{+} a:=\begin{cases}
	\ln a, \quad \text{if} \quad a\geq 1\\
	0, \quad \text{otherwise}
	\end{cases}.
	\end{align*}
\end{lemma}
\begin{proof}
	Note the relations between rotation tensor $\Omega_{ij}$ and the vorticity $\omega:=(\omega_1,\omega_2,\omega_3)=\nabla\times \mathbf{v}=(\del{2}v^3-\del{3}v^2,\del{1}v^3-\del{3}v^1,\del{1}v^2-\del{2}v^1)$, i.e.,
	\begin{align*}
	(\Omega_{ij})=\frac{1}{2} \p{0 & -\omega_3 & -\omega_2\\
		\omega_3 & 0 & -\omega_1  \\
		\omega_2 & \omega_1 & 0},
	\end{align*}
	and $\theta:=\nabla\cdot \mathbf{v}=\Theta$.
	Let $v:=(\mathbf{v}, 0)$ ($v$ is defined as Theorem \ref{t:FSY}, Appendix \ref{s:FEY}), then the requirements of Theorem \ref{t:FSY} are verified and thus we obtain this lemma.
\end{proof}

The following theorem \ref{t:conpri2} gives the continuation criterion a more meticulous portrayal, providing a better blowup result. Recalling the notations of sets in \S\ref{s:sets} and the definition \eqref{e:omgep} of $\Omega_\epsilon$, let us state this theorem.

\begin{theorem}[The strong continuation criterion for Makino type solutions]\label{t:conpri2}
	Suppose $1<\gamma \leq 5/3$, $\Omega(0)\subset \Rbb^3$ is a bounded, precompact and simply connected domain, and let  $(\rrho,\rw^i,\rPhi,\Omega(t))$ be a Makino type solution with a $C^1$ domain (recalling Definition \ref{t:clsl2}) to  Problem \ref{prob} with the property $(w,\mathring{v}^i)\in C^0([0,T^\star),H^{s}(\Omega(t)))\cap C^1([0,T^\star),H^{s-1}(\Omega(t)))$, and $[0,T^\star)$ is the maximal  interval of the existence of  this solution. Then either
	\begin{enumerate}
		\item $T^\star=\infty$;
		\item or $T^\star<\infty$ and  there is a small constant $\epsilon>0$ such that
	\begin{align}\label{e:contpr2}
	&\int^{T^\star}_0 \Bigl(\|V_{jk}(s)\|_{L^\infty( \Omega_\epsilon(s))}+\|\Theta(s)\|_{L^\infty( \mathring{\Omega}_\epsilon(s) )}\notag  \\
	&\hspace{1.5cm} +\|\Omega_{jk}(s)\|_{ L^\infty( \mathring{\Omega}_\epsilon(s) )}+\|\nabla  w(s)\|_{ L^\infty(\Omega(s))} \Bigr) ds=\infty.
\end{align}
where $\mathring{\Omega}_\epsilon(s):=\Omega(s)\setminus \Omega_\epsilon(s)$.
	\end{enumerate}
\end{theorem}
\begin{remark}
	We point out that  the generalization of Ferrari--Shirota--Yanagisawa inequality  in \cite{Ogawa2003,Taniuchi2003} can further improve this strong continuation criterion. 	
\end{remark}
\begin{remark}
On the one hand, we point out this strong continuation criterion improves the one in \cite{Brauer1998} to make sure it can localize the singularities to the compact support of the density $\Omega(t)$. This localizing type of the continuation criterion has been conjectured in \cite[\S VII]{Brauer1998}. However, our strong continuation criterion, on the other hand, still has some \textit{defects} which are not as good as the conjecture in \cite[\S VII]{Brauer1998}. The defects mainly come from the near boundary behaviors. Although it localizes the singularities, but near the boundary, we can not rule out the shear $\int^{T^\star}_0\|\Xi_{jk}(s)\|_{L^\infty( \Omega_\epsilon(s))}ds=\infty$ where $\Xi_{jk}:=\Theta_{jk}-\frac{1}{3}\Theta \delta_{jk}$. In addition, this continuation criterion  requires $1<\gamma\leq 5/3$. For the general regular solutions, we do not consider them in this article and leave them for the near future.
	Thus, this theorem partially answers the conjectures in \cite[\S VII]{Brauer1998} to some extent.
\end{remark}
\begin{proof}
	Let us prove, under these assumptions, if $T^\star<\infty$, then there exists a small constant $\epsilon>0$,  \eqref{e:contpr2} holds. To do so, we use proof by contradictions. Let us assume for any small constant $\epsilon>0$,
		\begin{align}\label{e:contpr3}
		&\int^{T^\star}_0 \Bigl(\|V_{jk}(s)\|_{L^\infty( \Omega_\epsilon(s))}+\|\Theta(s)\|_{L^\infty( \mathring{\Omega}_\epsilon(s) )} +\|\Omega_{jk}(s)\|_{ L^\infty( \mathring{\Omega}_\epsilon(s) )}+\|\nabla  w(s)\|_{ L^\infty(\Omega(s))} \Bigr) ds<\infty.
	\end{align}
From this inequality, using Lemma \ref{e:estwv}, we first have bounds
\begin{align}\label{e:bdwv0}
	\| w(t)\|_{ \Li(\Omega(t))}  <\infty \AND
	\|v^i(t)\|_{ L^\infty( \Omega(t))}   <\infty.
\end{align}
	
	According to the estimate \eqref{e:aprest0} in the proof of Theorem \ref{t:conpri1}, a direct idea of this proof is to verify that $ \|V_{jk}(t)\|_{  L^\infty( \Omega(t))} $ is controlled by
		\begin{equation}\label{e:int1a}
		 \|\nabla  w(t)\|_{  L^\infty(\Omega(t))} , \quad \|V_{jk}(t)\|_{L^\infty( \Omega_\epsilon(t))}, \quad   \|\Theta(t)\|_{  L^\infty(\mathring{\Omega}_\epsilon(t))} , \AND \|\Omega_{jk}(t)\|_{ L^\infty(\mathring{\Omega}_\epsilon(t))}    
		\end{equation}
	for small $\epsilon>0$. However, in fact, besides \eqref{e:int1a}, Lemma \ref{t:FSY2} implies the estimate of $\|V_{jk}(t)\|_{L^\infty(\Omega(t))}$ also slightly depends on the $H^s$ norm of the unknowns $\|\mathcal{U}\|_{H^s(\Omega(t))}$. An approach from \cite{Beale1984} (it is widely used, for example, in \cite{Brauer1998,Ferrari1993}, etc.) indicates Lemma \ref{t:FSY2} along with the energy estimate \eqref{e:aprest0} and the Gronwall inequality (see Theorem \ref{e:grnwl}) will help us derive an improved a priori estimate on $\|\mathcal{U}\|_{H^s(\Omega(t))}$.  

Let us control $\int^{T^\star}_0 \|V_{jk}(t)\|_{  L^\infty( \Omega(t))} dt$.
	The idea of this proof is under the situation of Theorem \ref{t:conpri1}, we attempt to use Lemma \ref{t:FSY2} to control $\|v^i\|_{W^{1,\infty}(\Omega(t))}$ in terms of $\|\Theta\|_{\Li(\Omega(t))}$ and $\|\Omega_{jk}\|_{\Li(\Omega(t))}$. However, there are \textit{two difficulties} obstructing this idea.
	
	In order to use Lemma \ref{t:FSY2}, recalling that Lemma \ref{t:FSY2} requires $\Omega(t)$ is a bounded, simply connected domain of class $C^{2,\beta}$ for $\beta \in(0,1)$, the \textit{first difficulty} is the boundary in this theorem is of $C^2$ class instead of $C^{2,\beta}$ as Lemma \ref{t:FSY2} requires. This \textit{lower regularity of the boundary} leads to the first difficulty; On the other hand,  recalling that Lemma \ref{t:FSY2} requires $\mathbf{v}(t,\cdot):=(v^i(t,\cdot))\in H^3(\Omega(t),\Rbb^3)$ is a vector field and $v^i\delta_{ij} n^j=0$ on $\del{}\Omega(t)$ and $n^j$ is the unit outward normal at $\bx\in\del{}\Omega(t)$, but in this theorem, the velocity $v^i$ may be not orthogonal to the unit outward normal at $\bx\in\del{}\Omega(t)$. Thus the \textit{second difficulty} is the orthogonal condition of the boundary limiting velocity $v^i\delta_{ij} n^j=0$ is not always true in this theorem. It is clear both of these difficulties come from the  boundary.
	
Our idea to overcome these two difficulties and use the FSY inequality is: we first select a smooth surface (the $C^\infty$-approximation lemma of boundary \ref{t:Cinfapx} ensures this) near the boundary $\del{}\Omega(t)$, then decompose the velocity field into two parts around this $C^\infty$ surface, one of these parts is ``tangent'' to the $C^\infty$ surface on it and the other orthogonal to it, then analyze these two parts respectively. This decomposition of the velocity field will overcome the second difficulty.

	First using the $C^\infty$-approximation lemma of the boundary (see Lemma \ref{t:Cinfapx} in  Appendix \S \ref{s:sard}), for every time $t\in [0,T^\star)$ and for the constant $\epsilon>0$, there is an open domain $\mathcal{D}(t)$ with $C^\infty$ boundary, such that $\mathring{\Omega}_\epsilon(t) \subset \mathcal{D}(t)\subset \Omega(t)$ (see the left figure in Figure \ref{fig:1}). Then by the extension theorems of the outward unit normal (see, for example, Proposition \ref{t:normal1} or Theorem \ref{t:normal2}, see Appendix \ref{s:normal}), since $\mathcal{D}(t)$ is a $C^\infty$ domain, then there exists an open  neighborhood $\mathcal{N}(t)$ of the boundary $ \del{}\mathcal{D}(t)$ such that $\partial\mathcal{D}\subset \mathcal{N}(t)\subset \Omega_\epsilon(t)$, and exists a unit vector field $\mathbf{n}:\mathcal{N}(t) \rightarrow \Rbb^3$ of class $C^\infty$ in $\mathcal{N}(t)$ with the property that $\mathbf{n}|_{\del{}\mathcal{D}(t)}$ is the outward unit normal to $\del{}\mathcal{D}(t)$, then we further extend $\mathbf{n}$ from $\mathcal{N}(t)$ to $\overline{\Omega(t)}$ by assuming $\mathcal{O}(t)$ is a larger neighborhood of $\del{}\mathcal{D}(t)$ satisfying  $\del{}\mathcal{D}(t)\subset\mathcal{N}(t)$ and $\overline{\mathcal{N}(t)}\subset \mathcal{O}(t)$ (see the right figure in Figure \ref{fig:1}) and using Proposition \ref{t:normal3} (see Appendix \ref{s:normal}, we denote the extended vector field still by $\mathbf{n}$ and $\supp\mathbf{n}\subset \mathcal{O}(t)$).
	\begin{figure*}
		\includegraphics[width= 0.9\linewidth]{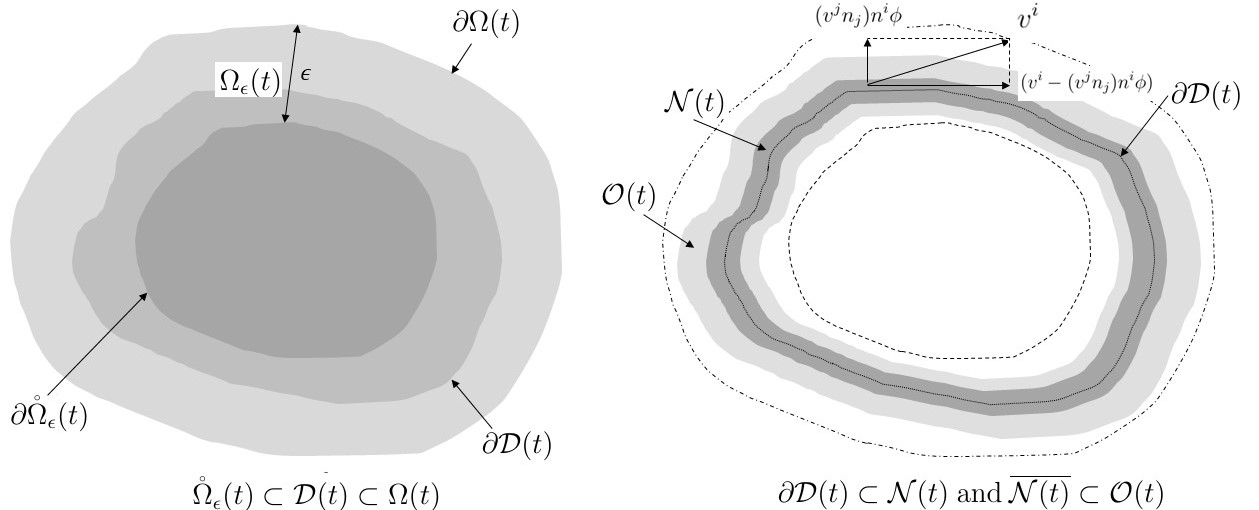}
		\caption{Areas: The left figure indicates the three domains $\Omega(t)$, $\ring{\Omega}_\epsilon(t)$, $\mathcal{D}(t)$ and the annular region $\Omega_\epsilon(t)$. The right one  shows the neighborhoods, $\mathcal{N}(t)$ and $\mathcal{O}(t)$, of $\del{}\mathcal{D}(t)$ (the dotted and dash lines indicate the corresponding boundaries of $\Omega(t)$, $\ring{\Omega}_\epsilon(t)$, $\mathcal{D}(t)$).  }
		\label{fig:1}
	\end{figure*}

	Then, there is a cut-off function $\phi\in C^\infty_0$ (see, for example, \cite[Proposition 2.25]{Lee2012}) such that for every $(t,\bx)\in[0,T^\star)\times \Rbb^3$,
	\begin{align}\label{e:phidef}
	\phi(t,\bx)\begin{cases}
	=0,  \quad &\text{for} \quad \bx\in
	\mathcal{O}^\mathbf{c}(t)  \\
	=1, \quad &\text{for} \quad \bx\in   \overline{\mathcal{N}(t)} \\
	\in(0,1), \quad &\text{for other} \quad \bx\in \Rbb^3
	\end{cases}.
	\end{align}
	Then $v^i$ can be decomposed as two fields
	\begin{equation}\label{e:wdecp}
	v^i= (v^jn_j)n^i \phi+(v^i-(v^jn_j)n^i \phi).
	\end{equation}
	for $(t,\bx)\in[0,T^\star)\times \overline{\Omega(t)}$.
	It is clear, by the definition \eqref{e:phidef} of $\phi$, that $(v^i-(v^jn_j)n^i \phi)|_{\mathcal{O}^\mathbf{c} }=v^i|_{\mathcal{O}^\mathbf{c} }$ and $(v^jn_j)n^i \phi|_{\overline{\mathcal{N}}}=(v^jn_j)n^i|_{\overline{\mathcal{N}}}$, $(v^jn_j)n^i \phi|_{\mathcal{O}^\mathsf{c}}=0$. We will see the vector field $v^i-(v^jn_j)n^i \phi$ and the smooth domain $\mathcal{D}(t)$ satisfy the requirements of Lemma \ref{t:FSY2}. Let us verify these conditions of Lemma \ref{t:FSY2} as follows.

	We can define the rotation and expansion of the vector field $v^i-(v^jn_j)n^i \phi$ as \eqref{e:WThOm0b} and \eqref{e:theta}, and denote them by $\widehat{\Omega}_{jk}$ and $\widehat{\Theta}$, respectively. Then in order to apply Lemma \ref{t:FSY2}, for every $t\in[0,T^\star)$ we verify, with the help of Corollary \ref{t:prodn2} and Sobolev embeddings $\|\cdot\|_{\Li} \lesssim \|\cdot\|_{H^s}$ for $s\geq 2$,
	\begin{align}
	&\|v^i-(v^jn_j)n^i \phi\|_{H^s(\mathcal{D}(t))} \leq \|v^i-(v^jn_j)n^i \phi\|_{H^s(\Omega(t))}  \notag  \\
	&\hspace{2cm} \leq  \|v^i\|_{H^s(\Omega(t))}+\|(v^jn_j)n^i \phi\|_{H^s(\Omega(t))} \notag  \\
	&\hspace{2cm}  \leq  \|v^i\|_{H^s(\Omega(t))}+\|v^jn_j\|_{H^s(\Omega(t))} \|n^i\|_{H^s(\Omega(t))} \|\phi\|_{H^s(\Omega(t))}  \leq C\|v^i\|_{H^s(\Omega(t))}. \label{e:west}
	\end{align}
where we have noted  $\|v^jn_j\|_{H^s(\Omega(t))} \leq C\|v^j\|_{H^s(\Omega(t))}$, $\|n^i\|_{H^s(\Omega(t))}\leq C$ and $\|\phi\|_{H^s(\Omega(t))} \leq C$ by their definitions.
	Since $\mathcal{D}(t)$ is a bounded, simply connected $C^\infty$ domain (this is because $\Omega(t)$ is  bounded, and $\Omega(t)$ is simply connected since $\Omega(0)$ is simply connected and the flow generated by $v^i$ is at least $C^1$), and direct examinations imply $(v^i-(v^jn_j)n^i \phi)n_i|_{\del{}\mathcal{D}(t)}= 0$. Hence, $\mathcal{D}(t)$ and the vector field $v^i-(v^jn_j)n^i \phi$ satisfy the conditions of Lemma \ref{t:FSY2}.

	Then Lemma \ref{t:FSY2} yields, for every $t\in[0,T^\star)$,
	\begin{align}\label{e:wtag}
	&\|v^i-(v^jn_j)n^i \phi\|_{W^{1,\infty}(\mathcal{D}(t))} \notag  \\
	& \leq C\bigl[\bigl(1+\ln^{+}\|v^i-(v^jn_j)n^i \phi\|_{H^s(\mathcal{D}(t))} \bigr)(\|\widehat{\Theta}\|_{\Li(\mathcal{D}(t))}+\|\widehat{\Omega}_{jk}\|_{\Li(\mathcal{D}(t))})+1\bigr].
	\end{align}
	
Since, by \eqref{e:phidef},  $(v^i-(v^jn_j)n^i \phi)|_{\mathcal{O}^\mathbf{c} }=v^i|_{\mathcal{O}^\mathbf{c} }$, we have $(v^i-(v^jn_j)n^i \phi)|_{\mathring{\Omega}_\epsilon(t) }=v^i|_{\mathring{\Omega}_\epsilon(t) }$ which in turn leads to
	\begin{align}\label{e:oth1}
	\widehat{\Omega}_{jk}|_{\mathring{\Omega}_\epsilon(t)}=\Omega_{jk}|_{\mathring{\Omega}_\epsilon(t)} \AND \widehat{\Theta}|_{\mathring{\Omega}_\epsilon(t)}=\Theta|_{\mathring{\Omega}_\epsilon(t)}.
	\end{align}
	Furthermore, since $\phi|_{\mathring{\Omega}_\epsilon(t)}=0$ (by the definition \eqref{e:phidef} of $\phi$, i.e., $\phi(t,\bx)=0$ for $\bx\in \mathcal{O}^\mathbf{c}(t)\supset \mathring{\Omega}_\epsilon(t)$), we note $(v^jn_j)n^i \phi=0$ for  $\bx\in\mathring{\Omega}_\epsilon(t)$, then this implies its norm
 \begin{equation}\label{e:v0}
 	\|(v^jn_j)n^i \phi\|_{W^{1,\infty}(\mathring{\Omega}_\epsilon(t))}=0.
 \end{equation}
 We also note the fact
	\begin{align*}
	\|v^i\|_{W^{1,\infty}(\Omega(t))}\leq& \|v^i\|_{W^{1,\infty}(\Omega_\epsilon(t))}+\|v^i\|_{W^{1,\infty}(\mathring{\Omega}_\epsilon(t))},
	\end{align*}
and
note the estimate $1+\ln^{+}\|v^i-(v^jn_j)n^i \phi\|_{H^s(\mathcal{D}(t))}=\ln e+\ln^{+}\|v^i-(v^jn_j)n^i \phi\|_{H^s(\mathcal{D}(t))}  \leq C\ln^{+}(e+C\|v^i\|_{H^s(\Omega(t))})$, and with the help of \eqref{e:oth1}, we obtain
\begin{align*}
	\|\widehat{\Theta}\|_{\Li(\mathcal{D}(t))}+\|\widehat{\Omega}_{jk}\|_{\Li(\mathcal{D}(t))} \leq &
	\|\widehat{\Theta}\|_{\Li(\Omega(t))}+\|\widehat{\Omega}_{jk}\|_{\Li(\Omega(t))} \notag  \\
	\leq & \|\Theta\|_{\Li(\mathring{\Omega}_\epsilon(t))}+\|\Omega_{jk}\|_{\Li(\mathring{\Omega}_\epsilon(t))}   + 	\|\widehat{\Theta}\|_{\Li(\Omega_\epsilon(t))}+\|\widehat{\Omega}_{jk}\|_{\Li(\Omega_\epsilon(t))} \notag  \\
	 \leq & \|\Theta\|_{\Li(\mathring{\Omega}_\epsilon(t))}+\|\Omega_{jk}\|_{\Li(\mathring{\Omega}_\epsilon(t))} +C\|v^i\|_{W^{1,\infty}(\Omega_\epsilon(t))},
\end{align*}
where the last term $C\|v^i\|_{W^{1,\infty}(\Omega_\epsilon(t))}$ takes care the near boundary estimates.
Then using above estimates,  \eqref{e:west1} (in Step $1$), \eqref{e:wdecp}, \eqref{e:west}, \eqref{e:wtag}, \eqref{e:v0} and $\mathring{\Omega}_\epsilon(t) \subset \mathcal{D}(t)\subset \Omega(t)$, we obtain
	\begin{align}\label{e:nw1}
	&\|V_{kj}(t)\|_{\Li(\Omega(t))}  \leq  \|v^i(t)\|_{W^{1,\infty}(\Omega(t))}  \notag  \\
	\leq& \|v^i\|_{W^{1,\infty}(\Omega_\epsilon(t))}+\|(v^jn_j)n^i \phi\|_{W^{1,\infty}(\mathring{\Omega}_\epsilon(t))} +\|v^i-(v^jn_j)n^i \phi\|_{W^{1,\infty}(\mathring{\Omega}_\epsilon(t))} \notag  \\
	\leq& \|v^i\|_{W^{1,\infty}(\Omega_\epsilon(t))}+ C\bigl[\bigl(1+\ln^{+}\|v^i-(v^jn_j)n^i \phi\|_{H^s(\mathcal{D}(t))} \bigr)\notag  \\
	&\times(\|\widehat{\Theta}\|_{\Li(\mathcal{D}(t))}+\|\widehat{\Omega}_{jk}\|_{\Li(\mathcal{D}(t))})+1\bigr]\notag  \\
	\leq& \|v^i\|_{W^{1,\infty}(\Omega_\epsilon(t))}+ C\bigl[ \bigl( \ln^{+}(e+C\|v^i\|_{H^s(\Omega(t))}) \bigr) (\|\Theta\|_{\Li(\mathring{\Omega}_\epsilon(t))} \notag  \\
	&+\|\Omega_{jk}\|_{\Li(\mathring{\Omega}_\epsilon(t))}+C\|v^i\|_{W^{1,\infty}(\Omega_\epsilon(t))})+1\bigr]  \notag  \\
	\leq& \|v^i\|_{W^{1,\infty}(\Omega_\epsilon(t))}+ C\bigl[ \bigl(\ln^{+}(e+C\|\mathcal{U}\|_{H^s(\Omega(t))}) \bigr) (\|\Theta\|_{\Li(\mathring{\Omega}_\epsilon(t))} \notag  \\
	&+\|\Omega_{jk}\|_{\Li(\mathring{\Omega}_\epsilon(t))}+C\|v^i\|_{W^{1,\infty}(\Omega_\epsilon(t))})+1\bigr]
	\end{align}	
for every $t\in[0,T^\star)$ where
\begin{equation*}
	\|v^i\|_{W^{1,\infty}(\Omega_\epsilon(t))} \leq \|v^i \|_{\Li(\Omega(t))}
	+\|V_{kj}\|_{\Li(\Omega_\epsilon(t))}
\end{equation*}
and $\|v^i \|_{\Li(\Omega(t))}$ can be estimated by \eqref{e:west1}.  Note there is an unexpected term $\|\mathcal{U}\|_{H^s(\Omega(t))}$ in this estimate \eqref{e:nw1}. This will be coped with in the later steps.

In the end, let us estimate $\|\mathcal{U}\|_{H^s(\Omega(t))}$. We use the following procedure (this procedure has been used, for example, in \cite{Beale1984,Brauer1998,Ferrari1993}) to derive an estimate of $\|\mathcal{U}\|_{H^s(\Omega(t))}$ in terms of \eqref{e:int1a}. That is, using \eqref{e:aprest0}, \eqref{e:bdwv0} and \eqref{e:nw1}, we arrive at	\begin{align}\label{e:aprest3}
	\|\mathcal{U}(t)\|_{H^s(\Omega(t))}
	\leq &
	\|\mathcal{U}(0)\|_{H^s(\Omega(0))}  \exp\Bigl[C\int^t_0\bigl[\|\nabla  w(s)\|_{\Li(\Omega(s))} \notag  \\
	&+\|  V_{ij}(s)\|_{\Li(\Omega(s))}+C(\| w(s)\|_{\Li(\Omega(s))})   \bigr] ds \Bigr] \notag  \\
	\leq &
	\|\mathcal{U}(0)\|_{H^s(\Omega(0))} \exp\Bigl\{C\int^t_0\bigl[b(s) + C(a (s) y(s) +1)  \bigr] ds \Bigr\}
	\end{align}
for every $t\in[0,T^\star)$
	where
	\begin{align*}
	y(t):=&\ln^{+}(C\|\mathcal{U} \|_{ H^s(\Omega(t))}+e)=\ln(C\|\mathcal{U} \|_{ H^s(\Omega(t))}+e), \\
	a(t):=&\|\Theta\|_{ \Li(\mathring{\Omega}_\epsilon(t))}
	+\|\Omega_{jk}\|_{ \Li(\mathring{\Omega}_\epsilon(t))}
	+C\|v^i\|_{W^{1,\infty}(\Omega_\epsilon(t))} ,
	\intertext{and}
	b(t):=&\|\nabla w\|_{\Li(\Omega(t))}+\|v^i\|_{W^{1,\infty}(\Omega_\epsilon(t))}+C(\| w(t)\|_{\Li(\Omega(t))}) .
	\end{align*}
	Then noting
	\begin{equation*}
		e\leq e\cdot \exp\Bigl\{C\int^t_0\bigl[b(s) + C(a (s) y(s) +1)  \bigr] ds \Bigr\}
	\end{equation*}
	and calculating $\ln(C\times$\eqref{e:aprest3}$+e)$ yield
	\begin{align*}
	y(t)\leq & y(0)+ C \int^t_0 \bigl[\bigl(b(s)+C\bigr)+C a(s) y(s) \bigr] ds \notag \\
	\leq & y(0)+  \int^t_0 \bigl(C_1 b(s)+C_2\bigr) ds+\int^t_0 C_3 a(s) y(s)  ds
	\end{align*}
 for every $t\in[0,T^\star)$.
	Letting $G(t):= y(0)+\int^t_0 C_1 b(s)ds +C_2 t$, an increasing function of $t$, and using generalized Gronwall inequality of integral form (see Theorem \ref{e:grnwl} in Appendix \S\ref{s:grnwl}),  we arrive at
	\begin{align*}
	y(t)\leq G(t)\exp{\Bigl(\int^t_0C_3a(s)ds\Bigr)}
	\end{align*}
for every $t\in[0,T^\star)$. From this inequality, we improve a priori estimate of  $\|\mathcal{U}(t)\|_{H^s(\Omega(t))}$,
	\begin{align}\label{e:Uest}
	&\|\mathcal{U}(t)\|_{H^s(\Omega(t))}\leq \notag \\
	& \hspace{0.5cm} C\biggl\{ (C\| \mathcal{U}(0)\|_{H^s(\Omega(0))}+e)\exp\Bigl( C_2t+ C_1 \int^t_0b(s) ds+\exp\Bigl[C_3 \int^t_0  a(s) ds\Bigr]\Bigr) -e\biggr\}
	\end{align}
for every $t\in[0,T^\star)$.
Using \eqref{e:Uest}, \eqref{e:bdwv0} and Condition \eqref{e:contpr3} guarantee $\|\mathcal{U}(t)\|_{H^s(\Omega(t))}<\infty$.
It means we can continue the solution to a larger interval $[0,T^\prime)$ for $T^\prime>T^\star$ by the standard continuation criterion (see, for example, \cite[\S $2.2$]{Majda2012}). It contradicts the fact that $T^\star$ is the maximal time of existence. We then complete the proof.
\end{proof}

As \cite[Corollary $2$]{Brauer1998} for the key ideas, we can similarly prove the following corollary. We omit the detailed proof.
\begin{corollary}
	Under the assumption of Theorem \ref{t:conpri2} there are the following possibilities:
	\begin{enumerate}
		\item $T^\star=\infty$;
		\item or $T^\star<\infty$ and there is a small constant $\epsilon>0$ such that
		\begin{equation*}
			\int^{T^\star}_0 \Bigl(\|V_{jk}(s)\|_{L^\infty( \Omega_\epsilon(s))}+\|\Theta_{jk}(s)\|_{L^\infty(\mathring{\Omega}_\epsilon(s))}+\|\nabla  w(s)\|_{ L^\infty(\Omega(s))} \Bigr) ds=\infty.
		\end{equation*}
		where $\mathring{\Omega}_\epsilon(s):=\Omega(s)\setminus \Omega_\epsilon(s)$.
	\end{enumerate}
\end{corollary}
\begin{remark}
	Similarly to \cite[\S VII]{Brauer1998}, we can discuss the possible singularities.  We only mention some of them here. As mentioned in \cite[\S VII]{Brauer1998}, if the singularities occur  in the interior of the fluids, $\int^{T^\star}_0 \|\nabla  w(s)\|_{ L^\infty(\ring{\Omega}_\epsilon(s))}   ds=\infty$ implies $\int^{T^\star}_0 \|\nabla  \rho(s)\|_{ L^\infty(\ring{\Omega}_\epsilon(s))}   ds=\infty$. If this singular is near the boundary, there is no information implying $\int^{T^\star}_0 \|\nabla  \rho(s)\|_{ L^\infty(\Omega_\epsilon(s))}   ds=\infty$, in fact, it may not be true. In addition, $\int^{T^\star}_0 \|\nabla  w(s)\|_{ L^\infty( \Omega_\epsilon(s))}   ds=\infty$ implies $\int^{T^\star}_0  \| \rrho^{-\frac{\gamma-1}{2}}\nabla c_s^2(s)\|_{ L^\infty(\Omega_\epsilon(s))}   ds =+\infty$. From this relation, we can see the physical vacuum boundary $
	-\infty<\nabla_{\mathbf{n}}(c^2_s)<0$ on $\partial\Omega(t)$  is also allowable.
\end{remark}

\appendix

\section{Preliminary lemmas of uniqueness theorem \ref{t:unithm}}\label{a:Lpf}

\subsection{Relative entropy estimates}	
We summarize the following useful lemma from  \cite[Step $2$, Page $772$--$777$]{Luo2014} where the authors use it to prove the uniqueness theorem for physical vacuum free boundary problem, while we will use it to prove the uniqueness theorem \ref{t:unithm} of Problem \ref{prob}.
	\begin{lemma}\label{t:unqr3}
		Suppose $\gamma\in(1,2]$, $T>0$, $\rrho_\ell \in C^1([0,T)\times\Rbb^3)$ and  $v^i_\ell \in C^1([0,T)\times\overline{\Omega_\ell(t)}, \Rbb^3)
		\cap
		 W^{1,\infty}([0,T)\times \Rbb^3, \Rbb^3)$,
		for $\ell=1,2$, both solve the system,
		\begin{align*}
		\del{t}\rrho_\ell+ \del{i} ( \rrho_\ell v^i_\ell) = & 0 \quad &&\text{in}\quad [0,T)\times \Rbb^3 ,  \\
		\del{t}( \rrho_\ell v^k_\ell)+\del{i} ( \rrho  v^i_\ell  v^k_\ell) +\delta^{ik} \del{i} p_\ell = & -  \rrho_\ell\del{ }^k  \rPhi_\ell, && \text{in}\quad [0,T)\times \Rbb^3,  \\
		\rrho_\ell> & 0, &&  \text{in} \quad \Omega_\ell(t),  \\
		\rrho_\ell = & 0, &&  \text{in} \quad  \Rbb^3\setminus\Omega_\ell(t),
		\end{align*}
		where
		\begin{equation*}
		\rPhi_\ell(t,\bx)=   -\frac{1}{4 \pi}\int_{\Rbb^3} \frac{ \rrho_\ell(t,\by)}{|\bx -\by |}d^3 \by, \quad \text{for}\quad  (t, \bx) \in  [0,T)\times\Rbb^3, 	\end{equation*}
		the equation of state \eqref{e:eos1} holds.
        Define the relative entropy
        \begin{equation*}
        	\eta^\star:=\frac{1}{\gamma-1}[\rrho^\gamma_2-\rrho^\gamma_1-\gamma\rrho^{\gamma-1}_1(\rrho_2-\rrho_1)]+\frac{1}{2}\rrho_2|\mathbf{ v}_2-\mathbf{ v}_1|^2.
        \end{equation*}
        Then
        \begin{align}\label{e:etaest}
        	\eta^\star(t,\bx)\geq & C(\|\rrho_2(\cdot,t)\|_{L^\infty}+\|\rrho_2(\cdot,t)\|_{L^\infty})^{\gamma-2}(\rrho_2-\rrho_1)^2+\frac{1}{2}\rrho_2|\mathbf{v}_2-\mathbf{v}_1|^2\geq 0,
        \end{align}
        and $\eta^\star(t) \in L^1(\Rbb^3)$ with an estimate
        \begin{align*}
        	\int_{\Rbb^3} \eta^\star(t,\bx)d\bx \leq &\int_{\Rbb^3} \eta^\star(0,\bx)d\bx+ C\sup_{0\leq \tau<T}(\|\nabla_\bx \mathbf{ v}_1(\tau,\cdot)\|_{L^\infty}+Z(\tau))\int^t_0\int_{\Rbb^3} \eta^\star(\tau,\bx) d\bx d\tau
        \end{align*}
        where
        \begin{equation*}
        	Z(\tau):=\|\rrho_2(\tau,\cdot)\|_{L^\infty}(\|\rrho_2(\tau,\cdot)\|_{L^\infty}+\|\rrho_1(\tau,\cdot)\|_{L^\infty})^{2-\gamma} (\mathrm{Vol}S(\tau))^{\frac{2}{3}}
        \end{equation*}
        and
        \begin{equation*}
        	S(\tau):=\{\bx \in \Rbb^3\;|\; |\rrho_1(\tau,\bx)-\rrho_2(\tau,\bx)|>0\}
        \end{equation*}
        for $\tau\in[0,T)$.
	\end{lemma}
	\begin{proof}
		See Step $2$ in \cite[Page $772$--$777$]{Luo2014} for detailed proof. We only prove $\eta^\star(t)\in L^1$. Since by the assumptions of Lemma B.$1$, $\rho_\ell=0$ ($\ell=1,2$) in $\Rbb^3\setminus \Omega_\ell(t)$, then by the definition of $\eta^\star$, we have, for $(t,\mathbf{x})\in [0,T)\times (\Omega_1(t)\cup \Omega_2(t))^\mathsf{c}$,
		\begin{equation*}
			\eta^\star:=\frac{1}{\gamma-1}[\rho^\gamma_2-\rho^\gamma_1-\gamma\rho^{\gamma-1}_1(\rho_2-\rho_1)]+\frac{1}{2}\rho_2|\mathbf{ v}_2-\mathbf{ v}_1|^2=0.
		\end{equation*}
		Again by the definition of $\eta^\star$, $\eta^\star(t) \in L^\infty(\Omega_1(t)\cup \Omega_2(t))\subset L^1(\Omega_1(t)\cup \Omega_2(t))$ due to the fact $\rho_\ell \in C^1([0,T)\times\Rbb^3)$ and  $v^i_\ell \in C^1([0,T)\times\overline{\Omega_\ell(t)}, \Rbb^3)
		\cap
		W^{1,\infty}([0,T)\times \Rbb^3, \Rbb^3)$.
	\end{proof}

\subsection{The proof of Proposition \ref{t:bdgl} for boundary gluing}\label{s:bdgl}

	Let us focus on $\mathcal{D}_l  \cup \mathcal{D}_b$ is Lipschitz and $\mathcal{D}_l  \cup \mathcal{D}_t$ can be proven similarly, and then conclude this proposition. The key idea of this proof
	is to construct a local Lipschitz function near $\mathcal{D}_l  \cap \mathcal{D}_b$ and further the local boundary of $\mathcal{D}_l  \cup \mathcal{D}_b$ near $\mathcal{D}_l  \cap \mathcal{D}_b$ is Lipschitz. We also prove $\mathcal{D}_l\in C^1$, and with the help of $\mathcal{D}_b\in C^1$, we obtain the complete boundary is Lipschitz.

\subsubsection{Step $1$: $\mathcal{D}_l \in C^1$ and  $\mathcal{D}_l$ and $\mathcal{D}_b$ are not tangential to each other near $\mathcal{D}_l\cap \mathcal{D}_b$}
Since $\del{}\Omega(0)$ is a $C^1$ boundary of $\Omega(0)$, by Definition \ref{t:bdryreg}, for every point $\xi^i\in\del{}\Omega(0)$, there is a  rigid map $\mathfrak{T}$, a neighborhood $\mathcal{C}:=B(0,d)\times(-h,h)\subset \Rbb^3$ where constants $d,h>0$, along with a $C^1$ function $\psi:\Rbb^{2}\supset \overline{B(0,d)}\rightarrow (-h,h)$, such that $\psi(0,0)=0$, $\nabla \psi(0,0)=0$, and satisfies
	\begin{align}
	\mathcal{C}\cap \mathfrak{T}(\Omega(0) )=&\{(\xi^1,\xi^2,\xi^3)\in \mathcal{C}\;|\;\xi^3>\psi(\xi^1,\xi^2)\};  \label{e:psiint} \\
	\mathcal{C}\cap \mathfrak{T}(\del{}\Omega(0) )=&\{(\xi^1,\xi^2,\xi^3)\in \mathcal{C} \;|\;\xi^3=\psi(\xi^1,\xi^2)\}. \label{e:psi}
\end{align}
Note \eqref{e:psiint} implies if $\xi^3>0$ and $(0,0,\xi^3)\in\mathcal{C}$, a point  $(0,0,\xi^3)\in\mathcal{C}\cap \mathfrak{T}(\Omega(0) )$  since $\xi^3>0=\psi(0,0)$.
	
By Lemma \ref{t:surfc2} and Theorem \ref{t:C1bdry},
\begin{equation*}
	\mathcal{D}_l=[0,T)  \times \del{}\Omega(t)
	=  \bigl\{\bigl(t,\varphi(t,0,\xi)\bigr)\;|\;t\in [0,T) , \xi\in \del{}\Omega(0) \bigr\}
\end{equation*}
where $\varphi$ is the $C^1$ flow generalized by the vector
field $\mathbf{v}$ and the inverse of $\bx=\varphi(t,0,\xi) = \varphi^{(t,0)}(\xi)$ is $\xi=\varphi(0,t,\bx) = \varphi^{(0,t)}(\bx)$. For simplicity of statements, we extend
 $C^1$ flow $\varphi$ to $t\in(-T_1,T)$ ($T_1>0$) and denote $\tilde{\mathcal{D}}_l=  \bigl\{\bigl(t,\varphi(t,0,\xi)\bigr)\;|\;t\in (-T_1,T) , \xi\in \del{}\Omega(0) \bigr\}$ and the components of $\xi$ by $\xi^i:=  \varphi^i(0, t, \bx)$.

Let us prove $\mathcal{D}_l\in C^1$. We need to prove for every point $x^\mu\in\mathcal{D}_l$, it locally can be represented by a graph of a $C^1$ function. However, without loss of generality, we only focus on the case $\xi^\mu=(0,0,0,0)\in \mathcal{D}_l$. Because for any  case $\xi^\mu=(0,\xi)\in \mathcal{D}_l$, we can use some rigid map $\mathfrak{T}$ to move the target point to the origin and further, for any more general case $x^\mu\in\mathcal{D}_l$, since Theorem \ref{t:C1bdry} implies $\del{}\Omega(t)$ is $C^1$ for every $t\in[0,T]$, we shift $x^\mu\in\mathcal{D}_l$ to $x^0=0$ by a rigid map and it turns to the previous case. From now on, we only concentrate on the case $\xi$ is on the origin.

By \eqref{e:psi}, we have
\begin{align*}
	\Psi(t,\bx):
	=&  \varphi^3(0, t,\bx)-\psi\bigl( \varphi^1(0, t,\bx), \varphi^2(0, t,\bx)\bigr)=0
\end{align*}	
for $(t,\bx)\in \tilde{\mathcal{D}}_l$ to represent the surface $ \mathcal{D}_l\subset \tilde{\mathcal{D}}_l$.	Therefore, $\Psi\in C^1$.
Note $\varphi^{(0,t)}\circ\varphi^{(t,0)}=\mathds{1}$ ($\mathds{1}$ denotes an identity map), i.e. $\varphi(0,t,\varphi(t,0,\xi))=\xi$, then by chain rules, differentiating this identity with respect to the time $t$, we obtain
\begin{align}\label{e:deltpsi}
	\del{t}\varphi^j(0,t,\bx)+v^i\frac{\del{}\xi^j}{\del{}x^i}=0.
\end{align}

Next, let us calculate the gradient of $\Psi$,
\begin{align*}
	\del{\mu}\Psi=& \del{\mu}\varphi^3(0, t,\bx)-   (\del{\xi^1}\psi) \del{\mu}\varphi^1(0, t,\bx)-  (\del{\xi^2}\psi)  \del{\mu}\varphi^2(0, t,\bx).
\end{align*}
That is, with the help of \eqref{e:deltpsi},
\begin{align*}
	\del{t}\Psi=&-v^i\frac{\del{}\xi^3}{\del{}x^i}+(\del{\xi^1}\psi) v^i\frac{\del{}\xi^1}{\del{}x^i}+(\del{\xi^2}\psi) v^i\frac{\del{}\xi^2}{\del{}x^i} \\
	\intertext{and}
	\del{k}\Psi=&\frac{\del{}\xi^3}{\del{}x^k} - (\del{\xi^1}\psi)\frac{\del{}\xi^1}{\del{}x^k}  - (\del{\xi^2}\psi)\frac{\del{}\xi^2}{\del{}x^k} .
\end{align*}
Since $\varphi$ is a flow, it is invertible (diffeomorphism), $\det(\frac{\del{}\xi^i}{\del{}x^k})_{3\times 3} \neq 0$ and $\frac{\del{}\xi^i}{\del{}x^k}\big|_{(t,\bx)=(0,\xi)}=\delta^i_k$ due to $\varphi(0,0,\xi)\equiv \xi$. Then a normal vector $\mathbf{n}_1$ of the three dimensional hypersurface $\tilde{\mathcal{D}}_l $ at the origin in $ \Rbb^4$ is,
\begin{align}\label{e:nabF}
	\mathbf{n}_1:=\del{\mu}\Psi(0,0,0,0)=(-v^3,0,0,1)\neq (1,0,0,0)=:\mathbf{n}_0
\end{align}
where $\mathbf{n}_0$ is a unit normal vector of $\mathcal{D}_b$. This means $\tilde{\mathcal{D}}_l$ and $\mathcal{D}_b$ are not tangential to each other at the origin.
	
Since $\del{3}\Psi(0,0,0,0)=1\neq 0$ as indicated by \eqref{e:nabF}, with the help of the implicit function theorem, there is a small neighborhood  $\mathcal{C}_0$ of the origin and a function $f\in C^1$, such that in this neighborhood, $x^3=f(x^0,x^1,x^2)$ satisfying $\xi^3=f(0,\xi^1,\xi^2)$, i.e., $0=f(0,0,0)$. Since Theorem \ref{t:C1bdry} implies $\del{}\Omega(t)$ is $C^1$ for every $t\in[0,T]$, the above arguments hold for every point $x^\mu\in[0,T]\times \del{}\Omega(t)$ by shifting every $t>0$ to $0$ (this shift is included in the following rigid map $\mathfrak{T}_1$). That is, for every $(t,\bx)\in[0,T]\times \del{}\Omega(t)$, there is a rigid map $\mathfrak{T}_1:\Rbb^4\rightarrow \Rbb^4$ with $\mathfrak{T}_1(t,\bx)=0$ and a neighborhood $\mathcal{C}_1:=B(0,d_1)	\times (-h_1,h_1)\subset \mathcal{C}_0$, along with a $C^1$ function  $\tilde{f}:\Rbb^3\supset \overline{B(0,d_1)}\rightarrow (-h_1,h_1)$ ($\tilde{f}$ is a function that $f$ compose a certain rotation such that $\nabla \tilde{f}(0)=0$, we omit the detailed expression), such that $\tilde{f}(0)=0$, $\nabla \tilde{f}(0)=0$, and Theorem \ref{t:bdrythm}, \eqref{e:psi} and \eqref{e:psiint} imply
	\begin{align*}
		\mathcal{C}_1\cap \mathfrak{T}_1([0,T]\times\Omega(t) ) =& \{x^\mu \in \mathcal{C}_1\;|\;\Psi(t,\bx)>0\}
		=  \{x^\mu \in \mathcal{C}_1\;|\;x^3>\tilde{f}(t,x^1,x^2)\} ; \\
		\mathcal{C}_1\cap \mathfrak{T}_1([0,T]\times\del{}\Omega(t) ) =& \{x^\mu \in \mathcal{C}_1\;|\;\Psi(t,\bx)=0\}
		=   \{x^\mu \in \mathcal{C}_1\;|\;x^3=\tilde{f}(t,x^1,x^2)\}.
	\end{align*}
This concludes $\mathcal{D}_l=[0,T]\times \del{}\Omega(t)\in C^1$.

\subsubsection{Step $2$: Construction of local graph}
In order to prove $\mathcal{D}_l\cup\mathcal{D}_b$ is Lipschitz, we focus on the points on $\mathcal{D}_l\cap\mathcal{D}_b$ and see if in neighborhoods of these points,  $\mathcal{D}_l\cup\mathcal{D}_b$ can be represented by a local graph of a Lipschitz function. In order to construct a local Lipschitz function (and to avoid the multi-valued function appearing), we introduce a rotation in $\Rbb^4$, and will rotate the original coordinate $\{x^\mu\}$ to a new one $\{\hat{x}^\mu\}$, such that $\mathcal{D}_l\cup \mathcal{D}_b$ is locally graph-represented by a suitable Lipschitz function in $\{\hat{x}^\mu\}$.
\begin{align*}
	\mathfrak{R}_\vartheta:=\p{\cos\vartheta & 0 & 0 & -\sin\vartheta\\
	0 & 1 & 0 & 0 \\
    0 & 0 & 1 & 0 \\
 \sin\vartheta & 0 & 0 & \cos\vartheta }
\end{align*}
where $\vartheta\in \bigl(0,\frac{\pi}{2}\bigr)$  is given by\footnote{In fact, $2\vartheta$ is the angular  between $\mathbf{n}_1$ and $\mathbf{n}_0$ (recall \eqref{e:nabF}), \begin{equation*}
		2\vartheta=  \arccos\frac{-v^3}{\sqrt{1+|v^3|^2}}.
\end{equation*}	
This can be verified by noting  $\vartheta\in\bigl(0,\frac{\pi}{2}\bigr)$ and
 \begin{align*}
	\cos(2\vartheta)=\frac{1-\tan^2 \vartheta}{1+\tan^2 \vartheta} =\frac{-v^3}{\sqrt{1+|v^3|^2}}.
\end{align*}}
\begin{equation}\label{e:rta}
	\tan \vartheta=v^3+\sqrt{1+(v^3)^2}>0.
\end{equation}
From this, we also know the inverse of the rotation is  $\mathfrak{R}_\vartheta^{-1}=\mathfrak{R}_{-\vartheta}$.
	
Then we obtain a coordinate transform  $x^\mu=\mathfrak{R}_\vartheta \hat{x}^\mu$ for fixed $\vartheta$ defined by \eqref{e:rta}. By reexpressing $\Psi(x^\mu)=0$, for $x^\mu \in \tilde{\mathcal{D}}_l$, in terms of $\hat{x}^\mu$, we obtain	$\Psi\circ\mathfrak{R}_\vartheta(\hat{x}^\mu)=0$. We still, with loss of generality, focus on the case  $x^\mu=(0,\xi^i)=(0,0,0,0)\in\overline{\mathcal{D}}_l\cap\overline{\mathcal{D}}_b$. For other general cases, if $x^\mu$ is not the origin, we use a rigid map to move it to the origin as before.

By assuming for a point $x^\mu=(0,0,0,0)$, we know $\Psi\circ\mathfrak{R}_\vartheta(0,0,0,0)=0$ and  $\del{\hat{x}^0}(\Psi\circ\mathfrak{R}_\vartheta)=\del{\mu}\Psi\cdot(\cos\vartheta,0,0,\sin\vartheta)=(-v^3,0,0,1)\cdot(\cos\vartheta,0,0,\sin\vartheta)=-v^3\cos\vartheta+\sin\vartheta=\cos\vartheta \sqrt{1+(v^3)^2}\neq 0$. Then with the help of the implicit theorem, there is a small neighborhood $0\in \mathcal{C}_2 :=B(0,d_2) \times (-h_2,h_2) \subset \Rbb^3\times \Rbb$
and in it there is a function $\hat{f}\in C^1(B(0,d_2))$, such that $\hat{x}^0=\hat{f}(\hat{x}^1,\hat{x}^2,\hat{x}^3)$ satisfying $\hat{f}(0,0,0)=0$ and
\begin{align}
	\del{\hat{x}^2}\hat{f}|_{(0,0,0)}=&-\frac{\del{\hat{x}^2}(\Psi\circ\mathfrak{R}_\vartheta)}{\del{\hat{x}^0}(\Psi\circ\mathfrak{R}_\vartheta)}=0 \AND  \del{\hat{x}^1}\hat{f}|_{(0,0,0)}=-\frac{\del{\hat{x}^1}(\Psi\circ\mathfrak{R}_\vartheta)}{\del{\hat{x}^0}(\Psi\circ\mathfrak{R}_\vartheta)}=0, \label{e:delf1} \\
	\del{\hat{x}^3}\hat{f}|_{(0,0,0)}=&-\frac{\del{\hat{x}^3}(\Psi\circ\mathfrak{R}_\vartheta)}{\del{\hat{x}^0}(\Psi\circ\mathfrak{R}_\vartheta)}=-\frac{v^3\sin\vartheta+\cos\vartheta}{-v^3\cos\vartheta+\sin\vartheta}=\tan\vartheta +\frac{1}{(v^3-\tan\vartheta)\cos^2\vartheta} \notag   \\
	=& \tan\vartheta  - \frac{1}{  \sqrt{1+(v^3)^2}\cos^2\vartheta}. \label{e:delf2}
\end{align}
Then because of $\hat{f}\in C^1(B(0,d_2))$, we obtain $\del{\hat{x}^k}\hat{f}\in C^0(B(0,d_2))$, i.e., for any $\epsilon\in\bigl(0,\frac{1}{\sqrt{1+(v^3)^2}}\bigr)$, there is a constant $\delta>0$, such that if $|\hat{\bx}|<\delta$, then $|\del{\hat{x}^k}\hat{f}(\hat{\bx})-\del{\hat{x}^k}\hat{f}(\mathbf{0})| <\epsilon$.

Denote $\mathcal{C}_4:=\mathfrak{R}_\vartheta\bigl(B(0,\delta)\times(-h_2,,h_2)\bigr)$ and define
\begin{align*}
	g(x^1,x^2,x^3):=\hat{f}(x^1,x^2,x^3 \cos \vartheta)-x^3\sin\vartheta
\end{align*}
Let us now prove
\begin{align}
	\mathcal{C}_4\cap\overline{\mathcal{D}_l}\cap \overline{\mathcal{D}_b}=&\mathcal{C}_4\cap\bigl(\{0\}\times \del{}\Omega(0)\bigr)=  \{(0,x^1,x^2,x^3)\in \mathcal{C}_4\;|\; g(x^1,x^2,x^3)=0\} \label{e:int1}
	\intertext{and}
	\mathcal{C}_4\cap \mathcal{D}_b=&\mathcal{C}_4\cap \bigl(\{0\}\times \Omega(0)\bigr)=  \{(0,x^1,x^2,x^3)\in \mathcal{C}_4\;|\; g(x^1,x^2,x^3)<0\}. \label{e:int2}
\end{align}

First we note for any $x^\mu \in \mathcal{C}_4\cap\overline{\mathcal{D}_l}\cap \overline{\mathcal{D}_b}$, since $x^\mu\in\overline{\mathcal{D}_l}$, we have  $\hat{x}^0=\hat{f}(\hat{x}^1,\hat{x}^2,\hat{x}^3)=\hat{f}(x^1,x^2,-x^0\sin\vartheta+x^3\cos\vartheta)$, and since  $\hat{x}^0=x^0\cos\vartheta+x^3\sin\vartheta$, and $x^\mu\in\overline{\mathcal{D}_b}$ implies $x^0=0$, we obtain $\hat{f}(x^1,x^2, x^3\cos\vartheta) =x^3\sin\vartheta$, then $g(x^1,x^2,x^3)=0$, i.e., we complete \eqref{e:int1}.

Let us prove \eqref{e:int2} now. Firstly, with the help of $\hat{x}^3 = - x^0 \sin \vartheta + x^3\cos \vartheta =  x^3\cos \vartheta$ (since $x^0 = 0$ for $x^\mu\in \mathcal{D}_b$) and the mean value theorem, there is a point $\hat{\bx}_0=\lambda_1 \hat{\bx}$ for some constant $\lambda_1\in(0,1)$,
\begin{align*}
	\hat{f}(\hat{x}^1,\hat{x}^2,\hat{x}^3) -\hat{x}^3 \tan \vartheta
	=  \hat{x}^1 \del{\hat{x}^1} 	\hat{f}(\hat{\bx}_0)  +\hat{x}^2 \del{\hat{x}^2} 	\hat{f}(\hat{\bx}_0)   +\hat{x}^3 \del{\hat{x}^3} 	\hat{f}(\hat{\bx}_0) -\hat{x}^3\tan \vartheta
\end{align*}
then for $(0,0,0,x^3)\in \mathcal{C}_4\cap\mathcal{D}_b$ ($x^3=\xi^3>0$),  using $\epsilon<\frac{1}{\sqrt{1+(v^3)^2}}$ and  \eqref{e:delf1}--\eqref{e:delf2},
\begin{align}\label{e:gest1}
	& g(0,0,x^3)=\hat{f}(0,0,\hat{x}^3) -\hat{x}^3 \tan \vartheta
	<   \hat{x}^3 \del{\hat{x}^3} 	\hat{f}(\mathbf{0})  +\hat{x}^3 \epsilon -\hat{x}^3\tan \vartheta  \notag  \\
	&\hspace{1cm}=    - \frac{\hat{x}^3 }{  \sqrt{1+(v^3)^2}\cos^2 \vartheta}  +\hat{x}^3 \epsilon
	=  \biggl( - \frac{1 }{  \sqrt{1+(v^3)^2} }+\epsilon \cos \vartheta\biggr)x^3 <0,
\end{align}
Next we prove for any $ (0,y^1,y^2,y^3) \in \mathcal{C}_4\cap \mathcal{D}_b$, $g(y^1,y^2,y^3)<0$. We prove this by contradictions. Assume $g(y^1,y^2,y^3) \geq 0$. We focus on $g(y^1,y^2,y^3) > 0$ (since $g(y^1,y^2,y^3) = 0$ implies $(0,y^1,y^2,y^3)\in \mathcal{C}_4\cap\overline{\mathcal{D}_l}\cap \overline{\mathcal{D}_b}$, this leads to a contradiction), then there is a connected open set $\mathcal{O}$, such that $(0,0,0,x^3)$ and $(0,y^1,y^2,y^3) \in\mathcal{O}\subset \mathcal{C}_4\cap\mathcal{D}_b$. Then by the continuity of $g$ and intermediate value theorem, there is a point $(0,z^1,z^2,z^3)\in\mathcal{O}$, such that $g(z^1,z^2,z^3)=0$, which implies $(0,z^1,z^2,z^3)\in \mathcal{C}_4\cap\overline{\mathcal{D}_l}\cap \overline{\mathcal{D}_b}$, this leads to a contradiction. Therefore, $g(y^1,y^2,y^3)<0$.

Gathering the above facts together, we conclude for every point in $\mathcal{D}_l\cap\mathcal{D}_b$, there is a rigid map $\mathfrak{T}_4$ (including above rigid movements and rotation $\mathfrak{R}_\vartheta$), a small neighborhood $ \widetilde{\mathcal{C}}_4:=B(0,\delta)\times (-h_0,h_0) $, along with a function defined by
\begin{align}\label{e:mfF}
	\mathfrak{F}(\hat{x}^1,\hat{x}^2,\hat{x}^3):
	=&\max\{\hat{f}(\hat{x}^1,\hat{x}^2,\hat{x}^3),\hat{x}^3	\tan\vartheta\}
\end{align}
for $\hat{x}^i\in\overline{B(0,\delta)}$,
such that
\begin{align}
	\mathcal{D}_4\cap\mathfrak{T}_4(\mathcal{D}_l\cup\mathcal{D}_b)=&\{\hat{x}^\mu\in \widetilde{\mathcal{C}}_4\;|\;\hat{x}^0=\mathfrak{F}(\hat{x}^1,\hat{x}^2,\hat{x}^3)\};  \label{e:dt1} \\
	\mathcal{D}_4\cap\mathfrak{T}_4((0,T)\times \Omega(t))=&\{\hat{x}^\mu\in \widetilde{\mathcal{C}}_4\;|\;\hat{x}^0>\mathfrak{F}(\hat{x}^1,\hat{x}^2,\hat{x}^3)\}. \label{e:dt2}
\end{align} 	

Firstly, let us first prove \eqref{e:dt1}. Since $\widetilde{\mathcal{C}}_4\cap \mathfrak{T}_4(\mathcal{D}_l\cup\mathcal{D}_b)=(\widetilde{\mathcal{C}}_4\cap\mathfrak{T}_4 \mathcal{D}_l)\cup(\widetilde{\mathcal{C}}_4\cap\mathfrak{T}_4\mathcal{D}_b)$.
If a point  $x^\mu\in\mathcal{D}_l \cap \mathcal{C}_4 $, then its time component $t=x^0 \geq 0$, i.e.,  $x^0=\hat{x}^0\cos\vartheta-\hat{x}^3 \sin \vartheta \geq 0$ and it implies $\hat{x}^0=\hat{f}(\hat{x}^1,\hat{x}^2,\hat{x}^3)\geq \hat{x}^3\tan\vartheta$. In this case, it means \eqref{e:mfF}, i.e., $\hat{x}^0=\hat{f}(\hat{x}^1,\hat{x}^2,\hat{x}^3)=\max\{\hat{f}(\hat{x}^1,\hat{x}^2,\hat{x}^3),\hat{x}^3	\tan\vartheta\} =\mathfrak{F}(\hat{x}^1,\hat{x}^2,\hat{x}^3)
$.
On the other hand, if a point  $x^\mu\in \mathcal{D}_b \cap \mathcal{C}_4 $, then its time component $t=x^0 = 0$, i.e.,  $x^0=\hat{x}^0\cos\vartheta-\hat{x}^3 \sin \vartheta = 0$ and it implies $\hat{x}^0 = \hat{x}^3\tan\vartheta$, and by \eqref{e:int2}, we arrive at $g(x^1,x^2,x^3)=\hat{f}(x^1,x^2,x^3\cos \vartheta)-x^3\sin\vartheta<0$, and by using $\hat{x}^\mu=\mathfrak{R}_{-\vartheta}x^\mu$ and $x^0=0$,  we obtain $ \hat{f}(\hat{x}^1,\hat{x}^2,\hat{x}^3)<\hat{x}^3\tan\vartheta$. This yields $\hat{x}^0=\hat{x}^3\tan\vartheta=\mathfrak{F}(\hat{x}^1,\hat{x}^2,\hat{x}^3)$.

Secondly, let us prove \eqref{e:dt2}. We first recall $(0,0,0,x^3)\in \mathcal{C}_4\cap \mathcal{D}_b$ (for $x^3>0$) satisfies $\hat{x}^0>\hat{f}(\hat{x}^1,\hat{x}^2,\hat{x}^3)$ due to the fact that $\hat{x}^\mu=\mathfrak{R}_{-\vartheta}x^\mu$ and $x^3\sin\vartheta>\hat{f}(0,0,x^3\cos\vartheta)$ (see \eqref{e:gest1}). Then, by intermediate value theorem again, we claim for any point $x^\mu\in(0,T)\times \Omega(t)$, on one hand,  $\hat{x}^0>\hat{f}(\hat{x}^1,\hat{x}^2,\hat{x}^3)$ for $\hat{x}^\mu=\mathfrak{R}_{-\vartheta}x^\mu$ (otherwise, proofs by contradictions, there is a point $z\in(-T_1,T)\times \Omega(t)$, such that $\hat{z}^0=\hat{f}(\hat{z}^1,\hat{z}^2,\hat{z}^3)$ and this can only happen when $z\in\mathcal{D}_l$, this is a contradiction). On the other hand, for any point $x^\mu\in(0,T)\times \Omega(t)$, since $x^0=\hat{x}^0\cos\vartheta-\hat{x}^3\sin\vartheta>0$, we obtain $\hat{x}^0>\hat{x}^3\tan \vartheta$. Therefore, $\hat{x}^0>\mathfrak{F}(\hat{x}^1,\hat{x}^2,\hat{x}^3)$.

In the end, let us verify that $\mathfrak{F}$ is a Lipschitz function. We only focus on the case for two points $\hat{x}^i_1$ and $\hat{x}^i_2\in B(0,\delta)$, such that $\mathfrak{F}(\hat{x}^i_1)= \hat{f} (\hat{x}^i_1)$ and $\mathfrak{F}(\hat{x}^i_2)=\hat{x}^3_2\tan\vartheta$ (i.e., $\hat{f}(\hat{x}_1^1,\hat{x}_1^2,\hat{x}_1^3) \geq \hat{x}_1^3	\tan\vartheta$ and $\hat{f}(\hat{x}_2^1,\hat{x}_2^2,\hat{x}_2^3)\leq \hat{x}_2^3	\tan\vartheta$). For other cases, they are either similar to the above case or, by using $\hat{f}\in C^1$ and $\hat{x}^3 \tan\vartheta\in C^1$, we can obtain the Lipschitz inequality directly. Since $\hat{f}(\hat{x}^i)-\hat{x}^3\tan\vartheta$ is continuous and $\hat{f}(\hat{x}^i_1)-\hat{x}^3_1 \tan\vartheta\geq 0$,  $\hat{f}(\hat{x}^i_2)-\hat{x}^3_2\tan\vartheta \leq 0$, then there exists a point $\hat{x}^i_0 \in B(0,\delta)$, such that $\hat{f}(\hat{x}^i_0)-\hat{x}^3_0\tan\vartheta=0$, then
	\begin{align*}
		&|\mathfrak{F}(\hat{x}^i_1)-\mathfrak{F}(\hat{x}^i_2)|= |\hat{f} (\hat{x}^i_1)-\hat{x}^3_2\tan\vartheta|=|\hat{f} (\hat{x}^i_1)-\hat{f}(\hat{x}^i_0)+\hat{x}^3_0\tan\vartheta-\hat{x}^3_2\tan\vartheta| \notag  \\
		&\hspace{0.2cm}\leq |\hat{f} (\hat{x}^i_1)-\hat{f}(\hat{x}^i_0)|+|\hat{x}^3_0\tan\vartheta-\hat{x}^3_2\tan\vartheta|\leq  K(|\hat{x}_1^i-\hat{x}_0^i|+|\hat{x}_2^i-\hat{x}_0^i|) = K|\hat{x}_1^i-\hat{x}_2^i|
	\end{align*}
where $K:=\max\{|\tan\vartheta|,\|\nabla\hat{f}\|_{\Li}\}$.
Then we obtain $\mathfrak{F}$ is a Lipschitz function. This means $\mathcal{D}_l\cup\mathcal{D}_b$ is a Lipschitz boundary. Furthermore, the similar proofs arrive at $(\mathcal{D}_l\cup\mathcal{D}_b)\cup\mathcal{D}_t$ is of Lipschitz as well, then we complete the proof.

\section{Tools of geometry}\label{s:surf}

\subsection{The regularity of boundary and the geometric description}

Let us first recall the definition of  regularities of the boundary. See, for example,  \cite{Barb2009}.

\begin{definition}\label{t:bdryreg}
	For a proper, non-empty, open subset $\Omega\subset \Rbb^n$, we say the boundary \textbf{$\del{}\Omega$ is of class $C^k$ (or Lipschitz) locally at a point $\check{\bx} \in \del{}\Omega$}, if it can be expressed as a graph of a $C^k$ (or Lipschitz) function locally at a neighborhood of $\check{\bx}$, that is,
	at the point $\check{\bx} \in \del{}\Omega$, there is  a composition of a rotation and a translation (i.e., rigid map) $T:\Rbb^n\rightarrow \Rbb^n$ with $T(\check{\bx})=0$, and a neighborhood $\mathcal{C}:=B_d(0)\times(-h,h)\subset \Rbb^n$ where constants $d,h>0$, along with a $C^k$ (or Lipschitz) function $\phi:\Rbb^{n-1}\supset \overline{B_d(0)}\rightarrow (-h,h)$, such that $\phi(0)=0$, $\nabla \phi(0)=0$, and
	\begin{align*}
		\mathcal{C}\cap T(\Omega )=&\{(x^\prime,x^n)\in \mathcal{C}\;|\;x^n<\phi(x^\prime)\};  \\
		\mathcal{C}\cap T(\del{}\Omega)=&\{(x^\prime,x^n)\in \mathcal{C} \;|\;x^n=\phi(x^\prime)\}.
	\end{align*}
	Furthermore, we say the boundary \textit{$\del{}\Omega$ is of class $C^k$ (or Lipschitz)} or $\Omega$ is a $C^k$ (or Lipschitz) domain, if  $\del{}\Omega$ is of class $C^k$ (or Lipschitz) locally at every point $\check{\bx} \in \del{}\Omega$, and we simply denote this by $\del{}\Omega\in C^k$ (or $\del{}\Omega\in \text{Lip}$).
\end{definition}

Another important domain we will mention is the one with Calder\'{o}n's uniform cone condition. This domain has a good extension property (see \S\ref{s:extention}) allowing us to obtain the expected extension.
\begin{definition}[Calder\'{o}n's uniform cone condition (see \cite{Calderon1961,Adams2003b})]\label{t:conecd}
	A domain $\Omega\subset \Rbb^n$ satisfies \textbf{Calder\'{o}n's uniform cone  condition} if there exists a positive constant $\epsilon>0$, a finite open cover  $\{U_j \;|\; j=1,2\cdots \ell\}$ of $\del{}\Omega$ and a corresponding finitely many sequence $\{C_j \}$ of finite cones, such that
	\begin{enumerate}
		\item every point of $\del{}\Omega$ is the center of a sphere of radius $\epsilon$ entirely contained in one of these sets $U_j$;
		\item every point of $U_j\cap\Omega$ is the vertex of a translate of $C_j$ entirely contained in $\Omega$.
	\end{enumerate}
\end{definition}

From these definitions, we obtain
\begin{lemma}\label{t:dm1}
	If $\Omega$ is a bounded domain and $\del{}\Omega$ is a $C^1$ boundary, then it is a Lipschitz boundary and also satisfies  Calder\'{o}n's uniform cone condition.
\end{lemma}
\begin{proof}
	For a bounded domain $\Omega$, there is a finite open cover $U_j$ of $\del{}\Omega$ which can be rigidly moved to $\mathcal{C}=B_d(0)\times(-h,h)$ for some constant $d,h>0$. If $\del{}\Omega$ is $C^1$, then there is a $C^1$ function $\phi$ satisfies the requirements of the definition and it is direct that this function $\phi$ is a Lipschitz function on $\overline{B_d(0)}$, which in turn implies this domain is a Lipschitz domain. Since there is a finite open cover in the form of $\mathcal{C}=B_d(0)\times(-h,h)$, we can take $\epsilon$ small enough to ensure Definition \ref{t:conecd}.$(1)$ holds. For every point of $U_j\cap\Omega$ (for simplicity, we have been a little vague here. To be more precise, one should shrink $h$ and $d$ a half and complement finite open sets in this form if necessary to make sure the cone all in $\Omega$), by taking the height of the cone to be smaller than $h/2$ and the angle of the vertex to be $\min\{\arctan\frac{1}{|\phi^\prime|}, \arctan\frac{h}{d}\}>0$ where $\phi^\prime\in C^0(\overline{B_d(0)})$, Definition \ref{t:conecd}.$(2)$ holds. Then it satisfies  Calder\'{o}n's uniform cone condition.
\end{proof}

Due to the implicit function theorem, the following theorem gives $C^k$ boundary another description in terms of $C^k$ submanifold. We highlight this equivalence only holds for $k\geq1$ because of the implicit function theorem. Therefore, this excludes the equivalence between the Lipschitz boundary and submanifold (see \cite{Gilbarg1983,Grisvard2011}).

\begin{theorem}\label{t:bdsbmf}
	Let $\Omega$ be a nonempty, open, bounded subset of $\Rbb^n$, $n \geq 2$, and assume
	that $k\in \mathbb{N}_{\geq1}\cup \{\infty\}$. Then $\Omega$ is a $C^k$ domain if and only if for every point $\check{\bx}\in \del{}\Omega$ there
	exists an open neighborhood $U$ of   $\check{\bx}$ in $\Rbb^n$, $r>0$, and a $C^k$ diffeomorphism
	\begin{align*}
	\psi=(\psi_1,\cdots,\psi_n):U\rightarrow B(0,r),
	\end{align*}
	for which $\psi(\check{\bx})=0$ and which satisfies
	\begin{align*}
	\psi(\Omega\cap U)=& B(0,r)\cap \Rbb^n_+, \\
	\psi(\overline{\Omega}^\mathbf{c}\cap U)=& B(0,r)\cap \Rbb^n_-, \\
	\psi(\del{}\Omega\cap U)=& B(0,r)\cap \del{} \Rbb^n_+.
	\end{align*}
\end{theorem}
\begin{proof}
	See, for example, \cite[Theorem $4.6.8$]{Barb2009} or \cite{Gilbarg1983,Grisvard2011}.
\end{proof}

\subsection{The Sard's Theorem and  $C^\infty$-approximation lemma of boundary}\label{s:sard}
For readers' convenience, we cite the famous Sard’s theorem here, and more information can be found in, for instance, \cite{Lee2012}. We will use it to prove a useful lemma of $C^\infty$-approximation of the boundary (see John Lee's answer in \cite{orr}).
\begin{theorem}[The Sard’s theorem]\label{t:sard}
	Suppose $M$ and $N$ are smooth manifolds with
	or without boundaries and $F: M \rightarrow  N$ is a smooth map. Then the set of critical values of $F$ has measure zero in $N$.
\end{theorem}

\begin{lemma}[The $C^\infty$-approximation of Boundary]\label{t:Cinfapx}
	Suppose $\Omega$ is an open set in a smooth $n$ dimensional manifold $M$, then for any small constant $\epsilon>0$, there is an open set $\mathcal{D}$ with $C^\infty$ boundary, such that $\Omega\setminus \Omega_\epsilon \subset \mathcal{D}\subset \Omega$ where $\Omega_\epsilon$ is defined by \eqref{e:omgep} (see \S\ref{s:sets}).
\end{lemma}
\begin{proof}
	By the existence theorem of smooth bump functions (see \cite[Proposition 2.25]{Lee2012}), there is a smooth bump function $f:M\rightarrow \Rbb$ satisfying $f\equiv 1$ on $\Omega\setminus \Omega_\epsilon$, $\supp f \in\Omega$ and $0\leq f\leq 1$ on $M$. By using Sard's Theorem \ref{t:sard}, since the set of critical values of the bump function $f$ has measure zero, there always is a small enough regular value $\delta>0$ of $f$, such that $f^{-1}(\delta)$ forms the boundary of domain $\mathcal{D}:=f^{-1}([\delta,1])$ (see the regular level set theorem in \cite[Corollary 5.14]{Lee2012}) and $\Omega\setminus \Omega_\epsilon \subset \mathcal{D}\subset \Omega$. We have used  \cite[Proposition $5.47$]{Lee2012} to conclude $\mathcal{D}$ is a regular domain in $M$.
	
	We have to prove the boundary $\del{}\mathcal{D}$ is of class $C^\infty$ next. This is because $f(\bx)=\delta$ for every regular point, under some local coordinate, $\bx=(x^1,\cdots,x^{n-1}, x^n)\in\del{}\mathcal{D}$ and by letting $C^\infty\ni F(\bx):=f(\bx)-\delta$, the implicit function theorem implies there is a small neighborhood around every $\bx$, such that there is a function $\phi\in C^\infty$, $x^n=\phi(x^1,\cdots,x^{n-1})$. Then by Definition \ref{t:bdryreg}, with suitable rotations and translations, we conclude the boundary is smooth and further complete this proof.
\end{proof}

\subsection{The extension of the outward unit normal}\label{s:normal}
The following theorems come from \cite[Theorem $5.3.1$ and Lemma $4.6.18$]{Barb2009} which are useful tools used in the proof of the strong continuation criterion (i.e.,  Theorem \ref{t:conpri2}). We only cite the statements without proof.

\begin{proposition}[see Lemma $4.6.18$ in \cite{Barb2009}]\label{t:normal1}
	Suppose that $\Omega$ is a $C^k$ domain in $\Rbb^m$, $m\geq2$, for some $k \in \mathbb{N} \cup \{\infty\}$,
	$k\geq2$. Then the outward unit normal $\mathbf{n}$ is a function of class $C^{k-1}$.
	That is, there exists an open set $\mathcal{U} \subset \Rbb^m,$ which contains $\del{}\Omega$, along with a function
	$\mathbf{N}: \mathcal{U} \rightarrow \Rbb^m$ of class $C^{k-1}$ in $\mathcal{U}$ with the property that $\mathbf{N}|_{\del{}\Omega} = \mathbf{n}$.
\end{proposition}
\begin{definition}
	Given an arbitrary set $\Omega\subset \Rbb^m$, the \textbf{signed distance} (to its boundary $\del{}\Omega$) is the function $d:\Rbb^m \rightarrow \Rbb$, defined by
	\begin{align*}
	d(\bx):=\begin{cases}
	+\dist(\bx,\del{}\Omega), \quad \text{if} \quad \bx\in\Omega  \\
	-\dist(\bx,\del{}\Omega), \quad \text{if} \quad \bx\in\Omega^\mathbf{c}
	\end{cases}
	\end{align*}
\end{definition}
\begin{theorem}[The distinguished extension of the outward unit normal of a domain]\label{t:normal2}
	Assume that $\Omega\subset \Rbb^m$, $m\geq 2$, is a domain of class $C^k$, for some $k\in \mathbb{N}\cup\{\infty\}$, $k\geq2$. Then, with $d$ denoting the signed distance, there exists $\mathcal{U}$ open neighborhood of $\del{}\Omega$ and a vector field
	\begin{equation*}
	\mathbf{N}=(N^1, \cdots,N^m): \mathcal{U} \rightarrow \Rbb^m, \quad \mathbf{N}(\bx):=(\nabla d)(\bx), \quad \text{for any} \quad \bx\in \mathcal{U},
	\end{equation*}
	is a vector-valued function of class $C^{k-1}$ in $\mathcal{U}$ which has the following properties:
	\begin{enumerate}
		\item $\|\mathbf{N}(\bx)\|=1$ for every $\bx\in \mathcal{U}$;
		\item $\mathbf{N}|_{\del{}\Omega}=\mathbf{n}$, the outward unit normal to $\Omega$;
		\item $\del{j}N_k=\del{k}N_j$ in $\mathcal{U}$, for all $j,k\in \{1,\cdots,m\}$;
		\item for every $j\in \{1,\cdots,m\}$, the directional derivative $D_\mathbf{N} N_j$ vanished in $\mathcal{U}$.
	\end{enumerate}
\end{theorem}

The next proposition from \cite[Lemma $8.6$]{Lee2012} extends a smooth vector field from a closed subset to the whole manifold with a compact support.

\begin{proposition}[The global extension for vector fields]\label{t:normal3}
	Let $M$ be a smooth manifold
	with or without boundary, and let $\mathcal{N}\subset M$ be a closed subset. Suppose $\mathbf{V}$ is a smooth vector field along $\mathcal{N}$. Given any open subset $\mathcal{U}$ containing $\mathcal{N}$, there exists a smooth
	global vector field $\tilde{\mathbf{V}}$ on $M$ such that $\tilde{\mathbf{V}}|_{\mathcal{N}}=\mathbf{V}$ and $\supp\tilde{\mathbf{V}}\subset \mathcal{U}$.
\end{proposition}

\subsection{Flows and fundamental theorems}

\begin{theorem}[The diffeomorphism invariance of boundary] \label{t:bdrythm}
	Suppose $M$ and
	$N$ are $C^k$ ($k\geq 1$) manifolds with boundary and $F:M\rightarrow N$ is a $C^k$diffeomorphism. Then
	$F(\del{}M)=\del{}N$, and $F$ restricts to a $C^k$ diffeomorphism from $\text{Int} M$ to $\text{Int} N$.
\end{theorem}

We denote $\mathbf{V}:=(V^1, \cdots, V^n)$ a continuous vector field and use a dot to denote an ordinary derivative with respect to $t$ in this Appendix.
Let us see the ordinary differential system,
\begin{align}
\dot{y}^i(t)= & V^i\bigl(t,y^1(t),\cdots, y^n(t)\bigr), \label{e:ode1} \\
y^i(t_0)=& c^i,  \label{e:ode2}
\end{align}
for $i=1,\cdots, n$, where $(t_0,\mathbf{c}):=(t_0, c^1,\cdots, c^n)$ is an arbitrary point. We first introduce an important Theorem \ref{t:flth2} without proofs, see \cite[Theorem D.$5$, D.$6$]{Lee2012} (or \cite[Theorem $1.11$]{Lang2002a}, \cite[Chapter XIV, \S $4$]{Lang1993}) for details.
\begin{theorem}\label{t:flth2}
	Let $J\subset \Rbb$ be an open interval and $U\subset \Rbb^n$ be an open subset, and let $\mathbf{V}:J\times U\rightarrow \Rbb^n$ be locally Lipschitz continuous and a $C^k$ vector-valued function for some $k\geq 0$.
	\begin{enumerate}
		\item \label{t:ex} (Existence) For any $s_0\in J$ and $\bx_0\in U$, there exists an open interval $J_0 \subset J$
		containing $s_0$ and an open subset $U_0\subset U$ containing $\bx_0$, such that for each $t_0\in J_0$ and $\mathbf{c}=(c^1, \cdots, c^n)\in U_0$, there is a $C^k$ map $y:J_0\rightarrow U$ that solves \eqref{e:ode1}--\eqref{e:ode2}.
		\item (Uniqueness) Any two differentiable solutions to \eqref{e:ode1}--\eqref{e:ode2} agree on their
		common domain.
		\item (Regularity) Let $J_0$ and $U_0$ be as in \eqref{t:ex}, and define a map $\theta:J_0\times J_0 \times U_0 \rightarrow U$ by letting $\theta(t,t_0,\mathbf{c})=y(t)$, where $y:J_0\rightarrow U$ is the unique solution
		to \eqref{e:ode1}--\eqref{e:ode2}. Then $\theta$ is of class $C^k$.
	\end{enumerate}
\end{theorem}

We have the fundamental theorem on time-dependent flows using the above Theorem \ref{t:flth2}. The proof can be found in \cite[Theorem $9.48$]{Lee2012} and \cite[IV, \S $1$]{Lang2002a}. 
\begin{theorem}[The fundamental theorem on time-dependent flows]\label{t:flowft}
	Let $M$ be
	a smooth manifold, let $\mathcal{I}\subset \Rbb$ be an open interval, and let $\mathbf{V}:\mathcal{I}\times M\rightarrow TM$ be
	a $C^k$ $(k\geq 1)$ time-dependent vector field on $M$. Then there exists an open subset $\mathcal{E}\subset \mathcal{I}\times\mathcal{I}\times M$ and a $C^k$ map $\varphi:\mathcal{E}\rightarrow M$ called the time-dependent flow of $\mathbf{V}$, with
	the following properties:
	\begin{enumerate}[(a)]
		\item For each $t_0\in \mathcal{I}$ and $p\in M$, the set $\mathcal{E}^{(t_0,p)}=\{t\in \mathcal{I}\;|\; (t,t_0,p)\in \mathcal{E}\}$ is an open
		interval containing $t_0$, and the curve $\varphi^{(t_0,p)}:\mathcal{E}^{(t_0,p)}\rightarrow M$ defined by
		$\varphi^{(t_0,p)}(t)=\varphi(t,t_0,p)$ is the unique maximal integral curve of $\mathbf{V}$ with initial
		condition $\varphi^{(t_0,p)}(t_0)=p$.
		\item If $t_1\in \mathcal{E}^{(t_0,p)}$ and $q=\varphi^{(t_0,p)}(t_1)$, then $\mathcal{E}^{(t_1,q)}=\mathcal{E}^{(t_0,p)}$ and $\varphi^{(t_1,q)}=\varphi^{(t_0,p)}$.
		\item For each $(t_1,t_0)\in \mathcal{I}\times \mathcal{I}$, the set $M_{t_1,t_0}=\{p\in M\;|\; (t_1,t_0,p)\in \mathcal{E}\}$ is open in
		$M$, and the map $\varphi^{(t_1,t_0)}:M_{t_1,t_0}\rightarrow M$ defined by $\varphi^{(t_1,t_0)}(p)=\varphi(t_1,t_0,p)$ is a
		diffeomorphism from $M_{t_1,t_0}$ onto $M_{t_0,t_1}$ with inverse $\varphi^{(t_0,t_1)}$.
		\item If $p\in M_{t_1,t_0}$ and $\varphi^{(t_1,t_0)}(p)\in M_{t_2,t_1}$, then $p\in M_{t_2,t_0}$ and
		$\varphi^{(t_2,t_1)}\circ \varphi^{(t_1,t_0)}(p)=\varphi^{(t_2,t_0)}(p)$.
	\end{enumerate}
\end{theorem}

\section{Tools of analysis}\label{s:ana}

\subsection{Sobolev spaces}\label{s:Sobsp}
Let us briefly review the famous Sobolev spaces over a domain $\Omega \subset \Rbb^n$. In this article, we mainly use the Sobolev spaces $W^{m,p}$ for $m$ is a positive integer (for detailed introductions, see, for example, \cite{Adams2003b,Tartar2007}). We recall it in the following Definition \ref{t:Sob0}. For fractional order spaces, one can find them in \cite[Chapt. $7$]{Adams2003b} and \cite[Chapt. $34$]{Tartar2007} for details. There are various ways to define $W^{s,p}(\Omega)$ for $s\in \Rbb$, for example, \cite[Chapt. $34$]{Tartar2007} gives two definitions: the first one is defined by restricting functions in Sobolev spaces $W^{s,p}(\Rbb^3)$ to $\Omega$ and the second by interpolations of $W^{m,p}(\Omega)$ with positive integer $m$, and if $s=m$ is a positive integer,  $W^{s,p}(\Omega)$ reduces to the classical Sobolev spaces $W^{m,p}(\Omega)$. If $\Omega$ is a \textit{strong local Lipschitz domain} (see \cite[\S$7.69$]{Adams2003b} and \cite[Chapt. $34-36$]{Tartar2007}), then all the definitions are equivalent and have an equivalently intrinsic characterization (see \cite[\S$7.47$]{Adams2003b} and \cite[Lemma $35.2$ and $36.1$]{Tartar2007}  for details). Due to these equivalences, we do not  distinguish these definitions in this article but focus on the following Definition \ref{t:Sob0} of the classical Sobolev spaces. We only use the fractional order spaces in Appendix \ref{s:calculus} for some derivations of basic inequalities.
\begin{definition}\label{t:Sob0}
	For any positive integer $m$  and  $1  \leq p \leq \infty$, we consider
a vector space:
\begin{align*}
	W^{m,p}(\Omega):=\{u\in L^p(\Omega)\;|\; D^\alpha u\in L^p(\Omega) \; \text{for} \; 0\leq |\alpha|\leq m \}
\end{align*}
where $D^\alpha u$ is the weak partial derivative. It is a normed space equipped with the norm $\|\cdot\|_{W^{m,p}(\Omega)}$ defined by
\begin{align*}
	\|u\|_{W^{k,p}(\Omega)}:=\begin{cases}
		 \Bigl(\sum_{|\alpha|\leq k}\int_\Omega|D^\alpha u|^p dx \Bigr)^{\frac{1}{p}} \quad &  	(1\leq p<\infty),  \\
	 \sum_{|\alpha|\leq k} \text{ess sup}_\Omega|D^\alpha u| &  (p=\infty) .
	\end{cases}
\end{align*}
\end{definition}

\subsection{Calder\'{o}n extension theorem}\label{s:extention}
This section states an important tool for the existence theorem. First  simple $(m,p)$-extension operators are introduced for expressing the extension theorem. These definitions and theorems can be found in \cite{Calderon1961,Adams2003b}.
\begin{definition}[The simple $(m,p)$-extension operator]\label{t:extention}
	Let $\Omega$ be a domain in $\Rbb^n$. For given numbers $m$ and $p$, a linear operator $E:W^{m,p}(\Omega)\rightarrow W^{m,p}(\Rbb^n)$ is called a \textit{simple $(m,p)$-extension operator} for $\Omega$ if there is a constant $K=K(m,p)$ such that for every $u\in W^{m,p}(\Omega)$, the following conditions hold,
	\begin{enumerate}
	    \item $Eu(x)=u(x)$ a.e. in $\Omega$;
	    \item $\|Eu\|_{W^{m,p}(\Rbb^n)}\leq K\|u\|_{W^{m,p}(\Omega	)}$.	
	\end{enumerate}
\end{definition}

\begin{theorem}[Calder\'{o}n extension theorem]\label{t:Cdnext}
    Let $\Omega$ be a domain in $\Rbb^n$ satisfying Calder\'{o}n's uniform cone condition. Then for any $m\in\{1,2,\cdots\}$ and any $p$ satisfying $1<p<\infty$, there exists a simple $(m,p)$-extension operator $E=E(m,p)$ for $\Omega$.
\end{theorem}
\begin{proof}
	See \cite[Theorem $5.28$]{Adams2003b} or \cite{Calderon1961} for the construction of this extension and the detailed proof.
\end{proof}

\subsection{The Stein extension theorem}\label{s:extentionS}
If the domain $\Omega\subset \Rbb^n$ is more regular, for example, a Lipschitz domain (in fact, see \cite{Elias1970a}, it requires the boundary with a weaker condition, minimal smoothness condition), there is a better and universal extension operator given by E. Stein \cite[Chapter $VI$ \S$3$]{Elias1970a}

\begin{theorem}[The Stein extension theorem]\label{t:extentionS}
	Let $\Omega\subset \Rbb^n$ be a Lipschitz domain. Then there exists a linear operator $E:W^{m,p}(\Omega)\rightarrow W^{m,p}(\Rbb^n)$ satisfying for every $u\in W^{m,p}(\Omega)$,
	\begin{enumerate}
		\item $Eu(x)=u(x)$ a.e. in $\Omega$;
		\item there is a constant $K>0$ such that $\|Eu\|_{W^{m,p}(\Rbb^n)}\leq K\|u\|_{W^{m,p}(\Omega	)}$ for all $p\in[1,\infty]$ and $m\in\Zbb_{\geq 0}$.
	\end{enumerate}
\end{theorem}
\begin{remark}
	For the case $E:W^{1,\infty}(\Omega)\rightarrow W^{1,\infty}(\Rbb^n)$, Evans \cite[\S $5.4$]{Evans2010} gives another extension theorem provided $\Omega$ is bounded and a $C^1$ domain.
\end{remark}

\subsection{Gagliardo–Nirenberg–Moser and Moser estimates on bounded domains}\label{s:calculus}

The famous Gagliardo–Nirenberg–Moser inequalities are widely used in energy estimates and can be found in many references. Since this article requires these inequalities on a bounded domain, we state them here (see \cite{Adams1977,Nirenberg2011,Brezis2018}) and derive some useful corollaries which are the bounded domain version of inequalities in \cite[\S$13.3$]{Taylor2010}.

In this section, we will always assume $\Omega$ is a \textit{standard domain} in $\Rbb^n$, which means  $\Omega$ is $\Rbb^n$, a half-space or a \textit{Lipschitz bounded domain}  in $\Rbb^n$ (in fact, to serve this paper, we emphasize estimates on the Lipschitz bounded domain). A key condition in the following theorem is given by
\begin{equation}\label{e:cond0}
	s_2\geq 1\text{ is an integer}, \quad p_2=1 \AND s_1-\frac{1}{p_1}\geq s_2-\frac{1}{p_2}.
\end{equation}

\begin{theorem}[The general Gagliardo–Nirenberg–Moser inequalities, see \cite{Brezis2018}]\label{t:GNMineq}
	Assume $\Omega$ is a standard domain in $\Rbb^n$, the real numbers $s,s_1,s_2 \geq 0$, $\theta\in(0,1)$ and $1\leq p_1,p_2,p\leq \infty$ satisfy the relations
	\begin{equation*}
		s=\theta s_1+(1-\theta)s_2,  \AND \frac{1}{p}=\frac{\theta}{p_1}+\frac{1-\theta}{p_2}.
	\end{equation*}
	\begin{enumerate}
		\item If Condition \eqref{e:cond0} fails, then for every $\theta\in(0,1)$, there exists a constant $C$ depending on $s_1,s_2,p_1,p_2, \theta$ and $\Omega$ such that
		\begin{equation*}\label{e:GNMineq}
			\|f\|_{W^{s,p}(\Omega)}\leq C \|f\|^\theta_{W^{s_1,p_1}(\Omega)}\|f\|^{1-\theta}_{W^{s_2,p_2}(\Omega)}
		\end{equation*}
		for all $f\in W^{s_1,p_1}(\Omega)\cap W^{s_2,p_2}(\Omega)$.
		\item If Condition \eqref{e:cond0} holds, there exists some $f\in W^{s_1,p_1}(\Omega)\cap W^{s_2,p_2}(\Omega)$ such that $f\notin W^{s,p}(\Omega)$ for any $\theta\in(0,1)$.
	\end{enumerate}
\end{theorem}
\begin{proof}
	See the detailed proof of \cite[Theorem 1]{Brezis2018}.
\end{proof}

By letting $s=\ell$, $p=2k/\ell$, $\theta=\ell/k$, $s_1=k$, $p_1=2$, $s_2=0$ and $p_2=\infty$, the following immediate result \eqref{e:GNMineq1} (we also refer to it as Gagliardo–Nirenberg–Moser inequality) derived from Theorem \ref{t:GNMineq} will play a crucial role in the later estimates,
\begin{align}\label{e:GNMineq1}
	\|f\|_{W^{\ell,\frac{2k}{\ell}}(\Omega)}\leq C \|f\|^{\frac{\ell}{k}}_{H^{k}(\Omega)}\|f\|^{1-\frac{\ell}{k}}_{L^\infty(\Omega)}
\end{align}
for all $0 < \ell<k$ and $f\in H^{k}(\Omega)\cap L^\infty(\Omega)$. The following estimates can be derived from the similar procedures to estimates in  \cite[\S$13.3$]{Taylor2010} with minor revisions.

\begin{proposition}\label{t:prodn}
	If $\Omega$ is a standard domain in $\Rbb^n$ and $|\beta|+|\gamma|=k$, then
	\begin{equation*}
		\|(D^\beta f)(D^\gamma g)\|_{L^2(\Omega)} \leq C\|f\|_{\Li(\Omega)}\|g\|_{H^k(\Omega)}+C\|f\|_{H^k(\Omega)}\|g\|_{\Li(\Omega)}
	\end{equation*}
	for all $f,g\in H^k(\Omega)\cap L^\infty(\Omega)$.
\end{proposition}
\begin{proof}
	Let $|\beta|=\ell$, $|\gamma|=m$ and $\ell+m=k$, then using H\"older's inequality,  Gagliardo–Nirenberg–Moser inequality \eqref{e:GNMineq1} and Young’s inequality for conjugate Hölder exponents, in turn, yield
	\begin{align*}
		& \|(D^\beta f)(D^\gamma g)\|_{L^2(\Omega)} \leq \|D^\beta f\|_{L^{\frac{2k}{\ell}}(\Omega)}\|D^\gamma g\|_{L^\frac{2k}{m}(\Omega)}\leq C \|f\|^{\frac{\ell}{k}}_{H^{k}(\Omega)}\|f\|^{1-\frac{\ell}{k}}_{L^\infty(\Omega)}
		\|g\|^{\frac{m}{k}}_{H^{k}(\Omega)}\|g\|^{1-\frac{m}{k}}_{L^\infty(\Omega)} \notag  \\
		& \hspace{4cm}=C (\|f\|_{H^{k}(\Omega)}\|g\|_{L^\infty(\Omega)}) ^{\frac{\ell}{k}}(\|f\|_{L^\infty(\Omega)}
		\|g\|_{H^{k}(\Omega)})^{\frac{m}{k}}  \notag  \\
		& \hspace{4cm} \leq  C\|f\|_{\Li(\Omega)}\|g\|_{H^k(\Omega)}+C\|f\|_{H^k(\Omega)}\|g\|_{\Li(\Omega)}.
	\end{align*}
	This completes the proof.
\end{proof}

Next, using the above Gagliardo–Nirenberg–Moser inequality \eqref{e:GNMineq1} and Proposition \ref{t:prodn} to replace inequality $(3.17)$ and Proposition $3.6$ in \cite[\S$13.3$]{Taylor2010} (also see \cite{Klainerman1981,Koch1990,Majda2012}, etc.) respectively, almost the same calculations as Proposition $3.7$, $3.9$ and Lemma $3.10$ in \cite[\S$13.3$]{Taylor2010}, we obtain the following three assertions. Omitting the detailed proofs, we only list the results below due to the similarities.
\begin{corollary}\label{t:prodn2}
	If $\Omega$ is a standard domain in $\Rbb^n$ and $|\beta_1|+\cdots+|\beta_\mu|=k$, then
	\begin{equation*}
	\|f_1^{(\beta_1)}\cdots f_\mu^{(\beta_\mu)}\|_{L^2(\Omega)} \leq C\sum_{\nu}\bigl[\|f_1\|_{\Li(\Omega)}\cdots\widehat{\|f_\nu\|}_{\Li(\Omega)}\cdots \|f_\mu\|_{\Li(\Omega)}\bigr]\|f_\nu\|_{H^k(\Omega)}
	\end{equation*}
	for all $f_\ell\in H^k(\Omega)\cap L^\infty(\Omega)$ ($\ell=1,\cdots,\mu$).
\end{corollary}

\begin{proposition}[Moser estimates I]\label{t:commuest}
	Suppose $\Omega$ is a standard domain in $\Rbb^n$, we have the estimates
	\begin{equation*}
		\|fg\|_{H^k(\Omega)} \leq C\|f\|_{\Li(\Omega)}\|g\|_{H^k(\Omega)}+C\|f\|_{H^k(\Omega)}\|g\|_{\Li(\Omega)}
	\end{equation*}
	and for $|\alpha|\leq k$,
	\begin{equation*}
		\|D^\alpha (fg)-fD^\alpha g\|_{L^2(\Omega)}\leq C\|\nabla f\|_{H^{k-1}(\Omega)}\|g\|_{\Li(\Omega)}+C\|\nabla f\|_{\Li(\Omega)}\|g\|_{H^{k-1}(\Omega)}.
	\end{equation*}
\end{proposition}

\begin{proposition}[Moser estimates II]\label{t:compoest}
	Suppose $\Omega$ is a standard domain in $\Rbb^n$, $F\in C^k(\Rbb)$ ($k\in \Zbb_{\geq1}$) and $F(0)=0$. Then, for $u\in H^k(\Omega)\cap \Li(\Omega)$ and  $u(x)\in \mathcal{G}$ where $\mathcal{G}$ is open and bounded set in $\Rbb$,
	\begin{equation*}
		\|F(u)\|_{H^k(\Omega)}\leq C_k(\|DF\|_{C^{k-1}(\bar{\mathcal{G}})}) (1+\|u\|^{k-1}_{\Li(\Omega)})\|u\|_{H^k(\Omega)}.
	\end{equation*}
\end{proposition}

\subsection{Ferrari--Shirota--Yanagisawa estimates}\label{s:FEY}
Ferrari--Shirota--Yanagisawa estimate firstly given by Ferrari \cite{Ferrari1993} and
Shirota, Yanagisawa \cite{Shirota1993} respectively in $1993$ is a bounded domain version of Beale--Kato--Majda estimate (in $\Rbb^3$, see \cite{Beale1984}). We present this inequality without proof, and readers may consult the details in \cite{Ferrari1993,Shirota1993}.
\begin{theorem}[Ferrari--Shirota--Yanagisawa estimates, \cite{Ferrari1993}]\label{t:FSY}
	Suppose $\Omega\subset \Rbb^3$ is a bounded, simply connected domain of class $C^{2,\beta}$, $\beta\in(0,1)$, and $v:=(\mathbf{u},\phi)\in H^3(\Omega)$ is a solution of the system
	\begin{align*}
		\nabla\times \mathbf{u}-\nabla \phi=&\omega \quad \text{in} \quad \Omega, \\
		\nabla\cdot \mathbf{u}=&\theta \quad \text{in} \quad \Omega, \\
		\mathbf{u}\cdot \mathbf{n}=& 0 \quad \text{on} \quad \del{}\Omega, \\
		\phi=& 0 \quad \text{on} \quad \del{}\Omega
	\end{align*}
	with $\psi:=(\omega,\theta)\in H^2(\Omega)$ and $\mathbf{n}$ is the unit outward normal at $\bx\in\del{}\Omega$, then
	\begin{equation*}
		\|v\|_{W^{1,\infty}(\Omega)} \leq C \Bigl(1+\ln^{+}\frac{\|\psi\|_{H^2(\Omega)}}{\|\psi\|_{\Li(\Omega)}}\Bigr)\|\psi\|_{\Li(\Omega)}
	\end{equation*}
	where
	\begin{align*}
		\ln^{+} a:=\begin{cases}
		\ln a, \quad \text{if} \quad a\geq 1\\
		0, \quad \text{otherwise}
		\end{cases}.
	\end{align*}
\end{theorem}

\subsection{The generalized Gronwall inequality of integral form}\label{s:grnwl}
The following form of Gronwall inequality comes from \cite[Theorem $1.3.1$]{Pachpatte1997}.
\begin{theorem}[The generalized Gronwall inequality of integral form]\label{e:grnwl}
	Let $u$ and $f$ be continuous and nonnegative functions
	defined  on  $\mathcal{J}  =  [\alpha,\beta]$,  and  let  $n(t)$  be a continuous, positive  and  non-decreasing function defined on $\mathcal{J}$; then
	\begin{equation*}
		u(t)\leq n(t)+\int^t_\alpha f(s)u(s) ds,
	\end{equation*}
	for $t\in \mathcal{J}$, implies that
	\begin{equation*}
	u(t)\leq n(t)\exp\Bigl(\int^t_\alpha f(s) \Bigr) ds,
	\end{equation*}
	for $t\in \mathcal{J}$.
\end{theorem}

\subsection{Reynold's transport theorem}\label{s:RTT}
Reynold’s transport theorem is a multidimensional version of the Leibniz integral rule, which states how to exchange the derivative and the integral if the integration region changes with time. We slightly improve Reynold's transport theorem using Theorem $9.41$ in Rudin \cite[Page $235$-$236$]{Rudin1976}.
\begin{theorem}[Reynold's transport theorem]\label{t:RTT}
	Suppose $\del{}\Omega(0) \in C^1$  and a field $C^1\ni f:[0,T)\times \overline{\Omega(t)} \rightarrow V$ and the flow $\chi:[0,T)\times \overline{\Omega(0)}\rightarrow \overline{\Omega(t)}\subset \Rbb^3$ generated by a vector field $v^i\in C^1([0,T)\times\overline{\Omega(t)},\Rbb^3)$, such that  $\chi(t,\xi)=\bx \in \overline{\Omega(t)}$ for every  $(t,\xi)\in[0,T)\times \overline{\Omega(0)}$ where $T>0$ is a constant, $\Omega(t):=\chi(t,\Omega(0)) \subset \Rbb^3$ is a domain depending on $t$ and $V \subset \Rbb^n$ for some $n \in \Zbb_{\geq 1}$. Then
	\begin{align*}
	\frac{d}{dt}\int_{\Omega(t)}f(t,\bx)d^3\bx = & \int_{\Omega(t)}\Bigl[\del{t} f(t,\bx)  + \del{i} \bigl(f(t,\bx)v^i(t,\bx)\bigr) \Bigr]d^3\bx \notag  \\
	=&  \int_{\Omega(t)}\del{t} f(t,\bx) d^3\bx  + \int_{\del{}\Omega(t)}   f(t,\bx)v^i(t,\bx) \nu_i d\sigma.
	\end{align*}
where $\nu_i:=\delta_{ij}\nu^j$ and $\nu^j$ is the unit outward-pointing normal to $\del{}\Omega(t)$
\end{theorem}
\begin{proof}
	Let us briefly state this proof.
	Since $\partial \Omega(0)\in C^1$ and $v^i\in C^1$, by Lemma \ref{t:surfc2}, we obtain $\partial \Omega(t)\in C^1$ and $\chi\in C^1$. Now let us calculate $\del{t} \int_{\Omega(t)} f(t,\bx) d^3\bx$. First let us fix the integration domain by using $\int_{\Omega(t)} f(t,\bx) d^3\bx=\int_{\Omega(0)} f(t,\chi(t,\xi)) |J(t,\xi)|d^3\xi$ where $J(t,\xi) =\det \bigl[\frac{\partial \chi_t^k}{\partial \xi^l}\bigr] \neq 0$ is the Jacobian (we assume $J>0$ without any loss of generality).
By denoting $\chi_t(\xi):=\chi(t,\xi)$, we obtain
	\begin{equation}\label{e:DtJ}
	\del{t}J(t,\xi)=  \det\Bigl[\del{t}\frac{\partial \chi_t^1}{\partial \xi^l},\frac{\partial \chi_t^2}{\partial \xi^l},\frac{\partial \chi_t^3}{\partial \xi^l}\Bigr]+\det\Bigl[\frac{\partial \chi_t^1}{\partial \xi^l},\del{t}\frac{\partial \chi_t^2}{\partial \xi^l},\frac{\partial \chi_t^3}{\partial \xi^l}\Bigr]+\det\Bigl[\frac{\partial \chi_t^1}{\partial \xi^l},\frac{\partial \chi_t^2}{\partial \xi^l},\del{t}\frac{\partial \chi_t^3}{\partial \xi^l}\Bigr]  .
\end{equation}
Let us consider $	\del{t}\chi^k(t,\xi)=v^k$.
Since $v^k\in C^1$ and $\chi\in C^1$, we have $\frac{\partial}{\partial \xi^l}[\del{t}\chi^k(t,\xi)]=\frac{\partial \chi_t^j}{\partial \xi^l}\frac{\partial}{\partial x^j}v^k\in C^0$, then, by \textit{Theorem $9.41$ in Rudin \cite[Page $235$-$236$]{Rudin1976}}, $\del{t}(\frac{\partial}{\partial \xi^l}\chi^k(t,\xi))$ exists and $\del{t}[\frac{\partial}{\partial \xi^l}\chi^k(t,\xi)]=\frac{\partial}{\partial \xi^l}[\del{t}\chi^k(t,\xi)]=\frac{\partial \chi_t^j}{\partial \xi^l}\frac{\partial v^k}{\partial x^j}$. Then substituting the above equality into the expression \eqref{e:DtJ} of $\del{t}J$, we arrive at $\del{t}J  =(\del{i} v^i )J$ (see \cite[\S$1$]{Chorin1993} and \cite[Page $197$-$198$]{Hieber2020} ).
Then differentiating $\del{t} \int_{\Omega(t)} f(t,\bx) d^3\bx$, using $\del{t}J  =(\del{i} v^i )J$,  similarly to  \cite[Page $578$]{Marsden2003}, \cite[\S $4$]{Flanders1973} and \cite[\S $13$]{Huilgol1997}, we obtain the formula in this theorem (noting the second equality is from the divergence theorem, see \cite[Theorem A$8.8$]{Alt2016}).  We complete the proof.
\end{proof}

\section*{Acknowledgement}
This work is partially supported by the China Postdoctoral Science Foundation Grant under the grant No. $2018M641054$ and the Fundamental Research Funds for the Central Universities, HUST: $5003011036$ and the National Natural Science Foundation of China (NSFC) under
the Grant No. $11971503$.  

\bigskip

\textbf{Data Availability} Data sharing is not applicable to this article as no datasets were
generated or analysed during the current study.

\bigskip

\textbf{Declarations}

\bigskip

\textbf{Conflict of interest} The authors declare that they have no conflict of interest.

\bibliographystyle{amsplain}
\bibliography{Reference_Chao}

\end{document}